\documentclass{amsart}
\usepackage{amssymb}
\usepackage{amsfonts}
\usepackage{geometry}

\newtheorem{theorem}{Theorem}
\theoremstyle{plain}

\newtheorem{condition}{Condition}
\newtheorem{conjecture}{Conjecture}

\newtheorem{definition}{Definition}
\newtheorem{example}{Example}

\newtheorem{lemma}{Lemma}
\newtheorem{notation}{Notation}
\newtheorem{problem}{Problem}
\newtheorem{proposition}{Proposition}
\newtheorem{remark}{Remark}

\input{tcilatex}
\begin{document}
\title[Two weight boundedness]{A two weight fractional singular integral
theorem with side conditions, energy and $k$-energy dispersed}
\author[E.T. Sawyer]{Eric T. Sawyer}
\address{ Department of Mathematics \& Statistics, McMaster University, 1280
Main Street West, Hamilton, Ontario, Canada L8S 4K1 }
\email{sawyer@mcmaster.ca}
\thanks{Research supported in part by NSERC}
\author[C.-Y. Shen]{Chun-Yen Shen}
\address{ Department of Mathematics \\
National Central University \\
Chungli, 32054, Taiwan }
\email{cyshen@math.ncu.edu.tw}
\thanks{C.-Y. Shen supported in part by the NSC, through grant NSC
104-2628-M-008 -003 -MY4}
\author[I. Uriarte-Tuero]{Ignacio Uriarte-Tuero}
\address{ Department of Mathematics \\
Michigan State University \\
East Lansing MI }
\email{ignacio@math.msu.edu}
\thanks{ I. Uriarte-Tuero has been partially supported by grants DMS-1056965
(US NSF), MTM2010-16232, MTM2015-65792-P (MINECO, Spain), and a Sloan
Foundation Fellowship. }
\date{April 16, 2016}

\begin{abstract}
This paper is a sequel to our paper \cite{SaShUr7}. Let $\sigma $ and $%
\omega $ be locally finite positive Borel measures on $\mathbb{R}^{n}$
(possibly having common point masses), and let $T^{\alpha }$\ be a standard $%
\alpha $-fractional Calder\'{o}n-Zygmund operator on $\mathbb{R}^{n}$ with $%
0\leq \alpha <n$. Suppose that $\Omega :\mathbb{R}^{n}\rightarrow \mathbb{R}%
^{n}$ is a globally biLipschitz map, and refer to the images $\Omega Q$ of
cubes $Q$ as \emph{quasicubes}. Furthermore, assume as side conditions the $%
\mathcal{A}_{2}^{\alpha }$ conditions, punctured $A_{2}^{\alpha }$
conditions, and certain $\alpha $\emph{-energy conditions} taken over
quasicubes. Then we show that $T^{\alpha }$ is bounded from $L^{2}\left(
\sigma \right) $ to $L^{2}\left( \omega \right) $ if the quasicube testing
conditions hold for $T^{\alpha }$\textbf{\ }and its dual, and if the
quasiweak boundedness property holds for $T^{\alpha }$.

Conversely, if $T^{\alpha }$ is bounded from $L^{2}\left( \sigma \right) $
to $L^{2}\left( \omega \right) $, then the quasitesting conditions hold, and
the quasiweak boundedness condition holds. If the vector of $\alpha $%
-fractional Riesz transforms $\mathbf{R}_{\sigma }^{\alpha }$ (or more
generally a strongly elliptic vector of transforms) is bounded from $%
L^{2}\left( \sigma \right) $ to $L^{2}\left( \omega \right) $, then both the 
$\mathcal{A}_{2}^{\alpha }$ conditions and the punctured $A_{2}^{\alpha }$
conditions hold.

We do not know if our quasienergy conditions are necessary when $n\geq 2$,
except for certain situations in which one of the measures is
one-dimensional \cite{LaSaShUrWi}, \cite{SaShUr8}, and for certain side
conditions placed on the measures such as doubling and $k$-energy dispersed,
which when $k=n-1$ is similar to the condition of uniformly full dimension
in \cite[versions 2 and 3]{LaWi}.
\end{abstract}

\maketitle
\tableofcontents

\begin{description}
\item[Dedication] This paper is dedicated to the memory of Cora Sadosky and
her contributions to the theory of weighted inequalities and her promotion
of women in mathematics.
\end{description}

\section{Introduction}

The boundedness of the Hilbert transform $Hf\left( x\right) =\int_{\mathbb{R}%
}\frac{f\left( y\right) }{y-x}dy$ on the real line $\mathbb{R}$ in the
Hilbert space $L^{2}\left( \mathbb{R}\right) $ has been known for at least a
century (perhaps dating back to A \& E\footnote{%
Peter Jones used A\&E to stand for Adam and Eve.}):%
\begin{equation}
\left\Vert Hf\right\Vert _{L^{2}\left( \mathbb{R}\right) }\lesssim
\left\Vert f\right\Vert _{L^{2}\left( \mathbb{R}\right) },\ \ \ \ \ f\in
L^{2}\left( \mathbb{R}\right) .  \label{Hilbert}
\end{equation}%
This inequality has been the subject of much generalization, to which we now
turn.

\subsection{A brief history of the $T1$ theorem}

The celebrated $T1$ theorem of David and Journ\'{e} \cite{DaJo} extends (\ref%
{Hilbert}) to more general kernels by characterizing those singular integral
operators $T$ on $\mathbb{R}^{n}$ that are bounded on $L^{2}\left( \mathbb{R}%
^{n}\right) $, and does so in terms of a weak boundedness property, and the
membership of the two functions $T\mathbf{1}$ and $T^{\ast }\mathbf{1}$ in
the space of bounded mean oscillation,%
\begin{eqnarray*}
\left\Vert T\mathbf{1}\right\Vert _{BMO\left( \mathbb{R}^{n}\right) }
&\lesssim &\left\Vert \mathbf{1}\right\Vert _{L^{\infty }\left( \mathbb{R}%
^{n}\right) }=1, \\
\left\Vert T^{\ast }\mathbf{1}\right\Vert _{BMO\left( \mathbb{R}^{n}\right)
} &\lesssim &\left\Vert \mathbf{1}\right\Vert _{L^{\infty }\left( \mathbb{R}%
^{n}\right) }=1.
\end{eqnarray*}%
These latter conditions are actually the following \emph{testing conditions}
in disguise,%
\begin{eqnarray*}
\left\Vert T\mathbf{1}_{Q}\right\Vert _{L^{2}\left( \mathbb{R}^{n}\right) }
&\lesssim &\left\Vert \mathbf{1}_{Q}\right\Vert _{L^{2}\left( \mathbb{R}%
^{n}\right) }=\sqrt{\left\vert Q\right\vert }\ , \\
\left\Vert T^{\ast }\mathbf{1}_{Q}\right\Vert _{L^{2}\left( \mathbb{R}%
^{n}\right) } &\lesssim &\left\Vert \mathbf{1}_{Q}\right\Vert _{L^{2}\left( 
\mathbb{R}^{n}\right) }=\sqrt{\left\vert Q\right\vert }\ ,
\end{eqnarray*}%
tested uniformly over all indicators of cubes $Q$ in $\mathbb{R}^{n}$ for
both $T$ and its dual operator $T^{\ast }$. This theorem was the culmination
of decades of investigation into the nature of cancellation conditions
required for boundedness of singular integrals\footnote{%
See e.g. chapter VII of Stein \cite{Ste} and the references given there for
a historical background.}.

A parallel thread of investigation had begun even earlier with the equally
celebrated theorem of Hunt, Muckenhoupt and Wheeden \cite{HuMuWh} that
extended (\ref{Hilbert}) to measures more general than Lebesgue's by
characterizing boundedness of the Hilbert transform on weighted spaces $%
L^{2}\left( \mathbb{R};w\right) $. This thread culminated in the theorem of
Coifman and Fefferman\footnote{%
See e.g. chapter V of \cite{Ste} and the references given there for the long
history of this investigation.} \cite{CoFe}\ that characterizes those
nonnegative weights $w$ on $\mathbb{R}^{n}$ for which all of the `nicest' of
the $L^{2}\left( \mathbb{R}^{n}\right) $ bounded singular integrals $T$
above are bounded on weighted spaces $L^{2}\left( \mathbb{R}^{n};w\right) $,
and does so in terms of the $A_{2}$ condition of Muckenhoupt,%
\begin{equation*}
\left( \frac{1}{\left\vert Q\right\vert }\int_{Q}w\left( x\right) dx\right)
\left( \frac{1}{\left\vert Q\right\vert }\int_{Q}\frac{1}{w\left( x\right) }%
dx\right) \lesssim 1\ ,
\end{equation*}%
taken uniformly over all cubes $Q$ in $\mathbb{R}^{n}$. This condition is
also a testing condition in disguise, in particular it follows from%
\begin{equation*}
\left\Vert T\left( \mathbf{s}_{Q}\frac{1}{w}\right) \right\Vert
_{L^{2}\left( \mathbb{R}^{n};w\right) }\lesssim \left\Vert \mathbf{s}_{Q}%
\frac{1}{w}\right\Vert _{L^{2}\left( \mathbb{R}^{n};w\right) }\ ,
\end{equation*}%
tested over all `indicators with tails' $\mathbf{s}_{Q}\left( x\right) =%
\frac{\ell \left( Q\right) }{\ell \left( Q\right) +\left\vert
x-c_{Q}\right\vert }$ of cubes $Q$ in $\mathbb{R}^{n}$.

A natural synthesis of these two threads leads to the `two weight' question
of characterizing those pairs of weights $\left( \sigma ,\omega \right) $
having the property that nice singular integrals are bounded from $%
L^{2}\left( \mathbb{R}^{n};\sigma \right) $ to $L^{2}\left( \mathbb{R}%
^{n};\omega \right) $. Returning to the simplest (nontrivial) singular
integral of all, namely the Hilbert transform $Hf\left( x\right) =\int_{%
\mathbb{R}}\frac{f\left( y\right) }{y-x}dy$ on the real line, Cotlar and
Sadosky gave a beautiful function theoretic characterization of the weight
pairs $\left( \sigma ,\omega \right) $ for which $H$ is bounded from $%
L^{2}\left( \mathbb{R};\sigma \right) $ to $L^{2}\left( \mathbb{R};\omega
\right) $, namely a two-weight extension of the Helson-Szego theorem. This
characterization illuminated a deep connection between two quite different
function theoretic conditions, but failed to shed much light on when either
of them held. On the other hand, the two weight inequality for the positive
fractional integrals were characterized using testing conditions by one of
us in \cite{Saw3}, but relying in a very strong way on the positivity of the
kernel, something the Hilbert kernel lacks. In light of these
considerations, Nazarov, Treil and Volberg formulated the two weight
question for the Hilbert transform \cite{Vol}, that in turn led to the
following NTV conjecture:

\begin{conjecture}
\cite{Vol} The Hilbert transform is bounded from $L^{2}\left( \mathbb{R}%
^{n};\sigma \right) $ to $L^{2}\left( \mathbb{R}^{n};\omega \right) $, i.e.%
\begin{equation}
\left\Vert H\left( f\sigma \right) \right\Vert _{L^{2}\left( \mathbb{R}%
^{n};\omega \right) }\lesssim \left\Vert f\right\Vert _{L^{2}\left( \mathbb{R%
}^{n};\sigma \right) },\ \ \ \ \ f\in L^{2}\left( \mathbb{R}^{n};\sigma
\right) ,  \label{Hilbert'}
\end{equation}%
if and only if the two weight $A_{2}$ condition with two tails holds,%
\begin{equation*}
\left( \frac{1}{\left\vert Q\right\vert }\int_{Q}\mathbf{s}_{Q}^{2}d\omega
\left( x\right) \right) \left( \frac{1}{\left\vert Q\right\vert }\int_{Q}%
\mathbf{s}_{Q}^{2}d\sigma \left( x\right) \right) \lesssim 1\ ,
\end{equation*}%
uniformly over all cubes $Q$, and the two testing conditions hold,%
\begin{eqnarray*}
\left\Vert H\mathbf{1}_{Q}\sigma \right\Vert _{L^{2}\left( \mathbb{R}%
^{n};\omega \right) } &\lesssim &\left\Vert \mathbf{1}_{Q}\right\Vert
_{L^{2}\left( \mathbb{R}^{n};\sigma \right) }=\sqrt{\left\vert Q\right\vert
_{\sigma }}\ , \\
\left\Vert H^{\ast }\mathbf{1}_{Q}\omega \right\Vert _{L^{2}\left( \mathbb{R}%
^{n};\sigma \right) } &\lesssim &\left\Vert \mathbf{1}_{Q}\right\Vert
_{L^{2}\left( \mathbb{R}^{n};\omega \right) }=\sqrt{\left\vert Q\right\vert
_{\omega }}\ ,
\end{eqnarray*}%
uniformly over all cubes $Q$.
\end{conjecture}

In a groundbreaking series of papers including \cite{NTV1},\cite{NTV2} and 
\cite{NTV3}, Nazarov, Treil and Volberg used weighted Haar decompositions
with random grids, introduced their `pivotal' condition, and proved the
above conjecture under the side assumption that the pivotal condition held.
Subsequently, in joint work of two of us, Sawyer and Uriarte-Tuero, with
Lacey \cite{LaSaUr2}, it was shown that the pivotal condition was not
necessary in general, a necessary `energy' condition was introduced as a
substitute, and a hybrid merging of these two conditions was shown to be
sufficient for use as a side condition. Eventually, these three authors with
Shen established the NTV conjecture in a two part paper; Lacey, Sawyer, Shen
and Uriarte-Tuero \cite{LaSaShUr3} and Lacey \cite{Lac}. A key ingredient in
the proof was an `energy reversal' phenomenon enabled by the Hilbert
transform kernel equality%
\begin{equation*}
\frac{1}{y-x}-\frac{1}{y-x^{\prime }}=\frac{x-x^{\prime }}{\left( y-x\right)
\left( y-x^{\prime }\right) },
\end{equation*}%
having the remarkable property that the denominator on the right hand side
remains \emph{positive} for all $y$ outside the smallest interval containing
both $x$ and $x^{\prime }$. This proof of the NTV conjecture was given in
the special case that the weights $\sigma $ and $\omega $ had no point
masses in common, largely to avoid what were then thought to be technical
issues. However, these issues turned out to be considerably more
interesting, and this final assumption of no common point masses was removed
shortly after by Hyt\"{o}nen \cite{Hyt2}, who also simplified some aspects
of the proof.

At this juncture, attention naturally turned to the analogous two weight
inequalities for higher dimensional singular integrals, as well as $\alpha $%
-fractional singular integrals such as the Cauchy transform in the plane. In
a long paper begun in \cite{SaShUr5} on the \textit{arXiv} in 2013, and
subsequently appearing in \cite{SaShUr7}, the authors introduced the
appropriate notions of Poisson kernel to deal with the $A_{2}^{\alpha }$
condition on the one hand, and the $\alpha $-energy condition on the other
hand (unlike for the Hilbert transform, these two Poisson kernels differ in
general). The main result of that paper established the $T1$ theorem for
`elliptic' vectors of singular integrals under the side assumption that an
energy condition and its dual held, thus identifying the \emph{culprit} in
higher dimensions as the energy conditions. A general $T1$ conjecture is
this (see below for definitions).

\begin{conjecture}
Let $\mathbf{T}^{\alpha ,n}$ denote an elliptic vector of standard $\alpha $%
-fractional singular integrals in $\mathbb{R}^{n}$. Then $\mathbf{T}^{\alpha
,n}$ is bounded from $L^{2}\left( \mathbb{R}^{n};\sigma \right) $ to $%
L^{2}\left( \mathbb{R}^{n};\omega \right) $, i.e.%
\begin{equation}
\left\Vert \mathbf{T}^{\alpha ,n}\left( f\sigma \right) \right\Vert
_{L^{2}\left( \mathbb{R}^{n};\omega \right) }\lesssim \left\Vert
f\right\Vert _{L^{2}\left( \mathbb{R}^{n};\sigma \right) },\ \ \ \ \ f\in
L^{2}\left( \mathbb{R}^{n};\sigma \right) ,  \label{T bound}
\end{equation}%
if and only if the two one-tailed $\mathcal{A}_{2}^{\alpha }$ conditions
with holes hold, the punctured $A_{2}^{\alpha ,\limfunc{punct}}$ conditions
hold, and the two testing conditions hold,%
\begin{eqnarray*}
\left\Vert \mathbf{T}^{\alpha ,n}\mathbf{1}_{Q}\sigma \right\Vert
_{L^{2}\left( \mathbb{R}^{n};\omega \right) } &\lesssim &\left\Vert \mathbf{1%
}_{Q}\right\Vert _{L^{2}\left( \mathbb{R}^{n};\sigma \right) }=\sqrt{%
\left\vert Q\right\vert _{\sigma }}\ , \\
\left\Vert \mathbf{T}^{\alpha ,n,\func{dual}}\mathbf{1}_{Q}\omega
\right\Vert _{L^{2}\left( \mathbb{R}^{n};\sigma \right) } &\lesssim
&\left\Vert \mathbf{1}_{Q}\right\Vert _{L^{2}\left( \mathbb{R}^{n};\omega
\right) }=\sqrt{\left\vert Q\right\vert _{\omega }}\ ,
\end{eqnarray*}%
for all cubes $Q$ in $\mathbb{R}^{n}$ (whose sides need not be parallel to
the coordinate axes).
\end{conjecture}

In view of the aforementioned main result in \cite{SaShUr7}, the following
conjecture is stronger.

\begin{conjecture}
Let $\mathbf{T}^{\alpha ,n}$ denote an elliptic vector of standard $\alpha $%
-fractional singular integrals in $\mathbb{R}^{n}$. If $\mathbf{T}^{\alpha
,n}$ is bounded from $L^{2}\left( \mathbb{R}^{n};\sigma \right) $ to $%
L^{2}\left( \mathbb{R}^{n};\omega \right) $, then the energy conditions hold
as defined in Definition \ref{energy condition} below.
\end{conjecture}

While no counterexamples have yet been discovered to the energy conditions,
there are some cases in which they have been proved to hold. Of course, the
energy conditions hold for the Hilbert transform on the line \cite{LaSaUr2},
and in recent joint work with M. Lacey and B. Wick, the five of us have
established that the energy conditions hold for the Cauchy transform in the
plane in the special case where one of the measures is supported on either a
straight line or a circle, thus proving the $T1$ theorem in this case. The
key to this result was an extension of the energy reversal phenomenon for
the Hilbert transform to the setting of the Cauchy transform, and here the
one-dimensional nature of the line and circle played a critical role. In
particular, a special decomposition of a $2$-dimensional measure into `end'
and `side' pieces played a crucial role, and was in fact discovered
independently in both \cite{SaShUr3} and \cite{LaWi1}. A further instance of
energy reversal occurs in our $T1$ theorem \cite{SaShUr8} when one measure
is compactly supported on a $C^{1,\delta }$ curve in $\mathbb{R}^{n}$.

The paper \cite[v3]{LaWi} by Lacey and Wick overlaps both our paper \cite%
{SaShUr7} and this paper to some extent, and we refer the reader to \cite%
{SaShUr7} for a more detailed discussion.

Finally, we mention an entirely different approach to investigating the two
weight problem that has attracted even more attention than the $T1$ approach
we just described. Nazarov has shown that the two-tailed $\mathcal{A}%
_{2}^{\alpha }$ condition of Muckenhoupt (see below) is insufficient for (%
\ref{T bound}), and this begs the question of strengthening the Muckenhoupt
condition enough to make it sufficient for (\ref{T bound}). The great of
advantage of this approach is that strengthened Muckenhoupt conditions are
generally `easy' to check as compared to the highly unstable testing
conditions. The disadvantage of course is that such conditions have never
been shown to characterize (\ref{T bound}). The literature devoted to these
issues is both too vast and too tangential to this paper to record here, and
we encourage the reader to search the web for more on `bumped-up'
Muckenhoupt conditions.

This paper is concerned with the $T1$ approach and is a sequel to our first
paper \cite{SaShUr7}. We prove here a two weight inequality for standard $%
\alpha $-fractional Calder\'{o}n-Zygmund operators $T^{\alpha }$ in
Euclidean space $\mathbb{R}^{n}$, where we assume $n$-dimensional $\mathcal{A%
}_{2}^{\alpha }$ conditions (with holes), punctured $A_{2}^{\alpha .\limfunc{%
punct}}$ conditions, and certain $\alpha $\emph{-energy conditions} as side
conditions on the weights (in higher dimensions the Poisson kernels used in
these two conditions differ). The two main differences in this theorem here
are that we state and prove\footnote{%
Very detailed proofs of all of the results here can be found on the arXiv 
\cite{SaShUr6}.} our theorem in the more general setting of \emph{quasicubes}
(as in \cite{SaShUr5}), and more notably, we now permit the weights, or
measures, to have common point masses, something not permitted in \cite%
{SaShUr7} (and only obtained for a partial range of $\alpha $ in \cite[%
version 3]{LaWi}). As a consequence, we use $\mathcal{A}_{2}^{\alpha }$
conditions with holes as in the one-dimensional setting of Hyt\"{o}nen \cite%
{Hyt2}, together with punctured $A_{2}^{\alpha ,\limfunc{punct}}$
conditions, as the usual $A_{2}^{\alpha }$ `\emph{without punctures}' fails
whenever the measures have a common point mass. The extension to permitting
common point masses uses the two weight Poisson inequality in \cite{Saw3} to
derive functional energy, together with a delicate adaptation of arguments
in \cite{SaShUr5}. The key point here is the use of the (typically
necessary) `punctured' Muckenhoupt $A_{2}^{\alpha ,\limfunc{punct}}$
conditions below. They turn out to be crucial in estimating the backward
Poisson testing condition later in the paper. We remark that Hyt\"{o}nen's
bilinear dyadic Poisson operator and shifted dyadic grids \cite{Hyt2} in
dimension $n=1$ can be extended to derive functional energy in higher
dimensions, but at a significant cost of increased complexity. See the
previous version of this paper on the \textit{arXiv} for this approach%
\footnote{%
Additional small arguments are needed to complete the shifted dyadic proof
given there, but we omit them in favour of the simpler approach here resting
on punctured Muckenhoupt conditions instead of holes. The authors can be
contacted regarding completion of the shifted dyadic proof.}, and also \cite%
{LaWi} where Lacey and Wick use this approach. Finally, we point out that
our use of punctured Muckenhoupt conditions provides a simpler alternative
to Hyt\"{o}nen's method of extending to common point masses the NTV
conjecture for the Hilbert transform \cite{Hyt2}. The Muckenhoupt $\mathcal{A%
}_{2}^{\alpha }$ conditions (with holes) are also typically necessary for
the norm inequality, but the proofs require extensive modification when
quasicubes and common point masses are included.

On the other hand, the extension to quasicubes in the setting of \emph{no}
common point masses turns out to be, after checking all the details, mostly
a cosmetic modification of the proof in \cite{SaShUr7}, as demonstrated in 
\cite{SaShUr5}. The use of quasicubes is however crucial in our $T1$ theorem
when one of the measures is compactly supported on a $C^{1,\delta }$ curve 
\cite{SaShUr8}, and this accounts for their inclusion here.

We also introduce a new side condition on a measure, that we call $k$\emph{%
-energy dispersed}, which captures the notion that a measure is \emph{not}
supported too near a $k$-dimensional plane at any scale. When $0\leq \alpha
<n$ is appropriately related to $k$, we are able to obtain the necessity of
the energy conditions for $k$-energy dispersed measures, and hence a $T1$
theorem for strongly elliptic operators $\mathbf{T}^{\alpha }$. The case $%
k=n-1$ is similar to the condition of uniformly full dimension introduced in 
\cite[versions 2 and 3]{LaWi}.

We begin by recalling the notion of quasicube used in \cite{SaShUr5} - a
special case of the classical notion used in quasiconformal theory.

\begin{definition}
We say that a homeomorphism $\Omega :\mathbb{R}^{n}\rightarrow \mathbb{R}%
^{n} $ is a globally biLipschitz map if%
\begin{equation}
\left\Vert \Omega \right\Vert _{Lip}\equiv \sup_{x,y\in \mathbb{R}^{n}}\frac{%
\left\Vert \Omega \left( x\right) -\Omega \left( y\right) \right\Vert }{%
\left\Vert x-y\right\Vert }<\infty ,  \label{rigid}
\end{equation}%
and $\left\Vert \Omega ^{-1}\right\Vert _{Lip}<\infty $.
\end{definition}

Note that a globally biLipschitz map $\Omega $ is differentiable almost
everywhere, and that there are constants $c,C>0$ such that%
\begin{equation*}
c\leq J_{\Omega }\left( x\right) \equiv \left\vert \det D\Omega \left(
x\right) \right\vert \leq C,\ \ \ \ \ x\in \mathbb{R}^{n}.
\end{equation*}

\begin{example}
\label{wild}Quasicubes can be wildly shaped, as illustrated by the standard
example of a logarithmic spiral in the plane $f_{\varepsilon }\left(
z\right) =z\left\vert z\right\vert ^{2\varepsilon i}=ze^{i\varepsilon \ln
\left( z\overline{z}\right) }$. Indeed, $f_{\varepsilon }:\mathbb{%
C\rightarrow C}$ is a globally biLipschitz map with Lipschitz constant $%
1+C\varepsilon $ since $f_{\varepsilon }^{-1}\left( w\right) =w\left\vert
w\right\vert ^{-2\varepsilon i}$ and%
\begin{equation*}
\nabla f_{\varepsilon }=\left( \frac{\partial f_{\varepsilon }}{\partial z},%
\frac{\partial f_{\varepsilon }}{\partial \overline{z}}\right) =\left(
\left\vert z\right\vert ^{2\varepsilon i}+i\varepsilon \left\vert
z\right\vert ^{2\varepsilon i},i\varepsilon \frac{z}{\overline{z}}\left\vert
z\right\vert ^{2\varepsilon i}\right) .
\end{equation*}%
On the other hand, $f_{\varepsilon }$ behaves wildly at the origin since the
image of the closed unit interval on the real line under $f_{\varepsilon }$
is an infinite logarithmic spiral.
\end{example}

\begin{notation}
We define $\mathcal{P}^{n}$ to be the collection of half open, half closed
cubes in $\mathbb{R}^{n}$ with sides parallel to the coordinate axes. A half
open, half closed cube $Q$ in $\mathbb{R}^{n}$ has the form $Q=Q\left(
c,\ell \right) \equiv \dprod\limits_{k=1}^{n}\left[ c_{k}-\frac{\ell }{2}%
,c_{k}+\frac{\ell }{2}\right) $ for some $\ell >0$ and $c=\left(
c_{1},...,c_{n}\right) \in \mathbb{R}^{n}$. The cube $Q\left( c,\ell \right) 
$ is described as having center $c$ and sidelength $\ell $.
\end{notation}

We repeat the natural \emph{quasi} definitions from \cite{SaShUr5}.

\begin{definition}
Suppose that $\Omega :\mathbb{R}^{n}\rightarrow \mathbb{R}^{n}$ is a
globally biLipschitz map.

\begin{enumerate}
\item If $E$ is a measurable subset of $\mathbb{R}^{n}$, we define $\Omega
E\equiv \left\{ \Omega \left( x\right) :x\in E\right\} $ to be the image of $%
E$ under the homeomorphism $\Omega $.

\begin{enumerate}
\item In the special case that $E=Q$ is a cube in $\mathbb{R}^{n}$, we will
refer to $\Omega Q$ as a quasicube (or $\Omega $-quasicube if $\Omega $ is
not clear from the context).

\item We define the center $c_{\Omega Q}=c\left( \Omega Q\right) $ of the
quasicube $\Omega Q$ to be the point $\Omega c_{Q}$ where $c_{Q}=c\left(
Q\right) $ is the center of $Q$.

\item We define the side length $\ell \left( \Omega Q\right) $ of the
quasicube $\Omega Q$ to be the sidelength $\ell \left( Q\right) $ of the
cube $Q$.

\item For $r>0$ we define the `dilation' $r\Omega Q$ of a quasicube $\Omega
Q $ to be $\Omega rQ$ where $rQ$ is the usual `dilation' of a cube in $%
\mathbb{R}^{n}$ that is concentric with $Q$ and having side length $r\ell
\left( Q\right) $.
\end{enumerate}

\item If $\mathcal{K}$ is a collection of cubes in $\mathbb{R}^{n}$, we
define $\Omega \mathcal{K}\equiv \left\{ \Omega Q:Q\in \mathcal{K}\right\} $
to be the collection of quasicubes $\Omega Q$ as $Q$ ranges over $\mathcal{K}
$.

\item If $\mathcal{F}$ is a grid of cubes in $\mathbb{R}^{n}$, we define the
inherited quasigrid structure on $\Omega \mathcal{F}$ by declaring that $%
\Omega Q$ is a child of $\Omega Q^{\prime }$ in $\Omega \mathcal{F}$ if $Q$
is a child of $Q^{\prime }$ in the grid $\mathcal{F}$.
\end{enumerate}
\end{definition}

Note that if $\Omega Q$ is a quasicube, then $\left\vert \Omega Q\right\vert
^{\frac{1}{n}}\approx \left\vert Q\right\vert ^{\frac{1}{n}}=\ell \left(
Q\right) =\ell \left( \Omega Q\right) $. For a quasicube $J=\Omega Q$, we
will generally use the expression $\left\vert J\right\vert ^{\frac{1}{n}}$
in the various estimates arising in the proofs below, but will often use $%
\ell \left( J\right) $ when defining collections of quasicubes. Moreover,
there are constants $R_{big}$ and $R_{small}$ such that we have the
comparability containments%
\begin{equation*}
Q+\Omega x_{Q}\subset R_{big}\Omega Q\text{ and }R_{small}\Omega Q\subset
Q+\Omega x_{Q}\ .
\end{equation*}

Given a fixed globally biLipschitz map $\Omega $ on $\mathbb{R}^{n}$, we
will define below the $n$-dimensional $\mathcal{A}_{2}^{\alpha }$ conditions
(with holes), punctured Muckenhoupt conditions $A_{2}^{\alpha ,\limfunc{punct%
}}$, testing conditions, and energy conditions using $\Omega $-quasicubes in
place of cubes, and we will refer to these new conditions as quasi$\mathcal{A%
}_{2}^{\alpha }$, quasitesting and quasienergy conditions. We will then
prove a $T1$ theorem with quasitesting and with quasi$\mathcal{A}%
_{2}^{\alpha }$ and quasienergy side conditions on the weights. Since $%
\limfunc{quasi}\mathcal{A}_{2}^{\alpha }\cap \limfunc{quasi}A_{2}^{\alpha ,%
\limfunc{punct}}=\mathcal{A}_{2}^{\alpha }\cap A_{2}^{\alpha ,\limfunc{punct}%
}$ (see \cite{SaShUr8}), we usually drop the prefix $\limfunc{quasi}$ from
the various Muckenhoupt conditions (warning: $\limfunc{quasi}\mathcal{A}%
_{2}^{\alpha }\neq \mathcal{A}_{2}^{\alpha }$).

Since the $\mathcal{A}_{2}^{\alpha }$ and punctured Muckenhoupt conditions
typically hold, this identifies the culprit in higher dimensions as the pair
of quasienergy conditions. We point out that these quasienergy conditions
are implied by higher dimensional analogues of essentially all the other
side conditions used previously in two weight theory, in particular doubling
conditions and the Energy Hypothesis (1.16) in \cite{LaSaUr2}, as well as
the condition of $k$-energy dispersed measures that is introduced below.
This leads to our second theorem, which establishes the $T1$ theorem for
strongly elliptic operators $\mathbf{T}^{\alpha }$ when both measures are $k$%
-energy dispered with $k$ and $\alpha $ appropriately related.

It turns out that in higher dimensions, there are two natural `Poisson
integrals' $\mathrm{P}^{\alpha }$ and $\mathcal{P}^{\alpha }$\ that arise,
the usual Poisson integral $\mathrm{P}^{\alpha }$ that emerges in connection
with energy considerations, and a different Poisson integral $\mathcal{P}%
^{\alpha }$ that emerges in connection with size considerations. The
standard Poisson integral $\mathrm{P}^{\alpha }$ appears in the energy
conditions, and the reproducing Poisson integral $\mathcal{P}^{\alpha }$
appears in the $\mathcal{A}_{2}^{\alpha }$ condition. These two kernels
coincide in dimension $n=1$ for the case $\alpha =0$ corresponding to the
Hilbert transform.

\section{Statements of results}

Now we turn to a precise description of our main two weight theorem.

\begin{description}
\item[Assumption] We fix once and for all a globally biLipschitz map $\Omega
:\mathbb{R}^{n}\rightarrow \mathbb{R}^{n}$ for use in all of our
quasi-notions.
\end{description}

We will prove a two weight inequality for standard $\alpha $-fractional
Calder\'{o}n-Zygmund operators $T^{\alpha }$ in Euclidean space $\mathbb{R}%
^{n}$, where we assume the $n$-dimensional $\mathcal{A}_{2}^{\alpha }$
conditions, new punctured $A_{2}^{\alpha }$ conditions, and certain $\alpha $%
\emph{-quasienergy conditions} as side conditions on the weights. In
particular, we show that for positive locally finite Borel measures $\sigma $
and $\omega $ in $\mathbb{R}^{n}$, and assuming that both the \emph{%
quasienergy condition }and its dual hold, a strongly elliptic vector of
standard $\alpha $-fractional Calder\'{o}n-Zygmund operators $\mathbf{T}%
^{\alpha }$ is bounded from $L^{2}\left( \sigma \right) $ to $L^{2}\left(
\omega \right) $ \emph{if and only if} the $\mathcal{A}_{2}^{\alpha }$
condition and its dual hold (we assume a mild additional condition on the
quasicubes for this), the punctured Muckenhoupt condition $A_{2}^{\alpha ,%
\limfunc{punct}}$ and its dual hold, the quasicube testing condition for $%
\mathbf{T}^{\alpha }$ and its dual hold, and the quasiweak boundedness
property holds. In order to state our theorem precisely, we define these
terms in the following subsections.

\begin{remark}
It is possible to collect our various Muckenhoupt and quasienergy
assumptions on the weight pair $\left( \sigma ,\omega \right) $ into just 
\emph{two} compact side conditions of Muckenhoupt and quasienergy type. We
prefer however, to keep the individual conditions separate so that the
interested reader can track their use below.
\end{remark}

\subsection{Standard fractional singular integrals and the norm inequality}

Let $0\leq \alpha <n$. We define a standard $\alpha $-fractional CZ kernel $%
K^{\alpha }(x,y)$ to be a function defined on $\mathbb{R}^{n}\times \mathbb{R%
}^{n}$ satisfying the following fractional size and smoothness conditions of
order $1+\delta $ for some $\delta >0$,%
\begin{eqnarray}
\left\vert K^{\alpha }\left( x,y\right) \right\vert &\leq &C_{CZ}\left\vert
x-y\right\vert ^{\alpha -n}\text{ and }\left\vert \nabla K^{\alpha }\left(
x,y\right) \right\vert \leq C_{CZ}\left\vert x-y\right\vert ^{\alpha -n-1},
\label{sizeandsmoothness'} \\
\left\vert \nabla K^{\alpha }\left( x,y\right) -\nabla K^{\alpha }\left(
x^{\prime },y\right) \right\vert &\leq &C_{CZ}\left( \frac{\left\vert
x-x^{\prime }\right\vert }{\left\vert x-y\right\vert }\right) ^{\delta
}\left\vert x-y\right\vert ^{\alpha -n-1},\ \ \ \ \ \frac{\left\vert
x-x^{\prime }\right\vert }{\left\vert x-y\right\vert }\leq \frac{1}{2}, 
\notag
\end{eqnarray}%
and the last inequality also holds for the adjoint kernel in which $x$ and $%
y $ are interchanged. We note that a more general definition of kernel has
only order of smoothness $\delta >0$, rather than $1+\delta $, but the use
of the Monotonicity and Energy Lemmas below, which involve first order
Taylor approximations to the kernel functions $K^{\alpha }\left( \cdot
,y\right) $, requires order of smoothness more than $1$.

\subsubsection{Defining the norm inequality}

We now turn to a precise definition of the weighted norm inequality%
\begin{equation}
\left\Vert T_{\sigma }^{\alpha }f\right\Vert _{L^{2}\left( \omega \right)
}\leq \mathfrak{N}_{T_{\sigma }^{\alpha }}\left\Vert f\right\Vert
_{L^{2}\left( \sigma \right) },\ \ \ \ \ f\in L^{2}\left( \sigma \right) .
\label{two weight'}
\end{equation}%
For this we introduce a family $\left\{ \eta _{\delta ,R}^{\alpha }\right\}
_{0<\delta <R<\infty }$ of nonnegative functions on $\left[ 0,\infty \right) 
$ so that the truncated kernels $K_{\delta ,R}^{\alpha }\left( x,y\right)
=\eta _{\delta ,R}^{\alpha }\left( \left\vert x-y\right\vert \right)
K^{\alpha }\left( x,y\right) $ are bounded with compact support for fixed $x$
or $y$. Then the truncated operators 
\begin{equation*}
T_{\sigma ,\delta ,R}^{\alpha }f\left( x\right) \equiv \int_{\mathbb{R}%
^{n}}K_{\delta ,R}^{\alpha }\left( x,y\right) f\left( y\right) d\sigma
\left( y\right) ,\ \ \ \ \ x\in \mathbb{R}^{n},
\end{equation*}%
are pointwise well-defined, and we will refer to the pair $\left( K^{\alpha
},\left\{ \eta _{\delta ,R}^{\alpha }\right\} _{0<\delta <R<\infty }\right) $
as an $\alpha $-fractional singular integral operator, which we typically
denote by $T^{\alpha }$, suppressing the dependence on the truncations.

\begin{definition}
We say that an $\alpha $-fractional singular integral operator $T^{\alpha
}=\left( K^{\alpha },\left\{ \eta _{\delta ,R}^{\alpha }\right\} _{0<\delta
<R<\infty }\right) $ satisfies the norm inequality (\ref{two weight'})
provided%
\begin{equation*}
\left\Vert T_{\sigma ,\delta ,R}^{\alpha }f\right\Vert _{L^{2}\left( \omega
\right) }\leq \mathfrak{N}_{T_{\sigma }^{\alpha }}\left\Vert f\right\Vert
_{L^{2}\left( \sigma \right) },\ \ \ \ \ f\in L^{2}\left( \sigma \right)
,0<\delta <R<\infty .
\end{equation*}
\end{definition}

It turns out that, in the presence of Muckenhoupt conditions, the norm
inequality (\ref{two weight'}) is essentially independent of the choice of
truncations used, and we now explain this in some detail. A \emph{smooth
truncation} of $T^{\alpha }$ has kernel $\eta _{\delta ,R}\left( \left\vert
x-y\right\vert \right) K^{\alpha }\left( x,y\right) $ for a smooth function $%
\eta _{\delta ,R}$ compactly supported in $\left( \delta ,R\right) $, $%
0<\delta <R<\infty $, and satisfying standard CZ estimates. A typical
example of an $\alpha $-fractional transform is the $\alpha $-fractional 
\emph{Riesz} vector of operators%
\begin{equation*}
\mathbf{R}^{\alpha ,n}=\left\{ R_{\ell }^{\alpha ,n}:1\leq \ell \leq
n\right\} ,\ \ \ \ \ 0\leq \alpha <n.
\end{equation*}%
The Riesz transforms $R_{\ell }^{n,\alpha }$ are convolution fractional
singular integrals $R_{\ell }^{n,\alpha }f\equiv K_{\ell }^{n,\alpha }\ast f$
with odd kernel defined by%
\begin{equation*}
K_{\ell }^{\alpha ,n}\left( w\right) \equiv \frac{w^{\ell }}{\left\vert
w\right\vert ^{n+1-\alpha }}\equiv \frac{\Omega _{\ell }\left( w\right) }{%
\left\vert w\right\vert ^{n-\alpha }},\ \ \ \ \ w=\left(
w^{1},...,w^{n}\right) .
\end{equation*}

However, in dealing with energy considerations, and in particular in the
Monotonicity Lemma below where first order Taylor approximations are made on
the truncated kernels, it is necessary to use the \emph{tangent line
truncation} of\emph{\ }the Riesz transform $R_{\ell }^{\alpha ,n}$ whose
kernel is defined to be $\Omega _{\ell }\left( w\right) \psi _{\delta
,R}^{\alpha }\left( \left\vert w\right\vert \right) $ where $\psi _{\delta
,R}^{\alpha }$ is continuously differentiable on an interval $\left(
0,S\right) $ with $0<\delta <R<S$, and where $\psi _{\delta ,R}^{\alpha
}\left( r\right) =r^{\alpha -n}$ if $\delta \leq r\leq R$, and has constant
derivative on both $\left( 0,\delta \right) $ and $\left( R,S\right) $ where 
$\psi _{\delta ,R}^{\alpha }\left( S\right) =0$. Here $S$ is uniquely
determined by $R$ and $\alpha $. Finally we set $\psi _{\delta ,R}^{\alpha
}\left( S\right) =0$ as well, so that the kernel vanishes on the diagonal
and common point masses do not `see' each other. Note also that the tangent
line extension of a $C^{1,\delta }$ function on the line is again $%
C^{1,\delta }$ with no increase in the $C^{1,\delta }$ norm.

It was shown in the one dimensional case with no common point masses in \cite%
{LaSaShUr3}, that boundedness of the Hilbert transform $H$ with one set of
appropriate truncations together with the $A_{2}^{\alpha }$ condition
without holes, is equivalent to boundedness of $H$ with any other set of
appropriate truncations. We need to extend this to $\mathbf{R}^{\alpha ,n}$
and more general operators in higher dimensions and to permit common point
masses, so that we are free to use the tangent line truncations throughout
the proof of our theorem. For this purpose, we note that the difference
between the tangent line truncated kernel $\Omega _{\ell }\left( w\right)
\psi _{\delta ,R}^{\alpha }\left( \left\vert w\right\vert \right) $ and the
corresponding cutoff kernel $\Omega _{\ell }\left( w\right) \mathbf{1}_{%
\left[ \delta ,R\right] }\left\vert w\right\vert ^{\alpha -n}$ satisfies
(since both kernels vanish at the origin)%
\begin{eqnarray*}
&&\left\vert \Omega _{\ell }\left( w\right) \psi _{\delta ,R}^{\alpha
}\left( \left\vert w\right\vert \right) -\Omega _{\ell }\left( w\right) 
\mathbf{1}_{\left[ \delta ,R\right] }\left\vert w\right\vert ^{\alpha
-n}\right\vert \\
&\lesssim &\sum_{k=0}^{\infty }2^{-k\left( n-\alpha \right) }\left\{ \left(
2^{-k}\delta \right) ^{\alpha -n}\mathbf{1}_{\left[ 2^{-k-1}\delta
,2^{-k}\delta \right] }\left( \left\vert w\right\vert \right) \right\}
+\sum_{k=1}^{\infty }2^{-k\left( n-\alpha \right) }\left\{ \left(
2^{k}R\right) ^{\alpha -n}\mathbf{1}_{\left[ 2^{k-1}R,2^{k}R\right] }\left(
\left\vert w\right\vert \right) \right\} \\
&\equiv &\sum_{k=0}^{\infty }2^{-k\left( n-\alpha \right) }K_{2^{-k}\delta
}\left( w\right) +\sum_{k=1}^{\infty }2^{-k\left( n-\alpha \right)
}K_{2^{k}R}\left( w\right) ,
\end{eqnarray*}%
where the kernels $K_{\rho }\left( w\right) \equiv \frac{1}{\rho ^{n-\alpha }%
}\mathbf{1}_{\left[ \rho ,2\rho \right] }\left( \left\vert w\right\vert
\right) $ are easily seen to satisfy, uniformly in $\rho $, the norm
inequality (\ref{two weight}) with constant controlled by the offset $%
A_{2}^{\alpha }$ condition (\ref{offset A2}) below. The equivalence of the
norm inequality for these two families of truncations now follows from the
summability of the series $\sum_{k=0}^{\infty }2^{-k\left( n-\alpha \right)
} $ for $0\leq \alpha <n$. The case of more general families of truncations
and operators is similar.

\subsection{Quasicube testing conditions}

The following `dual' quasicube testing conditions are necessary for the
boundedness of $T^{\alpha }$ from $L^{2}\left( \sigma \right) $ to $%
L^{2}\left( \omega \right) $,%
\begin{eqnarray*}
\mathfrak{T}_{T^{\alpha }}^{2} &\equiv &\sup_{Q\in \Omega \mathcal{P}^{n}}%
\frac{1}{\left\vert Q\right\vert _{\sigma }}\int_{Q}\left\vert T^{\alpha
}\left( \mathbf{1}_{Q}\sigma \right) \right\vert ^{2}\omega <\infty , \\
\left( \mathfrak{T}_{T^{\alpha }}^{\ast }\right) ^{2} &\equiv &\sup_{Q\in
\Omega \mathcal{P}^{n}}\frac{1}{\left\vert Q\right\vert _{\omega }}%
\int_{Q}\left\vert \left( T^{\alpha }\right) ^{\ast }\left( \mathbf{1}%
_{Q}\omega \right) \right\vert ^{2}\sigma <\infty ,
\end{eqnarray*}%
and where we interpret the right sides as holding uniformly over all tangent
line truncations of $T^{\alpha }$.

\begin{remark}
We alert the reader that the symbols $Q,I,J,K$ will all be used to denote
either cubes or quasicubes, and the context will make clear which is the
case. Throughout most of the proof of the main theorem only quasicubes are
considered.
\end{remark}

\subsection{Quasiweak boundedness property}

The quasiweak boundedness property for $T^{\alpha }$ with constant $C$ is
given by 
\begin{eqnarray*}
&&\left\vert \int_{Q}T^{\alpha }\left( 1_{Q^{\prime }}\sigma \right) d\omega
\right\vert \leq \mathcal{WBP}_{T^{\alpha }}\sqrt{\left\vert Q\right\vert
_{\omega }\left\vert Q^{\prime }\right\vert _{\sigma }}, \\
&&\ \ \ \ \ \text{for all quasicubes }Q,Q^{\prime }\text{ with }\frac{1}{C}%
\leq \frac{\left\vert Q\right\vert ^{\frac{1}{n}}}{\left\vert Q^{\prime
}\right\vert ^{\frac{1}{n}}}\leq C, \\
&&\ \ \ \ \ \text{and either }Q\subset 3Q^{\prime }\setminus Q^{\prime }%
\text{ or }Q^{\prime }\subset 3Q\setminus Q,
\end{eqnarray*}%
and where we interpret the left side above as holding uniformly over all
tangent line trucations of $T^{\alpha }$. Note that the quasiweak
boundedness property is implied by either the \emph{tripled} quasicube
testing condition,%
\begin{equation*}
\left\Vert \mathbf{1}_{3Q}\mathbf{T}^{\alpha }\left( \mathbf{1}_{Q}\sigma
\right) \right\Vert _{L^{2}\left( \omega \right) }\leq \mathfrak{T}_{\mathbf{%
T}^{\alpha }}^{\limfunc{triple}}\left\Vert \mathbf{1}_{Q}\right\Vert
_{L^{2}\left( \sigma \right) },\ \ \ \ \ \text{for all quasicubes }Q\text{
in }\mathbb{R}^{n},
\end{equation*}%
or its dual defined with $\sigma $ and $\omega $ interchanged and the dual
operator $\mathbf{T}^{\alpha ,\ast }$ in place of $\mathbf{T}^{\alpha }$. In
turn, the tripled quasicube testing condition can be obtained from the
quasicube testing condition for the truncated weight pairs $\left( \omega ,%
\mathbf{1}_{Q}\sigma \right) $.

\subsection{Poisson integrals and $\mathcal{A}_{2}^{\protect\alpha }$}

Let $\mu $ be a locally finite positive Borel measure on $\mathbb{R}^{n}$,
and suppose $Q$ is an $\Omega $-quasicube in $\mathbb{R}^{n}$. Recall that $%
\left\vert Q\right\vert ^{\frac{1}{n}}\approx \ell \left( Q\right) $ for a
quasicube $Q$. The two $\alpha $-fractional Poisson integrals of $\mu $ on a
quasicube $Q$ are given by:%
\begin{eqnarray*}
\mathrm{P}^{\alpha }\left( Q,\mu \right) &\equiv &\int_{\mathbb{R}^{n}}\frac{%
\left\vert Q\right\vert ^{\frac{1}{n}}}{\left( \left\vert Q\right\vert ^{%
\frac{1}{n}}+\left\vert x-x_{Q}\right\vert \right) ^{n+1-\alpha }}d\mu
\left( x\right) , \\
\mathcal{P}^{\alpha }\left( Q,\mu \right) &\equiv &\int_{\mathbb{R}%
^{n}}\left( \frac{\left\vert Q\right\vert ^{\frac{1}{n}}}{\left( \left\vert
Q\right\vert ^{\frac{1}{n}}+\left\vert x-x_{Q}\right\vert \right) ^{2}}%
\right) ^{n-\alpha }d\mu \left( x\right) ,
\end{eqnarray*}%
where we emphasize that $\left\vert x-x_{Q}\right\vert $ denotes Euclidean
distance between $x$ and $x_{Q}$ and $\left\vert Q\right\vert $ denotes the
Lebesgue measure of the quasicube $Q$. We refer to $\mathrm{P}^{\alpha }$ as
the \emph{standard} Poisson integral and to $\mathcal{P}^{\alpha }$ as the 
\emph{reproducing} Poisson integral.

We say that the pair $K,K^{\prime }$ in $\mathcal{P}^{n}$ are \emph{%
neighbours} if $K$ and $K^{\prime }$ live in a common dyadic grid and both $%
K\subset 3K^{\prime }\setminus K^{\prime }$ and $K^{\prime }\subset
3K\setminus K$, and we denote by $\mathcal{N}^{n}$ the set of pairs $\left(
K,K^{\prime }\right) $ in $\mathcal{P}^{n}\times \mathcal{P}^{n}$ that are
neighbours. Let 
\begin{equation*}
\Omega \mathcal{N}^{n}=\left\{ \left( \Omega K,\Omega K^{\prime }\right)
:\left( K,K^{\prime }\right) \in \mathcal{N}^{n}\right\}
\end{equation*}%
be the corresponding collection of neighbour pairs of quasicubes. Let $%
\sigma $ and $\omega $ be locally finite positive Borel measures on $\mathbb{%
R}^{n}$, possibly having common point masses, and suppose $0\leq \alpha <n$.
Then we define the classical \emph{offset }$A_{2}^{\alpha }$\emph{\ constants%
} by 
\begin{equation}
A_{2}^{\alpha }\left( \sigma ,\omega \right) \equiv \sup_{\left( Q,Q^{\prime
}\right) \in \Omega \mathcal{N}^{n}}\frac{\left\vert Q\right\vert _{\sigma }%
}{\left\vert Q\right\vert ^{1-\frac{\alpha }{n}}}\frac{\left\vert Q^{\prime
}\right\vert _{\omega }}{\left\vert Q\right\vert ^{1-\frac{\alpha }{n}}}.
\label{offset A2}
\end{equation}%
Since the cubes in $\mathcal{P}^{n}$ are products of half open, half closed
intervals $\left[ a,b\right) $, the neighbouring quasicubes $\left(
Q,Q^{\prime }\right) \in \Omega \mathcal{N}^{n}$ are disjoint, and the
common point masses of $\sigma $ and $\omega $ do not simultaneously appear
in each factor.

We now define the \emph{one-tailed} $\mathcal{A}_{2}^{\alpha }$ constant
using $\mathcal{P}^{\alpha }$. The energy constants $\mathcal{E}_{\alpha }^{%
\limfunc{strong}}$ introduced in the next subsection will use the standard
Poisson integral $\mathrm{P}^{\alpha }$.

\begin{definition}
The one-tailed constants $\mathcal{A}_{2}^{\alpha }$ and $\mathcal{A}%
_{2}^{\alpha ,\ast }$ for the weight pair $\left( \sigma ,\omega \right) $
are given by%
\begin{eqnarray*}
\mathcal{A}_{2}^{\alpha } &\equiv &\sup_{Q\in \Omega \mathcal{P}^{n}}%
\mathcal{P}^{\alpha }\left( Q,\mathbf{1}_{Q^{c}}\sigma \right) \frac{%
\left\vert Q\right\vert _{\omega }}{\left\vert Q\right\vert ^{1-\frac{\alpha 
}{n}}}<\infty , \\
\mathcal{A}_{2}^{\alpha ,\ast } &\equiv &\sup_{Q\in \Omega \mathcal{P}^{n}}%
\mathcal{P}^{\alpha }\left( Q,\mathbf{1}_{Q^{c}}\omega \right) \frac{%
\left\vert Q\right\vert _{\sigma }}{\left\vert Q\right\vert ^{1-\frac{\alpha 
}{n}}}<\infty .
\end{eqnarray*}
\end{definition}

Note that these definitions are the analogues of the corresponding
conditions with `holes' introduced by Hyt\"{o}nen \cite{Hyt} in dimension $%
n=1$ - the supports of the measures $\mathbf{1}_{Q^{c}}\sigma $ and $\mathbf{%
1}_{Q}\omega $ in the definition of $\mathcal{A}_{2}^{\alpha }$ are
disjoint, and so the common point masses of $\sigma $ and $\omega $ do not
appear simultaneously in each factor. Note also that, unlike in \cite%
{SaShUr5}, where common point masses were not permitted, we can no longer
assert the equivalence of $\mathcal{A}_{2}^{\alpha }$ with holes taken over 
\emph{quasicubes} with $\mathcal{A}_{2}^{\alpha }$ with holes taken over 
\emph{cubes}.

\subsubsection{Punctured $A_{2}^{\protect\alpha }$ conditions}

As mentioned earlier, the \emph{classical} $A_{2}^{\alpha }$ characteristic $%
\sup_{Q\in \Omega \mathcal{Q}^{n}}\frac{\left\vert Q\right\vert _{\omega }}{%
\left\vert Q\right\vert ^{1-\frac{\alpha }{n}}}\frac{\left\vert Q\right\vert
_{\sigma }}{\left\vert Q\right\vert ^{1-\frac{\alpha }{n}}}$ fails to be
finite when the measures $\sigma $ and $\omega $ have a common point mass -
simply let $Q$ in the $\sup $ above shrink to a common mass point. But there
is a substitute that is quite similar in character that is motivated by the
fact that for large quasicubes $Q$, the $\sup $ above is problematic only if
just \emph{one} of the measures is \emph{mostly} a point mass when
restricted to $Q$. The one-dimensional version of the condition we are about
to describe arose in Conjecture 1.12 of Lacey \cite{Lac2}, and it was
pointed out in \cite{Hyt2} that its necessity on the line follows from the
proof of Proposition 2.1 in \cite{LaSaUr2}. We now extend this condition to
higher dimensions, where its necessity is more subtle.

Given an at most countable set $\mathfrak{P}=\left\{ p_{k}\right\}
_{k=1}^{\infty }$ in $\mathbb{R}^{n}$, a quasicube $Q\in \Omega \mathcal{P}%
^{n}$, and a positive locally finite Borel measure $\mu $, define 
\begin{equation*}
\mu \left( Q,\mathfrak{P}\right) \equiv \left\vert Q\right\vert _{\mu }-\sup
\left\{ \mu \left( p_{k}\right) :p_{k}\in Q\cap \mathfrak{P}\right\} ,
\end{equation*}%
where the supremum is actually achieved since $\sum_{p_{k}\in Q\cap 
\mathfrak{P}}\mu \left( p_{k}\right) <\infty $ as $\mu $ is locally finite.
The quantity $\mu \left( Q,\mathfrak{P}\right) $ is simply the $\widetilde{%
\mu }$ measure of $Q$ where $\widetilde{\mu }$ is the measure $\mu $ with
its largest point mass from $\mathfrak{P}$ in $Q$ removed. Given a locally
finite measure pair $\left( \sigma ,\omega \right) $, let $\mathfrak{P}%
_{\left( \sigma ,\omega \right) }=\left\{ p_{k}\right\} _{k=1}^{\infty }$ be
the at most countable set of common point masses of $\sigma $ and $\omega $.
Then the weighted norm inequality (\ref{two weight'}) typically implies
finiteness of the following \emph{punctured} Muckenhoupt conditions:%
\begin{eqnarray*}
A_{2}^{\alpha ,\limfunc{punct}}\left( \sigma ,\omega \right) &\equiv
&\sup_{Q\in \Omega \mathcal{P}^{n}}\frac{\omega \left( Q,\mathfrak{P}%
_{\left( \sigma ,\omega \right) }\right) }{\left\vert Q\right\vert ^{1-\frac{%
\alpha }{n}}}\frac{\left\vert Q\right\vert _{\sigma }}{\left\vert
Q\right\vert ^{1-\frac{\alpha }{n}}}, \\
A_{2}^{\alpha ,\ast ,\limfunc{punct}}\left( \sigma ,\omega \right) &\equiv
&\sup_{Q\in \Omega \mathcal{P}^{n}}\frac{\left\vert Q\right\vert _{\omega }}{%
\left\vert Q\right\vert ^{1-\frac{\alpha }{n}}}\frac{\sigma \left( Q,%
\mathfrak{P}_{\left( \sigma ,\omega \right) }\right) }{\left\vert
Q\right\vert ^{1-\frac{\alpha }{n}}}.
\end{eqnarray*}

\begin{lemma}
\label{pointed A2}Let $\mathbf{T}^{\alpha }$ be an $\alpha $-fractional
singular integral operator as above, and suppose that there is a positive
constant $C_{0}$ such that%
\begin{equation*}
A_{2}^{\alpha }\left( \sigma ,\omega \right) \leq C_{0}\mathfrak{N}_{\mathbf{%
T}^{\alpha }}^{2}\left( \sigma ,\omega \right) ,
\end{equation*}%
for all pairs $\left( \sigma ,\omega \right) $ of positive locally finite
measures \textbf{having no common point masses}. Now let $\sigma $ and $%
\omega $ be positive locally finite Borel measures on $\mathbb{R}^{n}$ and
let $\mathfrak{P}_{\left( \sigma ,\omega \right) }$ be the possibly nonempty
set of common point masses. Then we have%
\begin{equation*}
A_{2}^{\alpha ,\limfunc{punct}}\left( \sigma ,\omega \right) +A_{2}^{\alpha
,\ast ,\limfunc{punct}}\left( \sigma ,\omega \right) \leq 4C_{0}\mathfrak{N}%
_{\mathbf{T}^{\alpha }}^{2}\left( \sigma ,\omega \right) .
\end{equation*}
\end{lemma}

\begin{proof}
Fix a quasicube $Q\in \Omega \mathcal{P}^{n}$. Suppose first that $\mathfrak{%
P}_{\left( \sigma ,\omega \right) }\cap Q=\left\{ p_{k}\right\} _{k=1}^{2N}$
is finite. Choose $k_{1}\in \mathbb{N}_{2N}=\left\{ 1,2,...,2N\right\} $ so
that 
\begin{equation*}
\sigma \left( p_{k_{1}}\right) =\max_{k\in \mathbb{N}_{2N}}\sigma \left(
p_{k}\right) .
\end{equation*}%
Then choose $k_{2}\in \mathbb{N}_{2N}\setminus \left\{ k_{1}\right\} $ such
that 
\begin{equation*}
\omega \left( p_{k_{2}}\right) =\max_{k\in \mathbb{N}_{2N}\setminus \left\{
k_{1}\right\} }\omega \left( p_{k}\right) .
\end{equation*}%
Repeat this procedure so that%
\begin{eqnarray*}
\sigma \left( p_{k_{2m+1}}\right) &=&\max_{k\in \mathbb{N}_{2N}\setminus
\left\{ k_{1},...k_{2m}\right\} }\sigma \left( p_{k}\right) ,\ \ \ \ \
k_{2m+1}\in \mathbb{N}_{2N}\setminus \left\{ k_{1},...k_{2m}\right\} , \\
\omega \left( p_{k_{2m+2}}\right) &=&\max_{k\in \mathbb{N}_{2N}\setminus
\left\{ k_{1},...k_{2m+1}\right\} }\omega \left( p_{k}\right) ,\ \ \ \ \
k_{2m+2}\in \mathbb{N}_{2N}\setminus \left\{ k_{1},...k_{2m+1}\right\} ,
\end{eqnarray*}%
for each $m\leq N-1$. It is now clear that both%
\begin{equation*}
\sum_{i=0}^{N-1}\sigma \left( p_{k_{2i+1}}\right) \geq \frac{1}{2}\sigma
\left( Q\cap \mathfrak{P}_{\left( \sigma ,\omega \right) }\right) \text{ and 
}\sum_{i=0}^{N-1}\omega \left( p_{k_{2i+2}}\right) \geq \frac{1}{2}\left[
\omega \left( Q\cap \mathfrak{P}_{\left( \sigma ,\omega \right) }\right)
-\omega \left( p_{1}\right) \right] .
\end{equation*}

Now define new measures $\widetilde{\sigma }$ and $\widetilde{\omega }$ by%
\begin{equation*}
\widetilde{\sigma }\equiv \mathbf{1}_{Q}\sigma -\sum_{i=0}^{N-1}\sigma
\left( p_{k_{2i+2}}\right) \delta _{p_{k_{2i+2}}}\text{ and }\widetilde{%
\omega }=\mathbf{1}_{Q}\omega -\sum_{i=0}^{N-1}\omega \left(
p_{k_{2i+1}}\right) \delta _{p_{k_{2i+1}}}
\end{equation*}%
so that%
\begin{equation*}
\left\vert Q\right\vert _{\widetilde{\sigma }}\geq \frac{1}{2}\left\vert
Q\right\vert _{\sigma }\text{ and }\left\vert Q\right\vert _{\widetilde{%
\omega }}\geq \frac{1}{2}\omega \left( Q,\mathfrak{P}_{\left( \sigma ,\omega
\right) }\right)
\end{equation*}%
Now $\widetilde{\sigma }$ and $\widetilde{\omega }$ have no common point
masses and $\mathfrak{N}_{\mathbf{T}^{\alpha }}\left( \sigma ,\omega \right) 
$ is monotone in each measure separately, so we have%
\begin{equation*}
\frac{\omega \left( Q,\mathfrak{P}_{\left( \sigma ,\omega \right) }\right) }{%
\left\vert Q\right\vert ^{1-\frac{\alpha }{n}}}\frac{\left\vert Q\right\vert
_{\sigma }}{\left\vert Q\right\vert ^{1-\frac{\alpha }{n}}}\leq
4A_{2}^{\alpha }\left( \widetilde{\sigma },\widetilde{\omega }\right) \leq
4C_{0}\mathfrak{N}_{\mathbf{T}^{\alpha }}^{2}\left( \widetilde{\sigma },%
\widetilde{\omega }\right) \leq 4C_{0}\mathfrak{N}_{\mathbf{T}^{\alpha
}}^{2}\left( \sigma ,\omega \right) .
\end{equation*}%
Thus $A_{2}^{\alpha ,\limfunc{punct}}\left( \sigma ,\omega \right) \leq
4C_{0}\mathfrak{N}_{\mathbf{T}^{\alpha }}^{2}\left( \sigma ,\omega \right) $
if the number of common point masses in $Q$ is finite. A limiting argument
proves the general case. The dual inequality $A_{2}^{\alpha ,\ast ,\limfunc{%
punct}}\left( \sigma ,\omega \right) \leq 4C_{0}\mathfrak{N}_{\mathbf{T}%
^{\alpha }}^{2}\left( \sigma ,\omega \right) $ now follows upon
interchanging the measures $\sigma $ and $\omega $.
\end{proof}

Now we turn to the definition of a quasiHaar basis of $L^{2}\left( \mu
\right) $.

\subsection{A weighted quasiHaar basis}

We will use a construction of a quasiHaar basis in $\mathbb{R}^{n}$ that is
adapted to a measure $\mu $ (c.f. \cite{NTV2} for the nonquasi case). Given
a dyadic quasicube $Q\in \Omega \mathcal{D}$, where $\mathcal{D}$ is a
dyadic grid of cubes from $\mathcal{P}^{n}$, let $\bigtriangleup _{Q}^{\mu }$
denote orthogonal projection onto the finite dimensional subspace $%
L_{Q}^{2}\left( \mu \right) $ of $L^{2}\left( \mu \right) $ that consists of
linear combinations of the indicators of\ the children $\mathfrak{C}\left(
Q\right) $ of $Q$ that have $\mu $-mean zero over $Q$:%
\begin{equation*}
L_{Q}^{2}\left( \mu \right) \equiv \left\{ f=\dsum\limits_{Q^{\prime }\in 
\mathfrak{C}\left( Q\right) }a_{Q^{\prime }}\mathbf{1}_{Q^{\prime
}}:a_{Q^{\prime }}\in \mathbb{R},\int_{Q}fd\mu =0\right\} .
\end{equation*}%
Then we have the important telescoping property for dyadic quasicubes $%
Q_{1}\subset Q_{2}$:%
\begin{equation}
\mathbf{1}_{Q_{0}}\left( x\right) \left( \dsum\limits_{Q\in \left[
Q_{1},Q_{2}\right] }\bigtriangleup _{Q}^{\mu }f\left( x\right) \right) =%
\mathbf{1}_{Q_{0}}\left( x\right) \left( \mathbb{E}_{Q_{0}}^{\mu }f-\mathbb{E%
}_{Q_{2}}^{\mu }f\right) ,\ \ \ \ \ Q_{0}\in \mathfrak{C}\left( Q_{1}\right)
,\ f\in L^{2}\left( \mu \right) .  \label{telescope}
\end{equation}%
We will at times find it convenient to use a fixed orthonormal basis $%
\left\{ h_{Q}^{\mu ,a}\right\} _{a\in \Gamma _{n}}$ of $L_{Q}^{2}\left( \mu
\right) $ where $\Gamma _{n}\equiv \left\{ 0,1\right\} ^{n}\setminus \left\{ 
\mathbf{1}\right\} $ is a convenient index set with $\mathbf{1}=\left(
1,1,...,1\right) $. Then $\left\{ h_{Q}^{\mu ,a}\right\} _{a\in \Gamma _{n}%
\text{ and }Q\in \Omega \mathcal{D}}$ is an orthonormal basis for $%
L^{2}\left( \mu \right) $, with the understanding that we add the constant
function $\mathbf{1}$ if $\mu $ is a finite measure. In particular we have%
\begin{equation*}
\left\Vert f\right\Vert _{L^{2}\left( \mu \right) }^{2}=\sum_{Q\in \Omega 
\mathcal{D}}\left\Vert \bigtriangleup _{Q}^{\mu }f\right\Vert _{L^{2}\left(
\mu \right) }^{2}=\sum_{Q\in \Omega \mathcal{D}}\sum_{a\in \Gamma
_{n}}\left\vert \widehat{f}\left( Q\right) \right\vert ^{2},\ \ \ \ \
\left\vert \widehat{f}\left( Q\right) \right\vert ^{2}\equiv \sum_{a\in
\Gamma _{n}}\left\vert \left\langle f,h_{Q}^{\mu ,a}\right\rangle _{\mu
}\right\vert ^{2},
\end{equation*}%
where the measure is suppressed in the notation $\widehat{f}$. Indeed, this
follows from (\ref{telescope}) and Lebesgue's differentiation theorem for
quasicubes. We also record the following useful estimate. If $I^{\prime }$
is any of the $2^{n}$ $\Omega \mathcal{D}$-children of $I$, and $a\in \Gamma
_{n}$, then 
\begin{equation}
\left\vert \mathbb{E}_{I^{\prime }}^{\mu }h_{I}^{\mu ,a}\right\vert \leq 
\sqrt{\mathbb{E}_{I^{\prime }}^{\mu }\left( h_{I}^{\mu ,a}\right) ^{2}}\leq 
\frac{1}{\sqrt{\left\vert I^{\prime }\right\vert _{\mu }}}.
\label{useful Haar}
\end{equation}

\subsection{The strong quasienergy conditions}

Given a dyadic quasicube $K\in \Omega \mathcal{D}$ and a positive measure $%
\mu $ we define the quasiHaar projection $\mathsf{P}_{K}^{\mu }\equiv
\sum_{_{J\in \Omega \mathcal{D}:\ J\subset K}}\bigtriangleup _{J}^{\mu }$ on 
$K$ by 
\begin{equation*}
\mathsf{P}_{K}^{\mu }f=\sum_{_{J\in \Omega \mathcal{D}:\ J\subset
K}}\dsum\limits_{a\in \Gamma _{n}}\left\langle f,h_{J}^{\mu ,a}\right\rangle
_{\mu }h_{J}^{\mu ,a}\text{ and }\left\Vert \mathsf{P}_{K}^{\mu
}f\right\Vert _{L^{2}\left( \mu \right) }^{2}=\sum_{_{J\in \Omega \mathcal{D}%
:\ J\subset K}}\dsum\limits_{a\in \Gamma _{n}}\left\vert \left\langle
f,h_{J}^{\mu ,a}\right\rangle _{\mu }\right\vert ^{2},
\end{equation*}%
and where a quasiHaar basis $\left\{ h_{J}^{\mu ,a}\right\} _{a\in \Gamma
_{n}\text{ and }J\in \Omega \mathcal{D}\Omega }$ adapted to the measure $\mu 
$ was defined in the subsubsection on a weighted quasiHaar basis above.

Now we define various notions for quasicubes which are inherited from the
same notions for cubes. The main objective here is to use the familiar
notation that one uses for cubes, but now extended to $\Omega $-quasicubes.
We have already introduced the notions of quasigrids $\Omega \mathcal{D}$,
and center, sidelength and dyadic associated to quasicubes $Q\in \Omega 
\mathcal{D}$, as well as quasiHaar functions, and we will continue to extend
to quasicubes the additional familiar notions related to cubes as we come
across them. We begin with the notion of \emph{deeply embedded}. Fix a
quasigrid $\Omega \mathcal{D}$. We say that a dyadic quasicube $J$ is $%
\left( \mathbf{r},\varepsilon \right) $-\emph{deeply embedded} in a (not
necessarily dyadic) quasicube $K$, which we write as $J\Subset _{\mathbf{r}%
,\varepsilon }K$, when $J\subset K$ and both 
\begin{eqnarray}
\ell \left( J\right) &\leq &2^{-\mathbf{r}}\ell \left( K\right) ,
\label{def deep embed} \\
\limfunc{qdist}\left( J,\partial K\right) &\geq &\frac{1}{2}\ell \left(
J\right) ^{\varepsilon }\ell \left( K\right) ^{1-\varepsilon },  \notag
\end{eqnarray}%
where we define the quasidistance $\limfunc{qdist}\left( E,F\right) $
between two sets $E$ and $F$ to be the Euclidean distance $\limfunc{dist}%
\left( \Omega ^{-1}E,\Omega ^{-1}F\right) $ between the preimages $\Omega
^{-1}E$ and $\Omega ^{-1}F$ of $E$ and $F$ under the map $\Omega $, and
where we recall that $\ell \left( J\right) \approx \left\vert J\right\vert ^{%
\frac{1}{n}}$. For the most part we will consider $J\Subset _{\mathbf{r}%
,\varepsilon }K$ when $J$ and $K$ belong to a common quasigrid $\Omega 
\mathcal{D}$, but an exception is made when defining the strong energy
constants below.

Recall that in dimension $n=1$, and for $\alpha =0$, the energy condition
constant was defined by%
\begin{equation*}
\mathcal{E}^{2}\equiv \sup_{I=\dot{\cup}I_{r}}\frac{1}{\left\vert
I\right\vert _{\sigma }}\sum_{r=1}^{\infty }\left( \frac{\mathrm{P}^{\alpha
}\left( I_{r},\mathbf{1}_{I}\sigma \right) }{\left\vert I_{r}\right\vert }%
\right) ^{2}\left\Vert \mathsf{P}_{I_{r}}^{\omega }\mathbf{x}\right\Vert
_{L^{2}\left( \omega \right) }^{2}\ ,
\end{equation*}%
where $I$, $I_{r}$ and $J$ are intervals in the real line. The extension to
higher dimensions we use here is that of `strong quasienergy condition'
below. Later on, in the proof of the theorem, we will break down this strong
quasienergy condition into various smaller quasienergy conditions, which are
then used in different ways in the proof.

We define a quasicube $K$ (not necessarily in $\Omega \mathcal{D}$) to be an 
\emph{alternate} $\Omega \mathcal{D}$-quasicube if it is a union of $2^{n}$ $%
\Omega \mathcal{D}$-quasicubes $K^{\prime }$ with side length $\ell \left(
K^{\prime }\right) =\frac{1}{2}\ell \left( K\right) $ (such quasicubes were
called shifted in \cite{SaShUr5}, but that terminology conflicts with the
more familiar notion of shifted quasigrid). Thus for any $\Omega \mathcal{D}$%
-quasicube $L$ there are exactly $2^{n}$ alternate $\Omega \mathcal{D}$%
-quasicubes of twice the side length that contain $L$, and one of them is of
course the $\Omega \mathcal{D}$-parent of $L$. We denote the collection of
alternate $\Omega \mathcal{D}$-quasicubes by $\mathcal{A}\Omega \mathcal{D}$.

The extension of the energy conditions to higher dimensions in \cite{SaShUr5}
used the collection 
\begin{equation*}
\mathcal{M}_{\mathbf{r},\varepsilon -\limfunc{deep}}\left( K\right) \equiv
\left\{ \text{maximal }J\Subset _{\mathbf{r},\varepsilon }K\right\}
\end{equation*}%
of \emph{maximal} $\left( \mathbf{r},\varepsilon \right) $-deeply embedded
dyadic subquasicubes of a quasicube $K$ (a subquasicube $J$ of $K$ is a 
\emph{dyadic} subquasicube of $K$ if $J\in \Omega \mathcal{D}$ when $\Omega 
\mathcal{D}$ is a dyadic quasigrid containing $K$). This collection of
dyadic subquasicubes of $K$ is of course a pairwise disjoint decomposition
of $K$.We also defined there a refinement and extension of the collection $%
\mathcal{M}_{\left( \mathbf{r},\varepsilon \right) -\limfunc{deep}}\left(
K\right) $ for certain $K$ and each $\ell \geq 1$. For an alternate
quasicube $K\in \mathcal{A}\Omega \mathcal{D}$, define $\mathcal{M}_{\left( 
\mathbf{r},\varepsilon \right) -\limfunc{deep},\Omega \mathcal{D}}\left(
K\right) $ to consist of the \emph{maximal} $\mathbf{r}$-deeply embedded $%
\Omega \mathcal{D}$-dyadic subquasicubes $J$ of $K$. (In the special case
that $K$ itself belongs to $\Omega \mathcal{D}$, then $\mathcal{M}_{\left( 
\mathbf{r},\varepsilon \right) -\limfunc{deep},\Omega \mathcal{D}}\left(
K\right) =\mathcal{M}_{\left( \mathbf{r},\varepsilon \right) -\limfunc{deep}%
}\left( K\right) $.) Then in \cite{SaShUr5} for $\ell \geq 1$ we defined the
refinement%
\begin{eqnarray*}
\mathcal{M}_{\left( \mathbf{r},\varepsilon \right) -\limfunc{deep},\Omega 
\mathcal{D}}^{\ell }\left( K\right) &\equiv &\left\{ J\in \mathcal{M}%
_{\left( \mathbf{r},\varepsilon \right) -\limfunc{deep},\Omega \mathcal{D}%
}\left( \pi ^{\ell }K^{\prime }\right) \text{ for some }K^{\prime }\in 
\mathfrak{C}_{\Omega \mathcal{D}}\left( K\right) :\right. \\
&&\ \ \ \ \ \ \ \ \ \ \ \ \ \ \ \ \ \ \ \ \ \ \ \ \ \ \ \ \ \ \left.
J\subset L\text{ for some }L\in \mathcal{M}_{\left( \mathbf{r},\varepsilon
\right) -\limfunc{deep}}\left( K\right) \right\} ,
\end{eqnarray*}%
where $\mathfrak{C}_{\Omega \mathcal{D}}\left( K\right) $ is the obvious
extension to alternate quasicubes of the set of $\Omega \mathcal{D}$-dyadic
children. Thus $\mathcal{M}_{\left( \mathbf{r},\varepsilon \right) -\limfunc{%
deep},\Omega \mathcal{D}}^{\ell }\left( K\right) $ is the union, over all
quasichildren $K^{\prime }$ of $K$, of those quasicubes in $\mathcal{M}%
_{\left( \mathbf{r},\varepsilon \right) -\limfunc{deep}}\left( \pi ^{\ell
}K^{\prime }\right) $ that happen to be contained in some $L\in \mathcal{M}%
_{\left( \mathbf{r},\varepsilon \right) -\limfunc{deep},\Omega \mathcal{D}%
}\left( K\right) $. We then define the \emph{strong} quasienergy condition
as follows.

\begin{definition}
\label{def strong quasienergy}Let $0\leq \alpha <n$ and fix parameters $%
\left( \mathbf{r},\varepsilon \right) $. Suppose $\sigma $ and $\omega $ are
positive Borel measures on $\mathbb{R}^{n}$ possibly with common point
masses. Then the \emph{strong} quasienergy constant $\mathcal{E}_{\alpha }^{%
\limfunc{strong}}$ is defined by 
\begin{eqnarray*}
\left( \mathcal{E}_{\alpha }^{\limfunc{strong}}\right) ^{2} &\equiv &\sup_{I=%
\dot{\cup}I_{r}}\frac{1}{\left\vert I\right\vert _{\sigma }}%
\sum_{r=1}^{\infty }\sum_{J\in \mathcal{M}_{\mathbf{r},\varepsilon -\limfunc{%
deep}}\left( I_{r}\right) }\left( \frac{\mathrm{P}^{\alpha }\left( J,\mathbf{%
1}_{I}\sigma \right) }{\left\vert J\right\vert ^{\frac{1}{n}}}\right)
^{2}\left\Vert \mathsf{P}_{J}^{\omega }\mathbf{x}\right\Vert _{L^{2}\left(
\omega \right) }^{2} \\
&&+\sup_{\Omega \mathcal{D}}\sup_{I\in \mathcal{A}\Omega \mathcal{D}%
}\sup_{\ell \geq 0}\frac{1}{\left\vert I\right\vert _{\sigma }}\sum_{J\in 
\mathcal{M}_{\left( \mathbf{r},\varepsilon \right) -\limfunc{deep},\Omega 
\mathcal{D}}^{\ell }\left( I\right) }\left( \frac{\mathrm{P}^{\alpha }\left(
J,\mathbf{1}_{I}\sigma \right) }{\left\vert J\right\vert ^{\frac{1}{n}}}%
\right) ^{2}\left\Vert \mathsf{P}_{J}^{\omega }\mathbf{x}\right\Vert
_{L^{2}\left( \omega \right) }^{2}\ .
\end{eqnarray*}
\end{definition}

Similarly we have a dual version of $\mathcal{E}_{\alpha }^{\limfunc{strong}%
} $ denoted $\mathcal{E}_{\alpha }^{\limfunc{strong},\ast }$, and both
depend on $\mathbf{r}$ and $\varepsilon $ as well as on $n$ and $\alpha $.
An important point in this definition is that the quasicube $I$ in the
second line is permitted to lie \emph{outside} the quasigrid $\Omega 
\mathcal{D}$, but only as an alternate dyadic quasicube $I\in \mathcal{A}%
\Omega \mathcal{D} $. In the setting of quasicubes we continue to use the
linear function $\mathbf{x}$ in the final factor $\left\Vert \mathsf{P}%
_{J}^{\omega }\mathbf{x}\right\Vert _{L^{2}\left( \omega \right) }^{2}$ of
each line, and not the pushforward of $\mathbf{x}$ by $\Omega $. The reason
of course is that this condition is used to capture the first order
information in the Taylor expansion of a singular kernel. There is a
logically weaker form of the quasienergy conditions that we discuss after
stating our main theorem, but these refined quasienergy conditions are more
complicated to state, and have as yet found no application - the strong
energy conditions above suffice for use when one measure is compactly
supported on a $C^{1,\delta }$ curve as in \cite{SaShUr8}.

\subsection{Statement of the Theorems}

We can now state our main quasicube two weight theorem for general measures
allowing common point masses, as well as our application to energy dispersed
measures. Recall that $\Omega :\mathbb{R}^{n}\rightarrow \mathbb{R}^{n}$ is
a globally biLipschitz map, and that $\Omega \mathcal{P}^{n}$ denotes the
collection of all quasicubes in $\mathbb{R}^{n}$ whose preimages under $%
\Omega $ are usual cubes with sides parallel to the coordinate axes. Denote
by $\Omega \mathcal{D}\subset \Omega \mathcal{P}^{n}$ a dyadic quasigrid in $%
\mathbb{R}^{n}$. For the purpose of obtaining necessity of $\mathcal{A}%
_{2}^{\alpha }$ for $\frac{n}{2}\leq \alpha <n$, we adapt the notion of
strong ellipticity from \cite{SaShUr7}.

\begin{definition}
\label{strong ell}Fix a globally biLipschitz map $\Omega $. Let $\mathbf{T}%
^{\alpha }=\left\{ T_{j}^{\alpha }\right\} _{j=1}^{J}$ be a vector of
singular integral operators with standard kernels $\left\{ K_{j}^{\alpha
}\right\} _{j=1}^{J}$. We say that $\mathbf{T}^{\alpha }$ is \emph{strongly
elliptic} with respect to $\Omega $ if for each $m\in \left\{ 1,-1\right\}
^{n}$, there is a sequence of coefficients $\left\{ \lambda _{j}^{m}\right\}
_{j=1}^{J}$ such that%
\begin{equation}
\left\vert \sum_{j=1}^{J}\lambda _{j}^{m}K_{j}^{\alpha }\left( x,x+t\mathbf{u%
}\right) \right\vert \geq ct^{\alpha -n},\ \ \ \ \ t\in \mathbb{R},
\label{Ktalpha strong}
\end{equation}%
holds for \emph{all} unit vectors $\mathbf{u}$ in the quasi-$n$-ant $\Omega
V_{m}$ where%
\begin{equation*}
V_{m}=\left\{ x\in \mathbb{R}^{n}:m_{i}x_{i}>0\text{ for }1\leq i\leq
n\right\} ,\ \ \ \ \ m\in \left\{ 1,-1\right\} ^{n}.
\end{equation*}
\end{definition}

\begin{theorem}
\label{T1 theorem}Suppose that $T^{\alpha }$ is a standard $\alpha $%
-fractional singular integral operator on $\mathbb{R}^{n}$, and that $\omega 
$ and $\sigma $ are positive Borel measures on $\mathbb{R}^{n}$ (possibly
having common point masses). Set $T_{\sigma }^{\alpha }f=T^{\alpha }\left(
f\sigma \right) $ for any smooth truncation of $T_{\sigma }^{\alpha }$. Let $%
\Omega :\mathbb{R}^{n}\rightarrow \mathbb{R}^{n}$ be a globally biLipschitz
map.

\begin{enumerate}
\item Suppose $0\leq \alpha <n$ and that $\gamma \geq 2$ is given. Then the
operator $T_{\sigma }^{\alpha }$ is bounded from $L^{2}\left( \sigma \right) 
$ to $L^{2}\left( \omega \right) $, i.e. 
\begin{equation}
\left\Vert T_{\sigma }^{\alpha }f\right\Vert _{L^{2}\left( \omega \right)
}\leq \mathfrak{N}_{T_{\sigma }^{\alpha }}\left\Vert f\right\Vert
_{L^{2}\left( \sigma \right) },  \label{two weight}
\end{equation}%
uniformly in smooth truncations of $T^{\alpha }$, and moreover%
\begin{equation*}
\mathfrak{N}_{T_{\sigma }^{\alpha }}\leq C_{\alpha }\left( \sqrt{\mathcal{A}%
_{2}^{\alpha }+\mathcal{A}_{2}^{\alpha ,\ast }+A_{2}^{\alpha ,\limfunc{punct}%
}+A_{2}^{\alpha ,\ast ,\limfunc{punct}}}+\mathfrak{T}_{T^{\alpha }}+%
\mathfrak{T}_{T^{\alpha }}^{\ast }+\mathcal{E}_{\alpha }^{\limfunc{strong}}+%
\mathcal{E}_{\alpha }^{\limfunc{strong},\ast }+\mathcal{WBP}_{T^{\alpha
}}\right) ,
\end{equation*}%
provided that the two dual $\mathcal{A}_{2}^{\alpha }$ conditions and the
two dual punctured Muckenhoupt conditions all hold, and the two dual
quasitesting conditions for $T^{\alpha }$ hold, the quasiweak boundedness
property for $T^{\alpha }$ holds for a sufficiently large constant $C$
depending on the goodness parameter $\mathbf{r}$, and provided that the two
dual strong quasienergy conditions hold uniformly over all dyadic quasigrids 
$\Omega \mathcal{D}\subset \Omega \mathcal{P}^{n}$, i.e. $\mathcal{E}%
_{\alpha }^{\limfunc{strong}}+\mathcal{E}_{\alpha }^{\limfunc{strong},\ast
}<\infty $, and where the goodness parameters $\mathbf{r}$ and $\varepsilon $
implicit in the definition of the collections $\mathcal{M}_{\left( \mathbf{r}%
,\varepsilon \right) -\limfunc{deep}}\left( K\right) $ and $\mathcal{M}%
_{\left( \mathbf{r},\varepsilon \right) -\limfunc{deep},\Omega \mathcal{D}%
}^{\ell }\left( K\right) $ appearing in the strong energy conditions, are
fixed sufficiently large and small respectively depending only on $n$ and $%
\alpha $.

\item Conversely, suppose $0\leq \alpha <n$ and that $\mathbf{T}^{\alpha
}=\left\{ T_{j}^{\alpha }\right\} _{j=1}^{J}$ is a vector of Calder\'{o}%
n-Zygmund operators with standard kernels $\left\{ K_{j}^{\alpha }\right\}
_{j=1}^{J}$. In the range $0\leq \alpha <\frac{n}{2}$, we assume the \emph{%
ellipticity} condition from (\cite{SaShUr7}): there is $c>0$ such that for 
\emph{each} unit vector $\mathbf{u}$ there is $j$ satisfying%
\begin{equation}
\left\vert K_{j}^{\alpha }\left( x,x+t\mathbf{u}\right) \right\vert \geq
ct^{\alpha -n},\ \ \ \ \ t\in \mathbb{R}.  \label{Ktalpha}
\end{equation}%
For the range $\frac{n}{2}\leq \alpha <n$, we assume the strong ellipticity
condition in Definition \ref{strong ell} above. Furthermore, assume that
each operator $T_{j}^{\alpha }$ is bounded from $L^{2}\left( \sigma \right) $
to $L^{2}\left( \omega \right) $, 
\begin{equation*}
\left\Vert \left( T_{j}^{\alpha }\right) _{\sigma }f\right\Vert
_{L^{2}\left( \omega \right) }\leq \mathfrak{N}_{T_{j}^{\alpha }}\left\Vert
f\right\Vert _{L^{2}\left( \sigma \right) }.
\end{equation*}%
Then the fractional $\mathcal{A}_{2}^{\alpha }$ conditions (with `holes')
hold as well as the punctured Muckenhoupt conditions, and moreover,%
\begin{equation*}
\sqrt{\mathcal{A}_{2}^{\alpha }+\mathcal{A}_{2}^{\alpha ,\ast
}+A_{2}^{\alpha ,\limfunc{punct}}+A_{2}^{\alpha ,\ast ,\limfunc{punct}}}\leq
C\mathfrak{N}_{\mathbf{T}^{\alpha }}.
\end{equation*}
\end{enumerate}
\end{theorem}

\begin{problem}
Given any strongly elliptic vector $\mathbf{T}^{\alpha }$ of classical $%
\alpha $-fractional Calder\'{o}n-Zygmund operators, it is an open question
whether or not the usual (quasi or not) energy conditions are necessary for
boundedness of $\mathbf{T}^{\alpha }$. See \cite{SaShUr4} for a failure of 
\emph{energy reversal} in higher dimensions - such an energy reversal was
used in dimension $n=1$ to prove the necessity of the energy condition for
the Hilbert transform, and also in \cite{SaShUr3} and \cite{LaSaShUrWi} for
the Riesz transforms and Cauchy transforms respectively when one of the
measures is supported on a line, and in \cite{SaShUr8} when one of the
measures is supported on a $C^{1,\delta }$ curve.
\end{problem}

\begin{remark}
If Definition \ref{strong ell} holds for some $\mathbf{T}^{\alpha }$ and $%
\Omega $, then $\Omega $ must be fairly tame, in particular the logarithmic
spirals in Example \ref{wild} are ruled out! On the other hand, the vector
of Riesz transforms $\mathbf{R}^{\alpha ,n}$ is easily seen to be strongly
elliptic with respect to $\Omega $ if $\Omega $ satisfies the following 
\emph{sector separation property}. Given a hyperplane $H$ and a
perpendicular line $L$ intersecting at point $P$, there exist spherical
cones $S_{H}$ and $S_{L}$ intersecting only at the point $P^{\prime }=\Omega
\left( P\right) $, such that $H^{\prime }\equiv \Omega H\subset S_{H}$ and $%
L^{\prime }\equiv \Omega L\subset S_{L}$ and%
\begin{equation*}
\limfunc{dist}\left( x,\partial S_{H}\right) \approx \left\vert x\right\vert
,\ x\in H\text{ and }\limfunc{dist}\left( x,\partial S_{L}\right) \approx
\left\vert x\right\vert ,\ x\in L\ .
\end{equation*}%
Examples of globally biLipshcitz maps $\Omega $ that satisfy the sector
separation property include finite compositions of maps of the form%
\begin{equation*}
\Omega \left( x_{1},x^{\prime }\right) =\left( x_{1},x^{\prime }+\psi \left(
x_{1}\right) \right) ,\ \ \ \ \ \left( x_{1},x^{\prime }\right) \in \mathbb{R%
}^{n},
\end{equation*}%
where $\psi :\mathbb{R}\rightarrow \mathbb{R}^{n-1}$ is a Lipschitz map with
sufficiently small Lipschitz constant.
\end{remark}

In order to state our application to energy dispersed measures, we introduce
some notation and a definition. Fix a globally biLipschitz map $\Omega :%
\mathbb{R}^{n}\rightarrow \mathbb{R}^{n}$. For $0\leq k\leq n-1$, denote by $%
\mathcal{L}_{k}^{n}$ the collection of all $k$-dimensional planes in $%
\mathbb{R}^{n}$. If in addition $J$ is an $\Omega $-quasicube in $\mathbb{R}%
^{n}$, denote by $M_{k}^{n}\left( J,\mu \right) $ the moments%
\begin{equation*}
\mathsf{M}_{k}^{n}\left( J,\mu \right) ^{2}\equiv \inf_{L\in \mathcal{L}%
_{k}^{n}}\int_{J}\limfunc{dist}\left( x,L\right) ^{2}d\mu \left( x\right) ,
\end{equation*}%
and note that $\mathsf{M}_{0}^{n}\left( J,\mu \right) $ is related to the
energy $\mathsf{E}\left( J,\mu \right) $: 
\begin{equation*}
\mathsf{M}_{0}^{n}\left( J,\mu \right) ^{2}=\int_{J}\left\vert x-\mathbb{E}%
_{J}^{\mu }x\right\vert ^{2}d\mu \left( x\right) =\left\vert J\right\vert
_{\mu }\left\vert J\right\vert ^{\frac{2}{n}}\mathsf{E}\left( J,\mu \right)
^{2}.
\end{equation*}%
Clearly the moments decrease in $k$ and we now give a name to various
reversals of this decrease.

\begin{definition}
Suppose $\mu $ is a locally finite Borel measure on $\mathbb{R}^{n}$, and
let $k$ be an integer satisfying $0\leq k\leq n-1$. We say that $\mu $ is $k$%
\emph{-energy dispersed} if there is a positive constant $C$ such that for
all $\Omega $-quasicubes $J$,%
\begin{equation*}
\mathsf{M}_{0}^{n}\left( J,\mu \right) \leq C\mathsf{M}_{k}^{n}\left( J,\mu
\right) .
\end{equation*}
\end{definition}

If both $\sigma $ and $\omega $ are appropriately energy dispersed relative
to the order $0\leq \alpha <n$, then the $T1$ theorem holds for the $\alpha $%
-fractional Riesz vector transform $\mathbf{R}^{\alpha ,n}$.

\begin{theorem}
\label{dispersed}Let $0\leq \alpha <n$ and $0\leq k\leq n-1$ sastisfy%
\begin{equation*}
\left\{ 
\begin{array}{ccc}
n-k<\alpha <n,\ \alpha \neq n-1 & \text{ if } & 1\leq k\leq n-2 \\ 
0\leq \alpha <n,\ \alpha \neq 1,n-1 & \text{ if } & k=n-1%
\end{array}%
\right. .
\end{equation*}%
Suppose that $\mathbf{R}^{\alpha ,n}$ is the $\alpha $-fractional Riesz
vector transform on $\mathbb{R}^{n}$, and that $\omega $ and $\sigma $ are $%
k $\emph{-energy dispersed} locally finite positive Borel measures on $%
\mathbb{R}^{n}$ (possibly having common point masses). Set $\mathbf{R}%
_{\sigma }^{\alpha ,n}f=\mathbf{R}_{\sigma }^{\alpha ,n}\left( f\sigma
\right) $ for any smooth truncation of $\mathbf{R}^{\alpha ,n}$. Let $\Omega
:\mathbb{R}^{n}\rightarrow \mathbb{R}^{n}$ be a globally biLipschitz map.
Then the operator $\mathbf{R}_{\sigma }^{\alpha ,n}$ is bounded from $%
L^{2}\left( \sigma \right) $ to $L^{2}\left( \omega \right) $, i.e. 
\begin{equation*}
\left\Vert T_{\sigma }^{\alpha }f\right\Vert _{L^{2}\left( \omega \right)
}\leq \mathfrak{N}_{\mathbf{R}_{\sigma }^{\alpha }}\left\Vert f\right\Vert
_{L^{2}\left( \sigma \right) },
\end{equation*}%
uniformly in smooth truncations of $\mathbf{R}^{\alpha ,n}$, if and only if
the Muckenhoupt conditions hold, the testing conditions hold and the weak
boundedness property holds. Moreover, we have the equivalence%
\begin{equation*}
\mathfrak{N}_{\mathbf{R}_{\sigma }^{\alpha ,n}}\approx \sqrt{\mathcal{A}%
_{2}^{\alpha }+\mathcal{A}_{2}^{\alpha ,\ast }+A_{2}^{\alpha ,\limfunc{punct}%
}+A_{2}^{\alpha ,\ast ,\limfunc{punct}}}+\mathfrak{T}_{\mathbf{R}^{\alpha
,n}}+\mathfrak{T}_{\mathbf{R}^{\alpha ,n}}^{\ast }+\mathcal{WBP}_{\mathbf{R}%
^{\alpha ,n}}\ .
\end{equation*}
\end{theorem}

The case $k=n-1$ of $k$-energy dispersed is similar to the notion of
uniformly full dimension introduced by Lacey and Wick in \cite[versions 2
and 3]{LaWi}. The proof of Theorem \ref{dispersed} shows that we can also
take $\omega $ and $\sigma $ to be $k_{1}$ and $k_{2}$\emph{\ }energy
dispersed respectively, provided $\alpha $ satisfies the hypotheses with
respect to \emph{both} $k_{1}$ and $k_{2}$.

\section{Proof of Theorem \protect\ref{T1 theorem}}

We now give the proof of Theorem \ref{T1 theorem} in the following sections.
Sections 5, 7 and 10 are largely taken verbatim from the corresponding
sections of \cite{SaShUr5}, but are included here since their omission here
would hinder the readability of an already complicated argument.

\subsection{Good quasicubes and energy Muckenhoupt conditions}

First we extend the notion of goodness to quasicubes.

\begin{definition}
Let $\mathbf{r}\in \mathbb{N}$ and $0<\varepsilon <1$. Fix a quasigrid $%
\Omega \mathcal{D}$. A dyadic quasicube $J$ is $\left( \mathbf{r}%
,\varepsilon \right) $\emph{-good}, or simply \emph{good}, if for \emph{every%
} dyadic superquasicube $I$, it is the case that \textbf{either} $J$ has
side length greater than $2^{-\mathbf{r}}$ times that of $I$, \textbf{or} $%
J\Subset _{\mathbf{r},\varepsilon }I$ is $\left( \mathbf{r},\varepsilon
\right) $-deeply embedded in $I$.
\end{definition}

Note that this definition simply asserts that a dyadic quasicube $J=\Omega
J^{\prime }$ is $\left( \mathbf{r},\varepsilon \right) $-good if and only if
the cube $J^{\prime }$ is $\left( \mathbf{r},\varepsilon \right) $-good in
the usual sense. Finally, we say that $J$ is $\mathbf{r}$\emph{-nearby} in $%
K $ when $J\subset K$ and%
\begin{equation*}
\ell \left( J\right) >2^{-\mathbf{r}}\ell \left( K\right) .
\end{equation*}%
The parameters $\mathbf{r},\varepsilon $ will be fixed sufficiently large
and small respectively later in the proof, and we denote the set of such
good dyadic quasicubes by $\Omega \mathcal{D}_{\left( \mathbf{r},\varepsilon
\right) -\limfunc{good}}$, or simply $\Omega \mathcal{D}_{\limfunc{good}}$
when the goodness parameters $\left( \mathbf{r},\varepsilon \right) $ are
understood. Note that if $J^{\prime }\in \Omega \mathcal{D}_{\left( \mathbf{r%
},\varepsilon \right) -\limfunc{good}}$ and if $J^{\prime }\subset K\in
\Omega \mathcal{D}$, then \textbf{either} $J^{\prime }$ is $\mathbf{r}$\emph{%
-nearby} in $K$ \textbf{or} $J^{\prime }\subset J\Subset _{\mathbf{r}%
,\varepsilon }K$.

Throughout the proof, it will be convenient to also consider pairs of
quasicubes $J,K$ where $J$ is $\left( \mathbf{\rho },\varepsilon \right) $%
\emph{-deeply embedded} in $K$, written $J\Subset _{\mathbf{\rho }%
,\varepsilon }K$ and meaning (\ref{def deep embed}) holds with the same $%
\varepsilon >0$ but with $\mathbf{\rho }$ in place of $\mathbf{r}$; as well
as pairs of quasicubes $J,K$ where $J$ is $\mathbf{\rho }$\emph{-nearby} in%
\emph{\ }$K$, $\ell \left( J\right) >2^{-\mathbf{\rho }}\ell \left( K\right) 
$, for a parameter $\mathbf{\rho }\gg \mathbf{r}$ that will be fixed later.

\begin{notation}
We will typically use the side length $\ell \left( J\right) $ of a $\Omega $%
-quasicube when we are describing collections of quasicubes, and when we
want $\ell \left( J\right) $ to be a dyadic or related number; while in
estimates we will typically use $\left\vert J\right\vert ^{\frac{1}{n}%
}\approx \ell \left( J\right) $, and when we want to compare powers of
volumes of quasicubes. We will continue to use the prefix `quasi' when
discussing quasicubes, quasiHaar, quasienergy and quasidistance in the text,
but will not use the prefix `quasi' when discussing other notions. In
particular, since $\limfunc{quasi}\mathcal{A}_{2}^{\alpha }+\limfunc{quasi}%
A_{2}^{\alpha ,\limfunc{punct}}\approx \mathcal{A}_{2}^{\alpha
}+A_{2}^{\alpha ,\limfunc{punct}}$ (see e.g. \cite{SaShUr8} for a proof) we
do not use $\limfunc{quasi}$ as a prefix for the Muckenhoupt conditions,
even though $\limfunc{quasi}\mathcal{A}_{2}^{\alpha }$ alone is not
comparable to$\mathcal{A}_{2}^{\alpha }$. Finally, we will not modify any
mathematical symbols to reflect quasinotions, except for using $\Omega 
\mathcal{D}$ to denote a quasigrid, and $\limfunc{qdist}\left( E,F\right)
\equiv \limfunc{dist}\left( \Omega ^{-1}E,\Omega ^{-1}F\right) $ to denote
quasidistance between sets $E$ and $F$, and using $\left\vert x-y\right\vert
_{\limfunc{qdist}}\equiv \left\vert \Omega ^{-1}x-\Omega ^{-1}y\right\vert $
to denote quasidistance between points $x$ and $y$. This limited use of
quasi in the text serves mainly to remind the reader we are working entirely
in the `quasiworld'.
\end{notation}

\subsubsection{Energy Muckenhoupt conditions}

We now show that the punctured Muckenhoupt conditions $A_{2}^{\alpha ,%
\limfunc{punct}}$ and $A_{2}^{\alpha ,\ast ,\limfunc{punct}}$ control
respectively the `energy $A_{2}^{\alpha }$ conditions', denoted $%
A_{2}^{\alpha ,\limfunc{energy}}$ and $A_{2}^{\alpha ,\ast ,\limfunc{energy}%
} $ where%
\begin{eqnarray}
A_{2}^{\alpha ,\limfunc{energy}}\left( \sigma ,\omega \right) &\equiv
&\sup_{Q\in \Omega \mathcal{P}^{n}}\frac{\left\Vert \mathsf{P}_{Q}^{\omega }%
\frac{\mathbf{x}}{\ell \left( Q\right) }\right\Vert _{L^{2}\left( \omega
\right) }^{2}}{\left\vert Q\right\vert ^{1-\frac{\alpha }{n}}}\frac{%
\left\vert Q\right\vert _{\sigma }}{\left\vert Q\right\vert ^{1-\frac{\alpha 
}{n}}},  \label{def energy A2} \\
A_{2}^{\alpha ,\ast ,\limfunc{energy}}\left( \sigma ,\omega \right) &\equiv
&\sup_{Q\in \Omega \mathcal{P}^{n}}\frac{\left\vert Q\right\vert _{\omega }}{%
\left\vert Q\right\vert ^{1-\frac{\alpha }{n}}}\frac{\left\Vert \mathsf{P}%
_{Q}^{\sigma }\frac{\mathbf{x}}{\ell \left( Q\right) }\right\Vert
_{L^{2}\left( \sigma \right) }^{2}}{\left\vert Q\right\vert ^{1-\frac{\alpha 
}{n}}}.  \notag
\end{eqnarray}%
These energy $A_{2}^{\alpha }$ conditions play a critical role in
controlling local parts of functional energy later in the paper, and it is a
crucial requirement that they are necessary conditions, as shown by the next
lemma.

\begin{lemma}
\label{energy A2}For any positive locally finite Borel measures $\sigma
,\omega $ we have%
\begin{eqnarray*}
A_{2}^{\alpha ,\limfunc{energy}}\left( \sigma ,\omega \right) &\leq &\max
\left\{ n,3\right\} A_{2}^{\alpha ,\limfunc{punct}}\left( \sigma ,\omega
\right) , \\
A_{2}^{\alpha ,\ast ,\limfunc{energy}}\left( \sigma ,\omega \right) &\leq
&\max \left\{ n,3\right\} A_{2}^{\alpha ,\ast ,\limfunc{punct}}\left( \sigma
,\omega \right) .
\end{eqnarray*}
\end{lemma}

\begin{proof}
Fix a quasicube $Q\in \Omega \mathcal{D}$. If $\omega \left( Q,\mathfrak{P}%
_{\left( \sigma ,\omega \right) }\right) \geq \frac{1}{2}\left\vert
Q\right\vert _{\omega }$, then we trivially have%
\begin{eqnarray*}
\frac{\left\Vert \mathsf{P}_{Q}^{\omega }\frac{\mathbf{x}}{\ell \left(
Q\right) }\right\Vert _{L^{2}\left( \omega \right) }^{2}}{\left\vert
Q\right\vert ^{1-\frac{\alpha }{n}}}\frac{\left\vert Q\right\vert _{\sigma }%
}{\left\vert Q\right\vert ^{1-\frac{\alpha }{n}}} &\leq &n\frac{\left\vert
Q\right\vert _{\omega }}{\left\vert Q\right\vert ^{1-\frac{\alpha }{n}}}%
\frac{\left\vert Q\right\vert _{\sigma }}{\left\vert Q\right\vert ^{1-\frac{%
\alpha }{n}}} \\
&\leq &2n\frac{\omega \left( Q,\mathfrak{P}_{\left( \sigma ,\omega \right)
}\right) }{\left\vert Q\right\vert ^{1-\frac{\alpha }{n}}}\frac{\left\vert
Q\right\vert _{\sigma }}{\left\vert Q\right\vert ^{1-\frac{\alpha }{n}}}\leq
2nA_{2}^{\alpha ,\limfunc{punct}}\left( \sigma ,\omega \right) .
\end{eqnarray*}%
On the other hand, if $\omega \left( Q,\mathfrak{P}_{\left( \sigma ,\omega
\right) }\right) <\frac{1}{2}\left\vert Q\right\vert _{\omega }$ then there
is a point $p\in Q\cap \mathfrak{P}_{\left( \sigma ,\omega \right) }$ such
that%
\begin{equation*}
\omega \left( \left\{ p\right\} \right) >\frac{1}{2}\left\vert Q\right\vert
_{\omega }\ ,
\end{equation*}%
and consequently, $p$ is the largest $\omega $-point mass in $Q$. Thus if we
define $\widetilde{\omega }=\omega -\omega \left( \left\{ p\right\} \right)
\delta _{p}$, then we have%
\begin{equation*}
\omega \left( Q,\mathfrak{P}_{\left( \sigma ,\omega \right) }\right)
=\left\vert Q\right\vert _{\widetilde{\omega }}\ .
\end{equation*}%
Now we observe from the construction of Haar projections that%
\begin{equation*}
\bigtriangleup _{J}^{\widetilde{\omega }}=\bigtriangleup _{J}^{\omega },\ \
\ \ \ \text{for all }J\in \Omega \mathcal{D}\text{ with }p\notin J.
\end{equation*}%
So for each $s\geq 0$ there is a unique quasicube $J_{s}\in \Omega \mathcal{D%
}$ with $\ell \left( J_{s}\right) =2^{-s}\ell \left( Q\right) $ that
contains the point $p$. For this quasicube we have, if $\left\{
h_{J}^{\omega ,a}\right\} _{J\in \Omega \mathcal{D},\ a\in \Gamma _{n}}$ is
a basis for $L^{2}\left( \omega \right) $,%
\begin{eqnarray*}
\left\Vert \bigtriangleup _{J_{s}}^{\omega }\mathbf{x}\right\Vert
_{L^{2}\left( \omega \right) }^{2} &=&\sum_{a\in \Gamma _{n}}\left\vert
\left\langle h_{J_{s}}^{\omega ,a},x\right\rangle _{\omega }\right\vert
^{2}=\sum_{a\in \Gamma _{n}}\left\vert \left\langle h_{J_{s}}^{\omega
,a},x-p\right\rangle _{\omega }\right\vert ^{2} \\
&=&\sum_{a\in \Gamma _{n}}\left\vert \int_{J_{s}}h_{J_{s}}^{\omega ,a}\left(
x\right) \left( x-p\right) d\omega \left( x\right) \right\vert
^{2}=\sum_{a\in \Gamma _{n}}\left\vert \int_{J_{s}}h_{J_{s}}^{\omega
,a}\left( x\right) \left( x-p\right) d\widetilde{\omega }\left( x\right)
\right\vert ^{2} \\
&\leq &\sum_{a\in \Gamma _{n}}\left\Vert h_{J_{s}}^{\omega ,a}\right\Vert
_{L^{2}\left( \widetilde{\omega }\right) }^{2}\left\Vert \mathbf{1}%
_{J_{s}}\left( x-p\right) \right\Vert _{L^{2}\left( \widetilde{\omega }%
\right) }^{2}\leq \sum_{a\in \Gamma _{n}}\left\Vert h_{J_{s}}^{\omega
,a}\right\Vert _{L^{2}\left( \omega \right) }^{2}\left\Vert \mathbf{1}%
_{J_{s}}\left( x-p\right) \right\Vert _{L^{2}\left( \widetilde{\omega }%
\right) }^{2} \\
&\leq &n2^{n}\ell \left( J_{s}\right) ^{2}\left\vert J_{s}\right\vert _{%
\widetilde{\omega }}\leq 2^{-2s}\ell \left( Q\right) ^{2}\left\vert
Q\right\vert _{\widetilde{\omega }}.
\end{eqnarray*}%
Thus we can estimate%
\begin{eqnarray*}
\left\Vert \mathsf{P}_{Q}^{\omega }\frac{\mathbf{x}}{\ell \left( Q\right) }%
\right\Vert _{L^{2}\left( \omega \right) }^{2} &=&\frac{1}{\ell \left(
Q\right) ^{2}}\sum_{J\in \Omega \mathcal{D}:\ J\subset Q}\left\Vert
\bigtriangleup _{J}^{\omega }\mathbf{x}\right\Vert _{L^{2}\left( \omega
\right) }^{2} \\
&=&\frac{1}{\ell \left( Q\right) ^{2}}\left( \sum_{J\in \Omega \mathcal{D}:\
p\notin J\subset Q}\left\Vert \bigtriangleup _{J}^{\widetilde{\omega }}%
\mathbf{x}\right\Vert _{L^{2}\left( \widetilde{\omega }\right)
}^{2}+\sum_{s=0}^{\infty }\left\Vert \bigtriangleup _{J_{s}}^{\omega }%
\mathbf{x}\right\Vert _{L^{2}\left( \omega \right) }^{2}\right) \\
&\leq &\frac{1}{\ell \left( Q\right) ^{2}}\left( \left\Vert \mathsf{P}_{Q}^{%
\widetilde{\omega }}\mathbf{x}\right\Vert _{L^{2}\left( \widetilde{\omega }%
\right) }^{2}+\sum_{s=0}^{\infty }2^{-2s}\ell \left( Q\right) ^{2}\left\vert
Q\right\vert _{\widetilde{\omega }}\right) \\
&\leq &\frac{1}{\ell \left( Q\right) ^{2}}\left( \ell \left( Q\right)
^{2}\left\vert Q\right\vert _{\widetilde{\omega }}+\sum_{s=0}^{\infty
}2^{-2s}\ell \left( Q\right) ^{2}\left\vert Q\right\vert _{\widetilde{\omega 
}}\right) \\
&\leq &3\left\vert Q\right\vert _{\widetilde{\omega }}\leq 3\omega \left( Q,%
\mathfrak{P}_{\left( \sigma ,\omega \right) }\right) ,
\end{eqnarray*}%
and so 
\begin{equation*}
\frac{\left\Vert \mathsf{P}_{Q}^{\omega }\frac{\mathbf{x}}{\ell \left(
Q\right) }\right\Vert _{L^{2}\left( \omega \right) }^{2}}{\left\vert
Q\right\vert ^{1-\frac{\alpha }{n}}}\frac{\left\vert Q\right\vert _{\sigma }%
}{\left\vert Q\right\vert ^{1-\frac{\alpha }{n}}}\leq \frac{3\omega \left( Q,%
\mathfrak{P}_{\left( \sigma ,\omega \right) }\right) }{\left\vert
Q\right\vert ^{1-\frac{\alpha }{n}}}\frac{\left\vert Q\right\vert _{\sigma }%
}{\left\vert Q\right\vert ^{1-\frac{\alpha }{n}}}\leq 3A_{2}^{\alpha ,%
\limfunc{punct}}\left( \sigma ,\omega \right) .
\end{equation*}%
Now take the supremum over $Q\in \Omega \mathcal{D}$ to obtain $%
A_{2}^{\alpha ,\limfunc{energy}}\left( \sigma ,\omega \right) \leq \max
\left\{ n,3\right\} A_{2}^{\alpha ,\limfunc{punct}}\left( \sigma ,\omega
\right) $. The dual inequality follows upon interchanging the measures $%
\sigma $ and $\omega $.
\end{proof}

\subsubsection{Plugged $\mathcal{A}_{2}^{\protect\alpha ,\limfunc{energy}%
\limfunc{plug}}$ conditions}

Using Lemma \ref{energy A2} we can control the `plugged' energy $\mathcal{A}%
_{2}^{\alpha }$ conditions:%
\begin{eqnarray*}
\mathcal{A}_{2}^{\alpha ,\limfunc{energy}\limfunc{plug}}\left( \sigma
,\omega \right) &\equiv &\sup_{Q\in \Omega \mathcal{P}^{n}}\frac{\left\Vert 
\mathsf{P}_{Q}^{\omega }\frac{\mathbf{x}}{\ell \left( Q\right) }\right\Vert
_{L^{2}\left( \omega \right) }^{2}}{\left\vert Q\right\vert ^{1-\frac{\alpha 
}{n}}}\mathcal{P}^{\alpha }\left( Q,\sigma \right) , \\
\mathcal{A}_{2}^{\alpha ,\ast ,\limfunc{energy}\limfunc{plug}}\left( \sigma
,\omega \right) &\equiv &\sup_{Q\in \Omega \mathcal{P}^{n}}\mathcal{P}%
^{\alpha }\left( Q,\omega \right) \frac{\left\Vert \mathsf{P}_{Q}^{\sigma }%
\frac{\mathbf{x}}{\ell \left( Q\right) }\right\Vert _{L^{2}\left( \sigma
\right) }^{2}}{\left\vert Q\right\vert ^{1-\frac{\alpha }{n}}}.
\end{eqnarray*}

\begin{lemma}
\label{energy A2 plugged}We have
\end{lemma}

\begin{eqnarray*}
\mathcal{A}_{2}^{\alpha ,\limfunc{energy}\limfunc{plug}}\left( \sigma
,\omega \right) &\mathcal{\lesssim }&\mathcal{A}_{2}^{\alpha }\left( \sigma
,\omega \right) +A_{2}^{\alpha ,\limfunc{energy}}\left( \sigma ,\omega
\right) , \\
\mathcal{A}_{2}^{\alpha ,\ast ,\limfunc{energy}\limfunc{plug}}\left( \sigma
,\omega \right) &\mathcal{\lesssim }&\mathcal{A}_{2}^{\alpha ,\ast }\left(
\sigma ,\omega \right) +A_{2}^{\alpha ,\ast ,\limfunc{energy}}\left( \sigma
,\omega \right) .
\end{eqnarray*}

\begin{proof}
We have%
\begin{eqnarray*}
\frac{\left\Vert \mathsf{P}_{Q}^{\omega }\frac{\mathbf{x}}{\ell \left(
Q\right) }\right\Vert _{L^{2}\left( \omega \right) }^{2}}{\left\vert
Q\right\vert ^{1-\frac{\alpha }{n}}}\mathcal{P}^{\alpha }\left( Q,\sigma
\right) &=&\frac{\left\Vert \mathsf{P}_{Q}^{\omega }\frac{\mathbf{x}}{\ell
\left( Q\right) }\right\Vert _{L^{2}\left( \omega \right) }^{2}}{\left\vert
Q\right\vert ^{1-\frac{\alpha }{n}}}\mathcal{P}^{\alpha }\left( Q,\mathbf{1}%
_{Q^{c}}\sigma \right) +\frac{\left\Vert \mathsf{P}_{Q}^{\omega }\frac{%
\mathbf{x}}{\ell \left( Q\right) }\right\Vert _{L^{2}\left( \omega \right)
}^{2}}{\left\vert Q\right\vert ^{1-\frac{\alpha }{n}}}\mathcal{P}^{\alpha
}\left( Q,\mathbf{1}_{Q}\sigma \right) \\
&\lesssim &\frac{\left\vert Q\right\vert _{\omega }}{\left\vert Q\right\vert
^{1-\frac{\alpha }{n}}}\mathcal{P}^{\alpha }\left( Q,\mathbf{1}%
_{Q^{c}}\sigma \right) +\frac{\left\Vert \mathsf{P}_{Q}^{\omega }\frac{%
\mathbf{x}}{\ell \left( Q\right) }\right\Vert _{L^{2}\left( \omega \right)
}^{2}}{\left\vert Q\right\vert ^{1-\frac{\alpha }{n}}}\frac{\left\vert
Q\right\vert _{\sigma }}{\left\vert Q\right\vert ^{1-\frac{\alpha }{n}}} \\
&\lesssim &\mathcal{A}_{2}^{\alpha }\left( \sigma ,\omega \right)
+A_{2}^{\alpha ,\limfunc{energy}}\left( \sigma ,\omega \right) .
\end{eqnarray*}
\end{proof}

\subsection{Random grids and shifted grids}

Using the analogue for dyadic quasigrids of the good random grids of
Nazarov, Treil and Volberg, a standard argument of NTV, see e.g. \cite{Vol},
reduces the two weight inequality (\ref{two weight}) for $T^{\alpha }$ to
proving boundedness of a bilinear form $\mathcal{T}^{\alpha }\left(
f,g\right) $ with uniform constants over dyadic quasigrids, and where the
quasiHaar supports $\limfunc{supp}\widehat{f}$ and $\limfunc{supp}\widehat{g}
$ of the functions $f$ and $g$ are contained in the collection $\Omega 
\mathcal{D}^{\limfunc{good}}$ of good quasicubes, whose children are all
good as well, with goodness parameters $\mathbf{r<\infty }$ and $\varepsilon
>0$ chosen sufficiently large and small respectively depending only on $n$
and $\alpha $. Here the quasiHaar support of $f$ is $\limfunc{supp}\widehat{f%
}\equiv \left\{ I\in \Omega \mathcal{D}:\bigtriangleup _{I}^{\sigma }f\neq
0\right\} $, and similarly for $g$. In fact we can assume even more, namely
that the quasiHaar supports $\limfunc{supp}\widehat{f}$ and $\limfunc{supp}%
\widehat{g}$ of $f$ and $g$ are contained in the collection of $\mathbf{\tau 
}$\emph{-good} quasicubes%
\begin{equation}
\Omega \mathcal{D}_{\left( \mathbf{r},\varepsilon \right) -\limfunc{good}}^{%
\mathbf{\tau }}\equiv \left\{ K\in \Omega \mathcal{D}:\mathfrak{C}%
_{K}\subset \Omega \mathcal{D}_{\left( \mathbf{r},\varepsilon \right) -%
\limfunc{good}}\text{ and }\pi _{\Omega \mathcal{D}}^{\ell }K\in \Omega 
\mathcal{D}_{\left( \mathbf{r},\varepsilon \right) -\limfunc{good}}\text{
for all }0\leq \ell \leq \mathbf{\tau }\right\} ,  \label{extended good grid}
\end{equation}%
that are $\left( \mathbf{r},\varepsilon \right) $-$\limfunc{good}$ and whose
children are also $\left( \mathbf{r},\varepsilon \right) $-$\limfunc{good}$,
and whose $\ell $-parents up to level $\mathbf{\tau }$\ are also $\left( 
\mathbf{r},\varepsilon \right) $-$\limfunc{good}$. Here $\mathbf{\tau }>%
\mathbf{r}$ is a parameter to be fixed later. We may assume this restriction
on the quasiHaar supports of $f$ and $g$ by the following lemma. See \cite%
{SaShUr6} for a proof\footnote{%
This lemma is misstated in \cite{SaShUr7}.}.

\begin{lemma}
\label{better good}Given $\mathbf{r}\geq 3$, $\mathbf{\tau }\geq 1$ and $%
\frac{1}{\mathbf{r}}<\varepsilon <1-\frac{1}{\mathbf{r}}$, we have 
\begin{equation*}
\Omega \mathcal{D}_{\left( \mathbf{r}-1,\delta \right) -\limfunc{good}%
}\subset \Omega \mathcal{D}_{\left( \mathbf{r},\varepsilon \right) -\limfunc{%
good}}^{\mathbf{\tau }}\ ,
\end{equation*}%
provided%
\begin{equation}
0<\delta \leq \frac{\mathbf{r}\varepsilon -1}{\mathbf{r}+\mathbf{\tau }}.
\label{choice of delta}
\end{equation}
\end{lemma}

For convenience in notation we will sometimes suppress the dependence on $%
\alpha $ in our nonlinear forms, but will retain it in the operators,
Poisson integrals and constants. More precisely, let $\Omega \mathcal{D}%
^{\sigma }=\Omega \mathcal{D}^{\omega }$ be an $\left( \mathbf{r}%
,\varepsilon \right) $-good quasigrid on $\mathbb{R}^{n}$, and let $\left\{
h_{I}^{\sigma ,a}\right\} _{I\in \Omega \mathcal{D}^{\sigma },\ a\in \Gamma
_{n}}$ and $\left\{ h_{J}^{\omega ,b}\right\} _{J\in \Omega \mathcal{D}%
^{\omega },\ b\in \Gamma _{n}}$ be corresponding quasiHaar bases as
described above, so that%
\begin{eqnarray*}
f &=&\sum_{I\in \Omega \mathcal{D}^{\sigma }}\bigtriangleup _{I}^{\sigma
}f=\sum_{I\in \Omega \mathcal{D}^{\sigma },\text{\ }a\in \Gamma _{n}\text{ }%
}\left\langle f,h_{I}^{\sigma ,a}\right\rangle \ h_{I}^{\sigma
,a}=\sum_{I\in \Omega \mathcal{D}^{\sigma },\text{\ }a\in \Gamma _{n}}%
\widehat{f}\left( I;a\right) \ h_{I}^{\sigma ,a}, \\
g &=&\sum_{J\in \Omega \mathcal{D}^{\omega }\text{ }}\bigtriangleup
_{J}^{\omega }g=\sum_{J\in \Omega \mathcal{D}^{\omega },\text{\ }b\in \Gamma
_{n}}\left\langle g,h_{J}^{\omega ,b}\right\rangle \ h_{J}^{\omega
,b}=\sum_{J\in \Omega \mathcal{D}^{\omega },\text{\ }b\in \Gamma _{n}}%
\widehat{g}\left( J;b\right) \ h_{J}^{\omega ,b},
\end{eqnarray*}%
where the appropriate measure is understood in the notation $\widehat{f}%
\left( I;a\right) $ and $\widehat{g}\left( J;b\right) $, and where these
quasiHaar coefficients $\widehat{f}\left( I;a\right) $ and $\widehat{g}%
\left( J;b\right) $ vanish if the quasicubes $I$ and $J$ are not good.
Inequality (\ref{two weight}) is equivalent to boundedness of the bilinear
form%
\begin{equation*}
\mathcal{T}^{\alpha }\left( f,g\right) \equiv \left\langle T_{\sigma
}^{\alpha }\left( f\right) ,g\right\rangle _{\omega }=\sum_{I\in \Omega 
\mathcal{D}^{\sigma }\text{ and }J\in \Omega \mathcal{D}^{\omega
}}\left\langle T_{\sigma }^{\alpha }\left( \bigtriangleup _{I}^{\sigma
}f\right) ,\bigtriangleup _{J}^{\omega }g\right\rangle _{\omega }
\end{equation*}%
on $L^{2}\left( \sigma \right) \times L^{2}\left( \omega \right) $, i.e.%
\begin{equation*}
\left\vert \mathcal{T}^{\alpha }\left( f,g\right) \right\vert \leq \mathfrak{%
N}_{T^{\alpha }}\left\Vert f\right\Vert _{L^{2}\left( \sigma \right)
}\left\Vert g\right\Vert _{L^{2}\left( \omega \right) },
\end{equation*}%
uniformly over all quasigrids and appropriate truncations. We may assume the
two quasigrids $\Omega \mathcal{D}^{\sigma }$ and $\Omega \mathcal{D}%
^{\omega }$ are equal here, and this we will do throughout the paper,
although we sometimes continue to use the measure as a superscript on $%
\Omega \mathcal{D}$ for clarity of exposition. Roughly speaking, we analyze
the form $\mathcal{T}^{\alpha }\left( f,g\right) $ by splitting it in a
nonlinear way into three main pieces, following in part the approach in \cite%
{LaSaShUr2} and \cite{LaSaShUr3}. The first piece consists of quasicubes $I$
and $J$ that are either disjoint or of comparable side length, and this
piece is handled using the section on preliminaries of NTV type. The second
piece consists of quasicubes $I$ and $J$ that overlap, but are `far apart'
in a nonlinear way, and this piece is handled using the sections on the
Intertwining Proposition and the control of the functional quasienergy
condition by the quasienergy condition. Finally, the remaining local piece
where the overlapping quasicubes are `close' is handled by generalizing
methods of NTV as in \cite{LaSaShUr}, and then splitting the stopping form
into two sublinear stopping forms, one of which is handled using techniques
of \cite{LaSaUr2}, and the other using the stopping time and recursion of M.
Lacey \cite{Lac}. See the schematic diagram in Subsection 7.4 below.

We summarize our assumptions on the Haar supports of $f$ and $g$, and on the
dyadic quasigrids $\Omega \mathcal{D}$.

\begin{condition}[on Haar supports and quasigrids]
\label{assume shift}We suppose the quasiHaar supports of the functions $f$
and $g$ satisfy $\limfunc{supp}\widehat{f},\limfunc{supp}\widehat{g}\subset
\Omega \mathcal{D}_{\left( \mathbf{r},\varepsilon \right) -\limfunc{good}}^{%
\mathbf{\tau }}$. We also assume that $\left\vert \partial Q\right\vert
_{\sigma +\omega }=0$ for all dyadic quasicubes $Q$ in the grids $\Omega 
\mathcal{D}$ (since this property holds with probability $1$ for random
grids $\Omega \mathcal{D}$).
\end{condition}

\section{Necessity of the $\mathcal{A}_{2}^{\protect\alpha }$ conditions}

Here we consider in particular the necessity of the fractional $\mathcal{A}%
_{2}^{\alpha }$ condition (with holes) when $0\leq \alpha <n$, for the
boundedness from $L^{2}\left( \sigma \right) $ to $L^{2}\left( \omega
\right) $ (where $\sigma $ and $\omega $ may have common point masses) of
the $\alpha $-fractional Riesz vector transform $\mathbf{R}^{\alpha }$
defined by%
\begin{equation*}
\mathbf{R}^{\alpha }\left( f\sigma \right) \left( x\right) =\int_{\mathbb{R}%
^{n}}K_{j}^{\alpha }(x,y)f\left( y\right) d\sigma \left( y\right) ,\ \ \ \ \
K_{j}^{\alpha }\left( x,y\right) =\frac{x^{j}-y^{j}}{\left\vert
x-y\right\vert ^{n+1-\alpha }},
\end{equation*}%
whose kernel $K_{j}^{\alpha }\left( x,y\right) $ satisfies (\ref%
{sizeandsmoothness'}) for $0\leq \alpha <n$. More generally, necessity holds
for elliptic operators as in the next lemma. See \cite{SaShUr7} for the
easier proof in the case without holes.

\begin{lemma}
\label{necc frac A2}Suppose $0\leq \alpha <n$. Let $T^{\alpha }$ be any
collection of operators with $\alpha $-standard fractional kernel
satisfying\ the ellipticity condition (\ref{Ktalpha}), and in the case $%
\frac{n}{2}\leq \alpha <n$, we also assume the more restrictive condition (%
\ref{Ktalpha strong}). Then for $0\leq \alpha <n$ we have%
\begin{equation*}
\sqrt{\mathcal{A}_{2}^{\alpha }}\lesssim \mathfrak{N}_{\alpha }\left(
T^{\alpha }\right) .
\end{equation*}
\end{lemma}

\begin{proof}
First we give the proof for the case when $T^{\alpha }$ is the $\alpha $%
-fractional Riesz transform $\mathbf{R}^{\alpha }$, whose kernel is $\mathbf{%
K}^{\alpha }\left( x,y\right) =\frac{x-y}{\left\vert x-y\right\vert
^{n+1-\alpha }}$ . Define the $2^{n}$ generalized $n$-ants $\mathcal{Q}_{m}$
for $m\in \left\{ -1,1\right\} ^{n}$, and their translates $\mathcal{Q}%
_{m}\left( w\right) $ for $w\in \mathbb{R}^{n}$ by 
\begin{equation*}
\mathcal{Q}_{m}=\left\{ \left( x_{1},...,x_{n}\right) :m_{k}x_{k}>0\right\}
,\ \ \ \mathcal{Q}_{m}\left( w\right) =\left\{ z:z-w\in \mathcal{Q}%
_{m}\right\} ,\ \ \ \ \ w\in \mathbb{R}^{n}.
\end{equation*}%
Fix $m\in \left\{ -1,1\right\} ^{n}$ and a quasicube $I$. For $a\in \mathbb{R%
}^{n}$ and $r>0$ let 
\begin{equation*}
s_{I}\left( x\right) =\frac{\ell \left( I\right) }{\ell \left( I\right)
+\left\vert x-\zeta _{I}\right\vert },\ \ \ \ \ f_{a,r}\left( y\right) =%
\mathbf{1}_{\mathcal{Q}_{-m}\left( a\right) \cap B\left( 0,r\right) }\left(
y\right) s_{I}\left( y\right) ^{n-\alpha },
\end{equation*}%
where $\zeta _{I}$ is the center of the cube $I$. Now%
\begin{equation*}
\ell \left( I\right) \left\vert x-y\right\vert \leq \ell \left( I\right)
\left\vert x-\zeta _{I}\right\vert +\ell \left( I\right) \left\vert \zeta
_{I}-y\right\vert \leq \left[ \ell \left( I\right) +\left\vert x-\zeta
_{I}\right\vert \right] \left[ \ell \left( I\right) +\left\vert \zeta
_{I}-y\right\vert \right]
\end{equation*}%
implies%
\begin{equation*}
\frac{1}{\left\vert x-y\right\vert }\geq \frac{1}{\ell \left( I\right) }%
s_{I}\left( x\right) s_{I}\left( y\right) ,\ \ \ \ \ x,y\in \mathbb{R}^{n}.
\end{equation*}%
Now the key observation is that with $L\zeta \equiv m\cdot \zeta $, we have%
\begin{equation*}
L\left( x-y\right) =m\cdot \left( x-y\right) \geq \left\vert x-y\right\vert
,\ \ \ \ \ x\in \mathcal{Q}_{m}\left( y\right) ,
\end{equation*}%
which yields%
\begin{equation}
L\left( \mathbf{K}^{\alpha }\left( x,y\right) \right) =\frac{L\left(
x-y\right) }{\left\vert x-y\right\vert ^{n+1-\alpha }}\geq \frac{1}{%
\left\vert x-y\right\vert ^{n-\alpha }}\geq \ell \left( I\right) ^{\alpha
-n}s_{I}\left( x\right) ^{n-\alpha }s_{I}\left( y\right) ^{n-\alpha },
\label{key separation}
\end{equation}%
provided $x\in \mathcal{Q}_{m}\left( y\right) $. Now we note that $x\in 
\mathcal{Q}_{m}\left( y\right) $ when $x\in \mathcal{Q}_{m}\left( a\right) $
and $y\in \mathcal{Q}_{-m}\left( a\right) $ to obtain that for $x\in 
\mathcal{Q}_{m}\left( a\right) $, 
\begin{eqnarray*}
L\left( T^{\alpha }\left( f_{a,r}\sigma \right) \left( x\right) \right)
&=&\int_{\mathcal{Q}_{-m}\left( a\right) \cap B\left( 0,r\right) }\frac{%
L\left( x-y\right) }{\left\vert x-y\right\vert ^{n+1-\alpha }}s_{I}\left(
y\right) d\sigma \left( y\right) \\
&\geq &\ell \left( I\right) ^{\alpha -n}s_{I}\left( x\right) ^{n-\alpha
}\int_{\mathcal{Q}_{-m}\left( a\right) \cap B\left( 0,r\right) }s_{I}\left(
y\right) ^{2n-2\alpha }d\sigma \left( y\right) .
\end{eqnarray*}

Applying $\left\vert L\zeta \right\vert \leq \sqrt{n}\left\vert \zeta
\right\vert $ and our assumed two weight inequality for the fractional Riesz
transform, we see that for $r>0$ large, 
\begin{align*}
& \ell \left( I\right) ^{2\alpha -2n}\int_{\mathcal{Q}_{m}\left( a\right)
}s_{I}\left( x\right) ^{2n-2\alpha }\left( \int_{\mathcal{Q}_{-m}\left(
a\right) \cap B\left( 0,r\right) }s_{I}\left( y\right) ^{2n-2\alpha }d\sigma
\left( y\right) \right) ^{2}d\omega \left( x\right) \\
& \leq \left\Vert LT(\sigma f_{a,r})\right\Vert _{L^{2}(\omega
)}^{2}\lesssim \mathfrak{N}_{\alpha }\left( \mathbf{R}^{\alpha }\right)
^{2}\left\Vert f_{a,r}\right\Vert _{L^{2}(\sigma )}^{2}=\mathfrak{N}_{\alpha
}\left( \mathbf{R}^{\alpha }\right) ^{2}\int_{\mathcal{Q}_{-m}\left(
a\right) \cap B\left( 0,r\right) }s_{I}\left( y\right) ^{2n-2\alpha }d\sigma
\left( y\right) .
\end{align*}

Rearranging the last inequality, and upon letting $r\rightarrow \infty $, we
obtain 
\begin{equation*}
\int_{\mathcal{Q}_{m}\left( a\right) }\frac{\ell \left( I\right) ^{n-\alpha }%
}{\left( \ell \left( I\right) +\left\vert x-\zeta _{I}\right\vert \right)
^{2n-2\alpha }}d\omega \left( x\right) \int_{\mathcal{Q}_{-m}\left( a\right)
}\frac{\ell \left( I\right) ^{n-\alpha }}{\left( \ell \left( I\right)
+\left\vert y-\zeta _{I}\right\vert \right) ^{2n-2\alpha }}d\sigma \left(
y\right) \lesssim \mathfrak{N}_{\alpha }\left( \mathbf{R}^{\alpha }\right)
^{2}.
\end{equation*}%
Note that the ranges of integration above are pairs of opposing $n$-ants.

Fix a quasicube $Q$, which without loss of generality can be taken to be
centered at the origin, $\zeta _{Q}=0$. Then choose $a=\left( 2\ell \left(
Q\right) ,2\ell \left( Q\right) \right) $ and $I=Q$ so that we have%
\begin{eqnarray*}
&&\left( \int_{\mathcal{Q}_{m}\left( a\right) }\frac{\ell \left( Q\right)
^{n-\alpha }}{\left( \ell \left( Q\right) +\left\vert x\right\vert \right)
^{2n-2\alpha }}d\omega \left( x\right) \right) \left( \ell \left( Q\right)
^{\alpha -n}\int_{Q}d\sigma \right) \\
&\leq &C_{\alpha }\int_{\mathcal{Q}_{m}\left( a\right) }\frac{\ell \left(
Q\right) ^{n-\alpha }}{\left( \ell \left( Q\right) +\left\vert x\right\vert
\right) ^{2n-2\alpha }}d\omega \left( x\right) \int_{\mathcal{Q}_{-m}\left(
a\right) }\frac{\ell \left( Q\right) ^{n-\alpha }}{\left( \ell \left(
Q\right) +\left\vert y\right\vert \right) ^{2n-2\alpha }}d\sigma \left(
y\right) \lesssim \mathfrak{N}_{\alpha }\left( \mathbf{R}^{\alpha }\right)
^{2}.
\end{eqnarray*}%
Now fix $m=\left( 1,1,...,1\right) $ and note that there is a fixed $N$
(independent of $\ell \left( Q\right) $) and a fixed collection of rotations 
$\left\{ \rho _{k}\right\} _{k=1}^{N}$, such that the rotates $\rho _{k}%
\mathcal{Q}_{m}\left( a\right) $, $1\leq k\leq N$, of the $n$-ant $\mathcal{Q%
}_{m}\left( a\right) $ cover the complement of the ball $B\left( 0,4\sqrt{n}%
\ell \left( Q\right) \right) $: 
\begin{equation*}
B\left( 0,4\sqrt{n}\ell \left( Q\right) \right) ^{c}\subset
\bigcup_{k=1}^{N}\rho _{k}\mathcal{Q}_{m}\left( a\right) .
\end{equation*}%
Then we obtain, upon applying the same argument to these rotated pairs of $n$%
-ants, 
\begin{equation}
\left( \int_{B\left( 0,4\sqrt{n}\ell \left( Q\right) \right) ^{c}}\frac{\ell
\left( Q\right) ^{n-\alpha }}{\left( \ell \left( Q\right) +\left\vert
x\right\vert \right) ^{2n-2\alpha }}d\omega \left( x\right) \right) \left(
\ell \left( Q\right) ^{\alpha -n}\int_{Q}d\sigma \right) \lesssim \mathfrak{N%
}_{\alpha }\left( \mathbf{R}^{\alpha }\right) ^{2}.  \label{prelim A2}
\end{equation}

Now we assume for the moment the offset $A_{2}^{\alpha }$ condition 
\begin{equation*}
\ell \left( Q\right) ^{2\left( \alpha -n\right) }\left( \int_{Q^{\prime
}}d\omega \right) \left( \int_{Q}d\sigma \right) \leq A_{2}^{\alpha },
\end{equation*}%
where $Q^{\prime }$ and $Q$ are neighbouring quasicubes, i.e. $\left(
Q^{\prime },Q\right) \in \Omega \mathcal{N}^{n}$. If we use this offset
inequality with $Q^{\prime }$ ranging over $3Q\setminus Q$, and then use the
separation of $B\left( 0,4\sqrt{n}\ell \left( Q\right) \right) \setminus 3Q$
and $Q$ to obtain the inequality%
\begin{equation*}
\ell \left( Q\right) ^{2\left( \alpha -n\right) }\left( \int_{B\left( 0,4%
\sqrt{n}\ell \left( Q\right) \right) \setminus 3Q}d\omega \right) \left(
\int_{Q}d\sigma \right) \lesssim A_{2}^{\alpha }\ ,
\end{equation*}%
together with (\ref{prelim A2}), we obtain%
\begin{equation*}
\left( \int_{\mathbb{R}^{n}\setminus Q}\frac{\ell \left( Q\right) ^{n-\alpha
}}{\left( \ell \left( Q\right) +\left\vert x\right\vert \right) ^{2n-2\alpha
}}d\omega \left( x\right) \right) ^{\frac{1}{2}}\left( \ell \left( Q\right)
^{\alpha -n}\int_{Q}d\sigma \right) ^{\frac{1}{2}}\lesssim \mathfrak{N}%
_{\alpha }\left( \mathbf{R}^{\alpha }\right) +\sqrt{A_{2}^{\alpha }}.
\end{equation*}%
Clearly we can reverse the roles of the measures $\omega $ and $\sigma $ and
obtain 
\begin{equation*}
\sqrt{\mathcal{A}_{2}^{\alpha ,\ast }}\lesssim \mathfrak{N}_{\alpha }\left( 
\mathbf{R}^{\alpha }\right) +\sqrt{A_{2}^{\alpha }}
\end{equation*}%
for the kernels $\mathbf{K}^{\alpha }$, $0\leq \alpha <n$.

More generally, to obtain the case when $T^{\alpha }$ is elliptic and the
offset $A_{2}^{\alpha }$ condition holds, we note that the key estimate (\ref%
{key separation}) above extends to the kernel $\sum_{j=1}^{J}\lambda
_{j}^{m}K_{j}^{\alpha }$ of $\sum_{j=1}^{J}\lambda _{j}^{m}T_{j}^{\alpha }$
in (\ref{Ktalpha strong}) if the $n$-ants above are replaced by thin cones
of sufficently small aperture, and there is in addition sufficient
separation between opposing cones, which in turn may require a larger
constant than $4\sqrt{n}$ in the choice of $Q^{\prime }$ above.

Finally, we turn to showing that the offset $A_{2}^{\alpha }$ condition is
implied by the norm inequality, i.e.%
\begin{eqnarray*}
&&\sqrt{A_{2}^{\alpha }}\equiv \sup_{\left( Q^{\prime },Q\right) \in \Omega 
\mathcal{N}^{n}}\ell \left( Q\right) ^{\alpha }\left( \frac{1}{\left\vert
Q^{\prime }\right\vert }\int_{Q^{\prime }}d\omega \right) ^{\frac{1}{2}%
}\left( \frac{1}{\left\vert Q\right\vert }\int_{Q}d\sigma \right) ^{\frac{1}{%
2}}\lesssim \mathfrak{N}_{\alpha }\left( \mathbf{R}^{\alpha }\right) ; \\
&&\text{i.e. }\left( \int_{Q^{\prime }}d\omega \right) \left(
\int_{Q}d\sigma \right) \lesssim \mathfrak{N}_{\alpha }\left( \mathbf{R}%
^{\alpha }\right) ^{2}\left\vert Q\right\vert ^{2-\frac{2\alpha }{n}},\ \ \
\ \ \left( Q^{\prime },Q\right) \in \Omega \mathcal{N}^{n}.
\end{eqnarray*}%
In the range $0\leq \alpha <\frac{n}{2}$ where we only assume (\ref{Ktalpha}%
), we adapt a corresponding argument from \cite{LaSaUr1}.

The `one weight' argument on page 211 of Stein \cite{Ste} yields the \emph{%
asymmetric} two weight $A_{2}^{\alpha }$ condition%
\begin{equation}
\left\vert Q^{\prime }\right\vert _{\omega }\left\vert Q\right\vert _{\sigma
}\leq C\mathfrak{N}_{\alpha }\left( \mathbf{R}^{\alpha }\right) \left\vert
Q\right\vert ^{2\left( 1-\frac{\alpha }{n}\right) },  \label{asym}
\end{equation}%
where $Q$ and $Q^{\prime }$ are quasicubes of equal side length $r$ and
distance $C_{0}r$ apart for some (fixed large) positive constant $C_{0}$
(for this argument we choose the unit vector $\mathbf{u}$ in (\ref{Ktalpha})
to point in the direction from $Q$ to $Q^{\prime }$). In the one weight case
treated in \cite{Ste} it is easy to obtain from this (even for a \emph{single%
} direction $\mathbf{u}$) the usual (symmetric) $A_{2}$ condition. Here we
will have to employ a different approach.

Now recall (see {Sec 2 of} \cite{Saw3} for the case of usual cubes, and the
case of half open, half closed quasicubes here is no different) that given
an open subset $\Phi $ of $\mathbb{R}^{n}$, we can choose $R\geq 3$
sufficiently large, depending only on the dimension, such that if $\left\{
Q_{j}^{k}\right\} _{j}$ are the dyadic quasicubes maximal among those dyadic
quasicubes $Q$ satisfying $RQ\subset \Phi $, then the following properties
hold:%
\begin{equation}
\left\{ 
\begin{array}{ll}
\text{(disjoint cover)} & \Phi =\bigcup_{j}Q_{j}\text{ and }Q_{j}\cap
Q_{i}=\emptyset \text{ if }i\neq j \\ 
\text{(Whitney condition)} & RQ_{j}\subset \Phi \text{ and }3RQ_{j}\cap \Phi
^{c}\neq \emptyset \text{ for all }j \\ 
\text{(finite overlap)} & \sum_{j}\chi _{3Q_{j}}\leq C\chi _{\Phi }%
\end{array}%
\right. .  \label{Whitney}
\end{equation}

So fix a pair of neighbouring quasicubes $\left( Q_{0}^{\prime
},Q_{0}\right) \in \Omega \mathcal{N}^{n}$,\ and let $\left\{ \mathsf{Q}%
_{i}\right\} _{i}$ be a Whitney decomposition into quasicubes of the set $%
\Phi \equiv \left( Q_{0}^{\prime }\times Q_{0}\right) \setminus \mathfrak{D}$
relative to the diagonal $\mathfrak{D}$ in $\mathbb{R}^{n}\times \mathbb{R}%
^{n}$. Of course, there are no common point masses of $\omega $ in $%
Q_{0}^{\prime }$ and $\sigma $ in $Q_{0}$ since the quasicubes $%
Q_{0}^{\prime }$ and $Q_{0}$ are disjoint. Note that if $\mathsf{Q}%
_{i}=Q_{i}^{\prime }\times Q_{i}$, then (\ref{asym}) can be written%
\begin{equation}
\left\vert \mathsf{Q}_{i}\right\vert _{\omega \times \sigma }\leq C\mathfrak{%
N}_{\alpha }\left( \mathbf{R}^{\alpha }\right) \left\vert \mathsf{Q}%
_{i}\right\vert ^{1-\frac{\alpha }{n}},  \label{asym'}
\end{equation}%
where $\omega \times \sigma $ denotes product measure on $\mathbb{R}%
^{n}\times \mathbb{R}^{n}$. We choose $R$ sufficiently large in the Whitney
decomposition (\ref{Whitney}), depending on $C_{0}$, such that (\ref{asym'})
holds for all the Whitney quasicubes $\mathsf{Q}_{i}$. We have $%
\sum_{i}\left\vert \mathsf{Q}_{i}\right\vert =\left\vert Q^{\prime }\times
Q\right\vert =\left\vert Q\right\vert ^{2}$.

Moreover, if $\mathsf{R}=Q^{\prime }\times Q$ is a rectangle in $\mathbb{R}%
^{n}\times \mathbb{R}^{n}$ (i.e. $Q^{\prime },Q$ are quasicubes in $\mathbb{R%
}^{n}$), and if $\mathsf{R}=\overset{\cdot }{\cup }_{i}\mathsf{R}_{i}$ is a
finite disjoint union of rectangles $\mathsf{R}_{\alpha }$, then by
additivity of the product measure $\omega \times \sigma $, 
\begin{equation*}
\left\vert \mathsf{R}\right\vert _{\omega \times \sigma }=\sum_{i}\left\vert 
\mathsf{R}_{i}\right\vert _{\omega \times \sigma }.
\end{equation*}

Let $\mathsf{Q}_{0}=Q_{0}^{\prime }\times Q_{0}$ and set 
\begin{equation*}
\Lambda \equiv \left\{ \mathsf{Q}=Q^{\prime }\times Q:\mathsf{Q}\subset 
\mathsf{Q}_{0},\ell \left( Q\right) =\ell \left( Q^{\prime }\right) \approx
C_{0}^{-1}\limfunc{qdist}\left( Q,Q^{\prime }\right) \text{ and (\ref{asym})
holds}\right\} .
\end{equation*}%
Divide $\mathsf{Q}_{0}$ into $2n\times 2n=4n^{2}$ congruent subquasicubes $%
\mathsf{Q}_{0}^{1},...,\mathsf{Q}_{0}^{4^{n}}$ of side length $\frac{1}{2}$,
and set aside those $\mathsf{Q}_{0}^{j}\in \Lambda $ (those for which (\ref%
{asym}) holds) into a collection of stopping cubes $\Gamma $. Continue to
divide the remaining $\mathsf{Q}_{0}^{j}\in \Lambda $ of side length $\frac{1%
}{4}$, and again, set aside those $\mathsf{Q}_{0}^{j,i}\in \Phi $ into $%
\Gamma $, and continue subdividing those that remain. We continue with such
subdivisions for $N$ generations so that all the cubes \emph{not} set aside
into $\Gamma $ have side length $2^{-N}$ . The important property these
latter cubes have is that they all lie within distance $r2^{-N}$ of the
diagonal $\mathfrak{D}=\left\{ \left( x,x\right) :\left( x,x\right) \in
Q_{0}^{\prime }\times Q_{0}\right\} $ in $\mathsf{Q}_{0}=Q_{0}^{\prime
}\times Q_{0}$ since (\ref{asym}) holds for all pairs of cubes $Q^{\prime }$
and $Q$ of equal side length $r$ having distance at least $C_{0}r$ apart.
Enumerate the cubes in $\Gamma $ as $\left\{ \mathsf{Q}_{i}\right\} _{i}$
and those remaining that are not in $\Gamma $ as $\left\{ \mathsf{P}%
_{j}\right\} _{j}$. Thus we have the pairwise disjoint decomposition%
\begin{equation*}
\mathsf{Q}_{0}=\left( \dbigcup\limits_{i}\mathsf{Q}_{i}\right) \dbigcup
\left( \dbigcup\limits_{j}\mathsf{P}_{j}\right) .
\end{equation*}%
The countable additivity of the product measure $\omega \times \sigma $
shows that%
\begin{equation*}
\left\vert \mathsf{Q}_{0}\right\vert _{\omega \times \sigma
}=\sum_{i}\left\vert \mathsf{Q}_{i}\right\vert _{\omega \times \sigma
}+\sum_{j}\left\vert \mathsf{P}_{j}\right\vert _{\omega \times \sigma }\ .
\end{equation*}

Now we have%
\begin{equation*}
\sum_{i}\left\vert \mathsf{Q}_{i}\right\vert _{\omega \times \sigma
}\lesssim \sum_{i}\mathfrak{N}_{\alpha }\left( \mathbf{R}^{\alpha }\right)
^{2}\left\vert \mathsf{Q}_{i}\right\vert ^{1-\frac{\alpha }{n}},
\end{equation*}%
and 
\begin{eqnarray*}
\sum_{i}\left\vert \mathsf{Q}_{i}\right\vert ^{1-\frac{\alpha }{n}}
&=&\sum_{k\in \mathbb{Z}:\ 2^{k}\leq \ell \left( Q_{0}\right) }\sum_{i:\
\ell \left( Q_{i}\right) =2^{k}}\left( 2^{2nk}\right) ^{1-\frac{\alpha }{n}%
}\approx \sum_{k\in \mathbb{Z}:\ 2^{k}\leq \ell \left( Q_{0}\right) }\left( 
\frac{2^{k}}{\ell \left( Q_{0}\right) }\right) ^{-n}\left( 2^{2nk}\right)
^{1-\frac{\alpha }{n}}\ \ \ \text{(Whitney)} \\
&=&\ell \left( Q_{0}\right) ^{n}\sum_{k\in \mathbb{Z}:\ 2^{k}\leq \ell
\left( Q_{0}\right) }2^{nk\left( -1+2-\frac{2\alpha }{n}\right) }\leq
C_{\alpha }\ell \left( Q_{0}\right) ^{n}\ell \left( Q_{0}\right) ^{n\left( 1-%
\frac{2\alpha }{n}\right) }=C_{\alpha }\left\vert Q_{0}\times
Q_{0}\right\vert ^{2-\frac{2\alpha }{n}}=C_{\alpha }\left\vert \mathsf{Q}%
_{0}\right\vert ^{1-\frac{\alpha }{n}},
\end{eqnarray*}%
provided $0\leq \alpha <\frac{n}{2}$. Using that the side length of $\mathsf{%
P}_{j}=P_{j}\times P_{j}^{\prime }$ is $2^{-N}$ and $dist\left( \mathsf{P}%
_{j},\mathfrak{D}\right) \leq C_{r}2^{-N}$, we have the following limit,%
\begin{equation*}
\sum_{j}\left\vert \mathsf{P}_{j}\right\vert _{\omega \times \sigma
}=\left\vert \dbigcup\limits_{j}\mathsf{P}_{j}\right\vert _{\omega \times
\sigma }\rightarrow 0\text{ as }N\rightarrow \infty ,
\end{equation*}%
since $\dbigcup\limits_{j}\mathsf{P}_{j}$ shrinks to the empty set as $%
N\rightarrow \infty $, and since locally finite measures such as $\omega
\times \sigma $ are regular in Euclidean space. This completes the proof
that $\sqrt{A_{2}^{\alpha }}\lesssim \mathfrak{N}_{\alpha }\left( \mathbf{R}%
^{\alpha }\right) $ for the range $0\leq \alpha <\frac{n}{2}$.

Now we turn to proving $\sqrt{A_{2}^{\alpha }}\lesssim \mathfrak{N}_{\alpha
}\left( \mathbf{R}^{\alpha }\right) $ for the range $\frac{n}{2}\leq \alpha
<n$, where we assume the stronger ellipticity condition (\ref{Ktalpha strong}%
). So fix a pair of neighbouring quasicubes $\left( K^{\prime },K\right) \in
\Omega \mathcal{N}^{n}$, and assume that $\sigma +\omega $ doesn't charge
the intersection $\overline{K^{\prime }}\cap \overline{K}$ of the closures
of $K^{\prime }$ and $K$. It will be convenient to replace $n$ by $n+1$, i.e
to introduce an additional dimension, and work with the preimages $Q^{\prime
}=\Omega ^{-1}K^{\prime }$ and $Q=\Omega ^{-1}K$ that are usual cubes, and
with the corresponding pullbacks $\widetilde{\omega }=m_{1}\times \Omega
^{\ast }\omega $ and $\widetilde{\sigma }=m_{1}\times \Omega ^{\ast }\sigma $
of the measures $\omega $ and $\sigma $ where $m_{1}$ is Lebesgue measure on
the line. We may also assume that 
\begin{equation*}
Q^{\prime }=\left[ -1,0\right) \times \dprod\limits_{i=1}^{n}Q_{i},\ \ \ \ \
Q=\left[ 0,1\right) \times \dprod\limits_{i=1}^{n}Q_{i}.
\end{equation*}%
where $Q_{i}=\left[ a_{i},b_{i}\right] $ for $1\leq i\leq n$ (since the
other cases are handled in similar fashion). It is important to note that we
are considering the intervals $Q_{i}$ here to be closed, and we will track
this difference as we proceed.

Choose $\theta _{1}\in \left[ a_{1},b_{1}\right] $ so that both 
\begin{equation*}
\left\vert \left[ -1,0\right) \times \left[ a_{1},\theta _{1}\right] \times
\dprod\limits_{i=2}^{n}Q_{i}\right\vert _{\widetilde{\omega }},\ \ \
\left\vert \left[ -1,0\right) \times \left[ \theta _{1},b_{1}\right] \times
\dprod\limits_{i=2}^{n}Q_{i}\right\vert _{\widetilde{\omega }}\geq \frac{1}{2%
}\left\vert Q^{\prime }\right\vert _{\widetilde{\omega }}.
\end{equation*}%
Now denote the two intervals $\left[ a_{1},\theta _{1}\right] $ and $\left[
\theta _{1},b_{1}\right] $ by $\left[ a_{1}^{\ast },b_{1}^{\ast }\right] $
and $\left[ a_{1}^{\ast \ast },b_{1}^{\ast \ast }\right] $ where the order
is chosen so that 
\begin{equation*}
\left\vert \left[ 0,1\right) \times \left[ a_{1}^{\ast },b_{1}^{\ast }\right]
\times \dprod\limits_{i=2}^{n}Q_{i}\right\vert _{\widetilde{\sigma }}\leq
\left\vert \left[ 0,1\right) \times \left[ a_{1}^{\ast \ast },b_{1}^{\ast
\ast }\right] \times \dprod\limits_{i=2}^{n}Q_{i}\right\vert _{\widetilde{%
\sigma }}.
\end{equation*}%
Then we have both%
\begin{equation*}
\left\vert \left[ -1,0\right) \times \left[ a_{1}^{\ast },b_{1}^{\ast }%
\right] \times \dprod\limits_{i=2}^{n}Q_{i}\right\vert _{\widetilde{\omega }%
}\geq \frac{1}{2}\left\vert Q\right\vert _{\widetilde{\omega }}\text{ and }%
\left\vert \left[ 0,1\right) \times \left[ a_{1}^{\ast \ast },b_{1}^{\ast
\ast }\right] \times \dprod\limits_{i=2}^{n}Q_{i}\right\vert _{\widetilde{%
\sigma }}\geq \frac{1}{2}\left\vert Q\right\vert _{\widetilde{\sigma }}\ .
\end{equation*}%
Now choose $\theta _{2}\in \left[ a_{2},b_{2}\right] $ so that both%
\begin{equation*}
\left\vert \left[ -1,0\right) \times \left[ a_{1}^{\ast },b_{1}^{\ast }%
\right] \times \left[ a_{2},\theta _{2}\right] \times
\dprod\limits_{i=3}^{n}Q_{i}\right\vert _{\widetilde{\omega }},\ \ \
\left\vert \left[ -1,0\right) \times \left[ a_{1}^{\ast },b_{1}^{\ast }%
\right] \times \left[ \theta _{2},b_{2}\right] \times
\dprod\limits_{i=3}^{n}Q_{i}\right\vert _{\widetilde{\omega }}\geq \frac{1}{4%
}\left\vert Q\right\vert _{\widetilde{\omega }},
\end{equation*}%
and denote the two intervals $\left[ a_{2},\theta _{2}\right] $ and $\left[
\theta _{2},b_{2}\right] $ by $\left[ a_{2}^{\ast },b_{2}^{\ast }\right] $
and $\left[ a_{2}^{\ast \ast },b_{2}^{\ast \ast }\right] $ where the order
is chosen so that%
\begin{equation*}
\left[ 0,1\right) \times \left\vert \left[ a_{1}^{\ast \ast },b_{1}^{\ast
\ast }\right] \times \left[ a_{2}^{\ast },b_{2}^{\ast }\right] \times
\dprod\limits_{i=2}^{n}Q_{i}\right\vert _{\widetilde{\sigma }}\leq
\left\vert \left[ 0,1\right) \times \left[ a_{1}^{\ast \ast },b_{1}^{\ast
\ast }\right] \times \left[ a_{2}^{\ast \ast },b_{2}^{\ast \ast }\right]
\times \dprod\limits_{i=2}^{n}Q_{i}\right\vert _{\widetilde{\sigma }}.
\end{equation*}%
Then we have both%
\begin{eqnarray*}
\left\vert \left[ -1,0\right) \times \left[ a_{1}^{\ast },b_{1}^{\ast }%
\right] \times \left[ a_{2}^{\ast },b_{2}^{\ast }\right] \times
\dprod\limits_{i=3}^{n}Q_{i}\right\vert _{\widetilde{\omega }} &\geq &\frac{1%
}{4}\left\vert Q\right\vert _{\widetilde{\omega }}\ , \\
\left\vert \left[ 0,1\right) \times \left[ a_{1}^{\ast \ast },b_{1}^{\ast
\ast }\right] \times \left[ a_{2}^{\ast \ast },b_{2}^{\ast \ast }\right]
\times \dprod\limits_{i=3}^{n}Q_{i}\right\vert _{\widetilde{\sigma }} &\geq &%
\frac{1}{4}\left\vert Q\right\vert _{\widetilde{\sigma }}\ ,
\end{eqnarray*}%
and continuing in this way we end up with two rectangles,%
\begin{eqnarray*}
G &\equiv &\left[ -1,0\right) \times \left[ a_{1}^{\ast },b_{1}^{\ast }%
\right] \times \left[ a_{2}^{\ast },b_{2}^{\ast }\right] \times ...\left[
a_{n}^{\ast },b_{n}^{\ast }\right] , \\
H &\equiv &\left[ 0,1\right) \times \left[ a_{1}^{\ast \ast },b_{1}^{\ast
\ast }\right] \times \left[ a_{2}^{\ast \ast },b_{2}^{\ast \ast }\right]
\times ...\left[ a_{n}^{\ast \ast },b_{n}^{\ast \ast }\right] ,
\end{eqnarray*}%
that satisfy%
\begin{eqnarray*}
\left\vert G\right\vert _{\widetilde{\omega }} &=&\left\vert \left[
-1,0\right) \times \left[ a_{1}^{\ast },b_{1}^{\ast }\right] \times \left[
a_{2}^{\ast },b_{2}^{\ast }\right] \times ...\left[ a_{n}^{\ast
},b_{n}^{\ast }\right] \right\vert _{\widetilde{\omega }}\geq \frac{1}{2^{n}}%
\left\vert Q\right\vert _{\widetilde{\omega }}, \\
\left\vert H\right\vert _{\widetilde{\sigma }} &=&\left\vert \left[
0,1\right) \times \left[ a_{1}^{\ast \ast },b_{1}^{\ast \ast }\right] \times %
\left[ a_{2}^{\ast \ast },b_{2}^{\ast \ast }\right] \times ...\left[
a_{n}^{\ast \ast },b_{n}^{\ast \ast }\right] \right\vert _{\widetilde{\sigma 
}}\geq \frac{1}{2^{n}}\left\vert Q\right\vert _{\widetilde{\sigma }}.
\end{eqnarray*}

However, the quasirectangles $\Omega G$ and $\Omega H$ lie in opposing quasi-%
$n$-ants at the vertex $\Omega \theta =\Omega \left( \theta _{1},\theta
_{2},...,\theta _{n}\right) $, and so we can apply (\ref{Ktalpha strong}) to
obtain that for $x\in \Omega G$,%
\begin{equation*}
\left\vert \sum_{j=1}^{J}\lambda _{j}^{m}T_{j}^{\alpha }\left( \mathbf{1}%
_{\Omega H}\sigma \right) \left( x\right) \right\vert =\left\vert
\int_{\Omega H}\sum_{j=1}^{J}\lambda _{j}^{m}K_{j}^{\alpha }\left(
x,y\right) d\sigma \left( y\right) \right\vert \gtrsim \int_{\Omega
H}\left\vert x-y\right\vert ^{\alpha -n}d\sigma \left( y\right) \gtrsim
\left\vert \Omega Q\right\vert ^{\frac{\alpha }{n}-1}\left\vert \Omega
H\right\vert _{\sigma }.
\end{equation*}%
For the inequality above, we need to know that the distinguished point $%
\Omega \theta $ is not a common point mass of $\sigma $ and $\omega $, but
this follows from our assumption that $\sigma +\omega $ doesn't charge the
intersection $\overline{K^{\prime }}\cap \overline{K}$ of the closures of $%
K^{\prime }$ and $K$. Then from the norm inequality we get%
\begin{eqnarray*}
\left\vert \Omega G\right\vert _{\omega }\left( \left\vert \Omega
Q\right\vert ^{\frac{\alpha }{n}-1}\left\vert \Omega H\right\vert _{\sigma
}\right) ^{2} &\lesssim &\int_{G}\left\vert \sum_{j=1}^{J}\lambda
_{j}^{m}T_{j}^{\alpha }\left( \mathbf{1}_{\Omega H}\sigma \right)
\right\vert ^{2}d\omega \\
&\lesssim &\mathfrak{N}_{\sum_{j=1}^{J}\lambda _{j}^{m}T_{j}^{\alpha
}}^{2}\int \mathbf{1}_{\Omega H}^{2}d\sigma =\mathfrak{N}_{\sum_{j=1}^{J}%
\lambda _{j}^{m}T_{j}^{\alpha }}^{2}\left\vert \Omega H\right\vert _{\sigma
},
\end{eqnarray*}%
from which we deduce that%
\begin{eqnarray*}
\left\vert \Omega Q\right\vert ^{2\left( \frac{\alpha }{n}-1\right)
}\left\vert \Omega Q^{\prime }\right\vert _{\omega }\left\vert \Omega
Q\right\vert _{\sigma } &\lesssim &2^{2n}\left\vert \Omega Q\right\vert
^{2\left( \frac{\alpha }{n}-1\right) }\left\vert \Omega G\right\vert
_{\omega }\left\vert \Omega H\right\vert _{\sigma }\lesssim 2^{2n}\mathfrak{N%
}_{\sum_{j=1}^{J}\lambda _{j}^{m}T_{j}^{\alpha }}^{2}; \\
\left\vert K\right\vert ^{2\left( \frac{\alpha }{n}-1\right) }\left\vert
K^{\prime }\right\vert _{\omega }\left\vert K\right\vert _{\sigma }
&\lesssim &2^{2n}\mathfrak{N}_{\sum_{j=1}^{J}\lambda _{j}^{m}T_{j}^{\alpha
}}^{2}\ ,
\end{eqnarray*}%
and hence%
\begin{equation*}
A_{2}^{\alpha }\lesssim 2^{2n}\mathfrak{N}_{\sum_{j=1}^{J}\lambda
_{j}^{m}T_{j}^{\alpha }}^{2}\ .
\end{equation*}

Thus we have obtained the offset $A_{2}^{\alpha }$ condition for pairs $%
\left( K^{\prime },K\right) \in \Omega \mathcal{N}^{n}$ such that $\sigma
+\omega $ doesn't charge the intersection $\overline{K^{\prime }}\cap 
\overline{K}$ of the closures of $K^{\prime }$ and $K$. From this and the
argument at the beginning of this proof, we obtain the one-tailed $\mathcal{A%
}_{2}^{\alpha }$ conditions. Indeed, we note that $\left\vert \partial
\left( rQ\right) \right\vert _{\sigma +\omega }>0$ for only a countable
number of dilates $r>1$, and so a limiting argument applies. This completes
the proof of Lemma \ref{necc frac A2}.
\end{proof}

\section{Monotonicity Lemma and Energy lemma}

The Monotonicity Lemma below will be used to prove the Energy Lemma, which
is then used in several places in the proof of Theorem \ref{T1 theorem}. The
formulation of the Monotonicity Lemma with $m=2$ for cubes is due to M.
Lacey and B. Wick \cite{LaWi}, and corrects that used in early versions of
our paper \cite{SaShUr5}.

\subsection{The Monotonicity Lemma}

For $0\leq \alpha <n$ and $m\in \mathbb{R}_{+}$, we recall the $m$-weighted
fractional Poisson integral%
\begin{equation*}
\mathrm{P}_{m}^{\alpha }\left( J,\mu \right) \equiv \int_{\mathbb{R}^{n}}%
\frac{\left\vert J\right\vert ^{\frac{m}{n}}}{\left( \left\vert J\right\vert
^{\frac{1}{n}}+\left\vert y-c_{J}\right\vert \right) ^{n+m-\alpha }}d\mu
\left( y\right) ,
\end{equation*}%
where $\mathrm{P}_{1}^{\alpha }\left( J,\mu \right) =\mathrm{P}^{\alpha
}\left( J,\mu \right) $ is the standard Poisson integral. The next lemma
holds for quasicubes and common point masses with the same proof as in \cite%
{SaShUr7}.

\begin{lemma}[Monotonicity]
\label{mono}Suppose that$\ I$ and $J$ are quasicubes in $\mathbb{R}^{n}$
such that $J\subset 2J\subset I$, and that $\mu $ is a signed measure on $%
\mathbb{R}^{n}$ supported outside $I$. Finally suppose that $T^{\alpha }$ is
a standard $\alpha $-fractional singular integral on $\mathbb{R}^{n}$ with $%
0<\alpha <n$. Then we have the estimate%
\begin{equation}
\left\Vert \bigtriangleup _{J}^{\omega }T^{\alpha }\mu \right\Vert
_{L^{2}\left( \omega \right) }\lesssim \Phi ^{\alpha }\left( J,\left\vert
\mu \right\vert \right) ,  \label{estimate}
\end{equation}%
where for a positive measure $\nu $,%
\begin{eqnarray*}
\Phi ^{\alpha }\left( J,\nu \right) ^{2} &\equiv &\left( \frac{\mathrm{P}%
^{\alpha }\left( J,\nu \right) }{\left\vert J\right\vert ^{\frac{1}{n}}}%
\right) ^{2}\left\Vert \bigtriangleup _{J}^{\omega }\mathbf{x}\right\Vert
_{L^{2}\left( \omega \right) }^{2}+\left( \frac{\mathrm{P}_{1+\delta
}^{\alpha }\left( J,\nu \right) }{\left\vert J\right\vert ^{\frac{1}{n}}}%
\right) ^{2}\left\Vert \mathbf{x}-\mathbf{m}_{J}\right\Vert _{L^{2}\left( 
\mathbf{1}_{J}\omega \right) }^{2}\ , \\
\mathbf{m}_{J} &\equiv &\mathbb{E}_{J}^{\omega }\mathbf{x}=\frac{1}{%
\left\vert J\right\vert _{\omega }}\int_{J}\mathbf{x}d\omega .
\end{eqnarray*}
\end{lemma}

\subsection{The Energy Lemma}

Suppose now we are given a subset $\mathcal{H}$ of the dyadic quasigrid $%
\Omega \mathcal{D}^{\omega }$. Let $\mathsf{P}_{\mathcal{H}}^{\omega
}=\sum_{J\in \mathcal{H}}\bigtriangleup _{J}^{\omega }$ be the corresponding 
$\omega $-quasiHaar projection. We define $\mathcal{H}^{\ast }\equiv
\dbigcup\limits_{J\in \mathcal{H}}\left\{ J^{\prime }\in \Omega \mathcal{D}%
^{\omega }:J^{\prime }\subset J\right\} $. The next lemma also holds for
quasicubes and common point masses with the same proof as in \cite{SaShUr7}.

\begin{lemma}[\textbf{Energy Lemma}]
\label{ener}Let $J\ $be a quasicube in $\Omega \mathcal{D}^{\omega }$. Let $%
\Psi _{J}$ be an $L^{2}\left( \omega \right) $ function supported in $J$ and
with $\omega $-integral zero, and denote its quasiHaar support by $\mathcal{H%
}=\limfunc{supp}\widehat{\Psi _{J}}$. Let $\nu $ be a positive measure
supported in $\mathbb{R}^{n}\setminus \gamma J$ with $\gamma \geq 2$, and
for each $J^{\prime }\in \mathcal{H}$, let $d\nu _{J^{\prime }}=\varphi
_{J^{\prime }}d\nu $ with $\left\vert \varphi _{J^{\prime }}\right\vert \leq
1$. Let $T^{\alpha }$ be a standard $\alpha $-fractional singular integral
operator with $0\leq \alpha <n$. Then with $\delta ^{\prime }=\frac{\delta }{%
2}$ we have%
\begin{eqnarray*}
\left\vert \sum_{J^{\prime }\in \mathcal{H}}\left\langle T^{\alpha }\left(
\nu _{J^{\prime }}\right) ,\bigtriangleup _{J^{\prime }}^{\omega }\Psi
_{J}\right\rangle _{\omega }\right\vert &\lesssim &\left\Vert \Psi
_{J}\right\Vert _{L^{2}\left( \omega \right) }\left( \frac{\mathrm{P}%
^{\alpha }\left( J,\nu \right) }{\left\vert J\right\vert ^{\frac{1}{n}}}%
\right) \left\Vert \mathsf{P}_{\mathcal{H}}^{\omega }\mathbf{x}\right\Vert
_{L^{2}\left( \omega \right) } \\
&&+\left\Vert \Psi _{J}\right\Vert _{L^{2}\left( \omega \right) }\frac{1}{%
\gamma ^{\delta ^{\prime }}}\left( \frac{\mathrm{P}_{1+\delta ^{\prime
}}^{\alpha }\left( J,\nu \right) }{\left\vert J\right\vert ^{\frac{1}{n}}}%
\right) \left\Vert \mathsf{P}_{\mathcal{H}^{\ast }}^{\omega }\mathbf{x}%
\right\Vert _{L^{2}\left( \omega \right) } \\
&&\lesssim \left\Vert \Psi _{J}\right\Vert _{L^{2}\left( \omega \right)
}\left( \frac{\mathrm{P}^{\alpha }\left( J,\nu \right) }{\left\vert
J\right\vert ^{\frac{1}{n}}}\right) \left\Vert \mathsf{P}_{\mathcal{H}^{\ast
}}^{\omega }\mathbf{x}\right\Vert _{L^{2}\left( \omega \right) },
\end{eqnarray*}%
and in particular the `pivotal' bound%
\begin{equation*}
\left\vert \left\langle T^{\alpha }\left( \nu \right) ,\Psi
_{J}\right\rangle _{\omega }\right\vert \leq C\left\Vert \Psi
_{J}\right\Vert _{L^{2}\left( \omega \right) }\mathrm{P}^{\alpha }\left(
J,\left\vert \nu \right\vert \right) \sqrt{\left\vert J\right\vert _{\omega }%
}\ .
\end{equation*}
\end{lemma}

\begin{remark}
The first term on the right side of the energy inequality above is the `big'
Poisson integral $\mathrm{P}^{\alpha }$ times the `small' energy term $%
\left\Vert \mathsf{P}_{\mathcal{H}}^{\omega }\mathbf{x}\right\Vert
_{L^{2}\left( \omega \right) }^{2}$ that is additive in $\mathcal{H}$, while
the second term on the right is the `small' Poisson integral $\mathrm{P}%
_{1+\delta ^{\prime }}^{\alpha }$ times the `big' energy term $\left\Vert 
\mathsf{P}_{\mathcal{H}^{\ast }}^{\omega }\mathbf{x}\right\Vert
_{L^{2}\left( \omega \right) }$ that is no longer additive in $\mathcal{H}$.
The first term presents no problems in subsequent analysis due solely to the
additivity of\ the `small' energy term. It is the second term that must be
handled by special methods. For example, in the Intertwining Proposition
below, the interaction of the singular integral occurs with a pair of
quasicubes $J\subset I$ at \emph{highly separated} levels, where the
goodness of $J$ can exploit the decay $\delta ^{\prime }$ in the kernel of
the `small' Poisson integral $\mathrm{P}_{1+\delta ^{\prime }}^{\alpha }$
relative to the `big' Poisson integral $\mathrm{P}^{\alpha }$, and results
in a bound directly by the quasienergy condition. On the other hand, in the
local recursion of M. Lacey at the end of the \ paper, the separation of
levels in the pairs $J\subset I$ can be as \emph{little} as a fixed
parameter $\mathbf{\rho }$, and here we must first separate the stopping
form into two sublinear forms that involve the two estimates respectively.
The form corresponding to the smaller Poisson integral $\mathrm{P}_{1+\delta
^{\prime }}^{\alpha }$ is again handled using goodness and the decay $\delta
^{\prime }$ in the kernel, while the form corresponding to the larger
Poisson integral $\mathrm{P}^{\alpha }$ requires the stopping time and
recursion argument of M. Lacey.
\end{remark}

\section{Preliminaries of NTV type}

An important reduction of our theorem is delivered by the following two
lemmas, that in the case of one dimension are due to Nazarov, Treil and
Volberg (see \cite{NTV3} and \cite{Vol}). The proofs given there do not
extend in standard ways to higher dimensions with common point masses, and
we use the quasiweak boundedness property to handle the case of touching
quasicubes, and an application of Schur's Lemma to handle the case of
separated quasicubes. The first lemma below is Lemmas 8.1 and 8.7 in \cite%
{LaWi} but with the larger constant $\mathcal{A}_{2}^{\alpha }$ there in
place of the smaller constant $A_{2}^{\alpha }$ here. We emphasize that only
the offset $A_{2}^{\alpha }$ condition is needed with testing and weak
boundedness in these preliminary estimates.

\begin{lemma}
\label{standard delta}Suppose $T^{\alpha }$ is a standard fractional
singular integral with $0\leq \alpha <n$, and that all of the quasicubes $%
I\in \Omega \mathcal{D}^{\sigma },J\in \Omega \mathcal{D}^{\omega }$ below
are good with goodness parameters $\varepsilon $ and $\mathbf{r}$. Fix a
positive integer $\mathbf{\rho }>\mathbf{r}$. For $f\in L^{2}\left( \sigma
\right) $ and $g\in L^{2}\left( \omega \right) $ we have%
\begin{equation}
\sum_{\substack{ \left( I,J\right) \in \Omega \mathcal{D}^{\sigma }\times
\Omega \mathcal{D}^{\omega }  \\ 2^{-\mathbf{\rho }}\ell \left( I\right)
\leq \ell \left( J\right) \leq 2^{\mathbf{\rho }}\ell \left( I\right) }}%
\left\vert \left\langle T_{\sigma }^{\alpha }\left( \bigtriangleup
_{I}^{\sigma }f\right) ,\bigtriangleup _{J}^{\omega }g\right\rangle _{\omega
}\right\vert \lesssim \left( \mathfrak{T}_{\alpha }+\mathfrak{T}_{\alpha
}^{\ast }+\mathcal{WBP}_{T^{\alpha }}+\sqrt{A_{2}^{\alpha }}\right)
\left\Vert f\right\Vert _{L^{2}\left( \sigma \right) }\left\Vert
g\right\Vert _{L^{2}\left( \omega \right) }  \label{delta near}
\end{equation}%
and 
\begin{equation}
\sum_{\substack{ \left( I,J\right) \in \Omega \mathcal{D}^{\sigma }\times
\Omega \mathcal{D}^{\omega }  \\ I\cap J=\emptyset \text{ and }\frac{\ell
\left( J\right) }{\ell \left( I\right) }\notin \left[ 2^{-\mathbf{\rho }},2^{%
\mathbf{\rho }}\right] }}\left\vert \left\langle T_{\sigma }^{\alpha }\left(
\bigtriangleup _{I}^{\sigma }f\right) ,\bigtriangleup _{J}^{\omega
}g\right\rangle _{\omega }\right\vert \lesssim \sqrt{A_{2}^{\alpha }}%
\left\Vert f\right\Vert _{L^{2}\left( \sigma \right) }\left\Vert
g\right\Vert _{L^{2}\left( \omega \right) }.  \label{delta far}
\end{equation}
\end{lemma}

\begin{lemma}
\label{standard indicator}Suppose $T^{\alpha }$ is a standard fractional
singular integral with $0\leq \alpha <n$, that all of the quasicubes $I\in
\Omega \mathcal{D}^{\sigma },J\in \Omega \mathcal{D}^{\omega }$ below are
good, that $\mathbf{\rho }>\mathbf{r}$, that $f\in L^{2}\left( \sigma
\right) $ and $g\in L^{2}\left( \omega \right) $, that $\mathcal{F}\subset
\Omega \mathcal{D}^{\sigma }$ and $\mathcal{G}\subset \Omega \mathcal{D}%
^{\omega }$ are $\sigma $-Carleson and $\omega $-Carleson collections
respectively, i.e.,%
\begin{equation*}
\sum_{F^{\prime }\in \mathcal{F}:\ F^{\prime }\subset F}\left\vert F^{\prime
}\right\vert _{\sigma }\lesssim \left\vert F\right\vert _{\sigma },\ \ \ \ \
F\in \mathcal{F},\text{ and }\sum_{G^{\prime }\in \mathcal{G}:\ G^{\prime
}\subset G}\left\vert G^{\prime }\right\vert _{\omega }\lesssim \left\vert
G\right\vert _{\omega },\ \ \ \ \ G\in \mathcal{G},
\end{equation*}%
that there are numerical sequences $\left\{ \alpha _{\mathcal{F}}\left(
F\right) \right\} _{F\in \mathcal{F}}$ and $\left\{ \beta _{\mathcal{G}%
}\left( G\right) \right\} _{G\in \mathcal{G}}$ such that%
\begin{equation}
\sum_{F\in \mathcal{F}}\alpha _{\mathcal{F}}\left( F\right) ^{2}\left\vert
F\right\vert _{\sigma }\leq \left\Vert f\right\Vert _{L^{2}\left( \sigma
\right) }^{2}\text{ and }\sum_{G\in \mathcal{G}}\beta _{\mathcal{G}}\left(
G\right) ^{2}\left\vert G\right\vert _{\sigma }\leq \left\Vert g\right\Vert
_{L^{2}\left( \sigma \right) }^{2}\ ,  \label{qo}
\end{equation}%
and finally that for each pair of quasicubes $\left( I,J\right) \in \Omega 
\mathcal{D}^{\sigma }\times \Omega \mathcal{D}^{\omega }$, there are bounded
functions $\beta _{I,J}$ and $\gamma _{I,J}$ supported in $I\setminus 2J$
and $J\setminus 2I$ respectively, satisfying%
\begin{equation*}
\left\Vert \beta _{I,J}\right\Vert _{\infty },\left\Vert \gamma
_{I,J}\right\Vert _{\infty }\leq 1.
\end{equation*}%
Then%
\begin{eqnarray}
&&\sum_{\substack{ \left( F,J\right) \in \mathcal{F}\times \Omega \mathcal{D}%
^{\omega }  \\ F\cap J=\emptyset \text{ and }\ell \left( J\right) \leq 2^{-%
\mathbf{\rho }}\ell \left( F\right) }}\left\vert \left\langle T_{\sigma
}^{\alpha }\left( \beta _{F,J}\mathbf{1}_{F}\alpha _{\mathcal{F}}\left(
F\right) \right) ,\bigtriangleup _{J}^{\omega }g\right\rangle _{\omega
}\right\vert +\sum_{\substack{ \left( I,G\right) \in \Omega \mathcal{D}%
^{\sigma }\times \mathcal{G}  \\ I\cap G=\emptyset \text{ and }\ell \left(
I\right) \leq 2^{-\mathbf{\rho }}\ell \left( G\right) }}\left\vert
\left\langle T_{\sigma }^{\alpha }\left( \bigtriangleup _{I}^{\sigma
}f\right) ,\gamma _{I,G}\mathbf{1}_{G}\beta _{\mathcal{G}}\left( G\right)
\right\rangle _{\omega }\right\vert  \label{indicator far} \\
&&\ \ \ \ \ \ \ \ \ \ \ \ \ \ \ \ \ \ \ \ \ \ \ \ \ \ \ \ \ \ \ \ \ \ \ \ \
\ \ \ \lesssim \sqrt{A_{2}^{\alpha }}\left\Vert f\right\Vert _{L^{2}\left(
\sigma \right) }\left\Vert g\right\Vert _{L^{2}\left( \omega \right) }. 
\notag
\end{eqnarray}
\end{lemma}

See \cite{SaShUr6} for complete details of the proofs when common point
masses are permitted.

\begin{remark}
If $\mathcal{F}$ and $\mathcal{G}$ are $\sigma $-Carleson and $\omega $%
-Carleson collections respectively, and if $\alpha _{\mathcal{F}}\left(
F\right) =\mathbb{E}_{F}^{\sigma }\left\vert f\right\vert $ and $\beta _{%
\mathcal{G}}\left( G\right) =\mathbb{E}_{G}^{\omega }\left\vert g\right\vert 
$, then the `quasi' orthogonality condition (\ref{qo}) holds (here `quasi'
has a different meaning than quasi), and this special case of Lemma \ref%
{standard indicator} serves as a basic example.
\end{remark}

\begin{remark}
Lemmas \ref{standard delta} and \ref{standard indicator} differ mainly in
that an orthogonal collection of quasiHaar projections is replaced by a
`quasi' orthogonal collection of indicators $\left\{ \mathbf{1}_{F}\alpha _{%
\mathcal{F}}\left( F\right) \right\} _{F\in \mathcal{F}}$. More precisely,
the main difference between (\ref{delta far}) and (\ref{indicator far}) is
that a quasiHaar projection $\bigtriangleup _{I}^{\sigma }f$ or $%
\bigtriangleup _{J}^{\omega }g$ has been replaced with a constant multiple
of an indicator $\mathbf{1}_{F}\alpha _{\mathcal{F}}\left( F\right) $ or $%
\mathbf{1}_{G}\beta _{\mathcal{G}}\left( G\right) $, and in addition, a
bounded function is permitted to multiply the indicator of the quasicube
having larger sidelength.
\end{remark}

\section{Corona Decompositions and splittings}

We will use two different corona constructions, namely a Calder\'{o}%
n-Zygmund decomposition and an energy decomposition of NTV type, to reduce
matters to the stopping form, the main part of which is handled by Lacey's
recursion argument. We will then iterate these coronas into a double corona.
We first recall our basic setup. For convenience in notation we will
sometimes suppress the dependence on $\alpha $ in our nonlinear forms, but
will retain it in the operators, Poisson integrals and constants. We will
assume that the good/bad quasicube machinery of Nazarov, Treil and Volberg 
\cite{Vol} is in force here as in \cite{SaShUr7}. Let $\Omega \mathcal{D}%
^{\sigma }=\Omega \mathcal{D}^{\omega }$ be an $\left( \mathbf{r}%
,\varepsilon \right) $-good quasigrid on $\mathbb{R}^{n}$, and let $\left\{
h_{I}^{\sigma ,a}\right\} _{I\in \Omega \mathcal{D}^{\sigma },\ a\in \Gamma
_{n}}$ and $\left\{ h_{J}^{\omega ,b}\right\} _{J\in \Omega \mathcal{D}%
^{\omega },\ b\in \Gamma _{n}}$ be corresponding quasiHaar bases as
described above, so that%
\begin{equation*}
f=\sum_{I\in \Omega \mathcal{D}^{\sigma }}\bigtriangleup _{I}^{\sigma }f%
\text{ and }g=\sum_{J\in \Omega \mathcal{D}^{\omega }\text{ }}\bigtriangleup
_{J}^{\omega }g\ ,
\end{equation*}%
where the quasiHaar projections $\bigtriangleup _{I}^{\sigma }f$ and $%
\bigtriangleup _{J}^{\omega }g$ vanish if the quasicubes $I$ and $J$ are not
good. Recall that we must show the bilinear inequality (\ref{star}), i.e. $%
\left\vert \mathcal{T}^{\alpha }\left( f,g\right) \right\vert \leq \mathfrak{%
N}_{T^{\alpha }}\left\Vert f\right\Vert _{L^{2}\left( \sigma \right)
}\left\Vert g\right\Vert _{L^{2}\left( \omega \right) }$.

We now proceed for the remainder of this section to follow the development
in \cite{SaShUr7}, pointing out just the highlights, and referring to \cite%
{SaShUr7} for proofs, when no changes are required by the inclusion of
quasicubes and common point masses.

\subsection{The Calder\'{o}n-Zygmund corona}

We now introduce a stopping tree $\mathcal{F}$ for the function $f\in
L^{2}\left( \sigma \right) $. Let $\mathcal{F}$ be a collection of Calder%
\'{o}n-Zygmund stopping quasicubes for $f$, and let $\Omega \mathcal{D}%
^{\sigma }=\dbigcup\limits_{F\in \mathcal{F}}\mathcal{C}_{F}$ be the
associated corona decomposition of the dyadic quasigrid $\Omega \mathcal{D}%
^{\sigma }$. See below and also \cite{SaShUr7} for the standard definitions
of corona, etc.

For a quasicube $I\in \Omega \mathcal{D}^{\sigma }$ let $\pi _{\Omega 
\mathcal{D}^{\sigma }}I$ be the $\Omega \mathcal{D}^{\sigma }$-parent of $I$
in the quasigrid $\Omega \mathcal{D}^{\sigma }$, and let $\pi _{\mathcal{F}%
}I $ be the smallest member of $\mathcal{F}$ that contains $I$. For $%
F,F^{\prime }\in \mathcal{F}$, we say that $F^{\prime }$ is an $\mathcal{F}$%
-child of $F$ if $\pi _{\mathcal{F}}\left( \pi _{\Omega \mathcal{D}^{\sigma
}}F^{\prime }\right) =F$ (it could be that $F=\pi _{\Omega \mathcal{D}%
^{\sigma }}F^{\prime }$), and we denote by $\mathfrak{C}_{\mathcal{F}}\left(
F\right) $ the set of $\mathcal{F}$-children of $F$. For $F\in \mathcal{F}$,
define the projection $\mathsf{P}_{\mathcal{C}_{F}}^{\sigma }$ onto the
linear span of the quasiHaar functions $\left\{ h_{I}^{\sigma ,a}\right\}
_{I\in \mathcal{C}_{F},\ a\in \Gamma _{n}}$ by%
\begin{equation*}
\mathsf{P}_{\mathcal{C}_{F}}^{\sigma }f=\sum_{I\in \mathcal{C}%
_{F}}\bigtriangleup _{I}^{\sigma }f=\sum_{I\in \mathcal{C}_{F},\ a\in \Gamma
_{n}}\left\langle f,h_{I}^{\sigma ,a}\right\rangle _{\sigma }h_{I}^{\sigma
,a}.
\end{equation*}%
The standard properties of these projections are%
\begin{equation*}
f=\sum_{F\in \mathcal{F}}\mathsf{P}_{\mathcal{C}_{F}}^{\sigma }f,\ \ \ \ \
\int \left( \mathsf{P}_{\mathcal{C}_{F}}^{\sigma }f\right) \sigma =0,\ \ \ \
\ \left\Vert f\right\Vert _{L^{2}\left( \sigma \right) }^{2}=\sum_{F\in 
\mathcal{F}}\left\Vert \mathsf{P}_{\mathcal{C}_{F}}^{\sigma }f\right\Vert
_{L^{2}\left( \sigma \right) }^{2}.
\end{equation*}

\subsection{The energy corona}

We also impose a quasienergy corona decomposition as in \cite{NTV3} and \cite%
{LaSaUr2}.

\begin{definition}
\label{def energy corona 3}Given a quasicube $S_{0}$, define $\mathcal{S}%
\left( S_{0}\right) $ to be the maximal subquasicubes $I\subset S_{0}$ such
that%
\begin{equation}
\sum_{J\in \mathcal{M}_{\mathbf{\tau }-\limfunc{deep}}\left( I\right)
}\left( \frac{\mathrm{P}^{\alpha }\left( J,\mathbf{1}_{S_{0}\setminus \gamma
J}\sigma \right) }{\left\vert J\right\vert ^{\frac{1}{n}}}\right)
^{2}\left\Vert \mathsf{P}_{J}^{\limfunc{subgood},\omega }\mathbf{x}%
\right\Vert _{L^{2}\left( \omega \right) }^{2}\geq C_{\limfunc{energy}}\left[
\left( \mathcal{E}_{\alpha }^{\limfunc{strong}}\right) ^{2}+A_{2}^{\alpha
}+A_{2}^{\alpha ,\limfunc{punct}}\right] \ \left\vert I\right\vert _{\sigma
},  \label{def stop 3}
\end{equation}%
where $\mathcal{E}_{\alpha }^{\limfunc{strong}}$ is the constant in the
strong quasienergy condition defined in Definition \ref{def strong
quasienergy}, and $C_{\limfunc{energy}}$ is a sufficiently large positive
constant depending only on $\mathbf{\tau }\geq \mathbf{r},n$ and $\alpha $.
Then define the $\sigma $-energy stopping quasicubes of $S_{0}$ to be the
collection 
\begin{equation*}
\mathcal{S}=\left\{ S_{0}\right\} \cup \dbigcup\limits_{n=0}^{\infty }%
\mathcal{S}_{n}
\end{equation*}%
where $\mathcal{S}_{0}=\mathcal{S}\left( S_{0}\right) $ and $\mathcal{S}%
_{n+1}=\dbigcup\limits_{S\in \mathcal{S}_{n}}\mathcal{S}\left( S\right) $
for $n\geq 0$.
\end{definition}

From the quasienergy condition in Definition \ref{def strong quasienergy} we
obtain the $\sigma $-Carleson estimate%
\begin{equation}
\sum_{S\in \mathcal{S}:\ S\subset I}\left\vert S\right\vert _{\sigma }\leq
2\left\vert I\right\vert _{\sigma },\ \ \ \ \ I\in \Omega \mathcal{D}%
^{\sigma }.  \label{sigma Carleson 3}
\end{equation}

Finally, we record the reason for introducing quasienergy stopping times. If 
\begin{equation}
X_{\alpha }\left( \mathcal{C}_{S}\right) ^{2}\equiv \sup_{I\in \mathcal{C}%
_{S}}\frac{1}{\left\vert I\right\vert _{\sigma }}\sum_{J\in \mathcal{M}_{%
\mathbf{\tau }-\limfunc{deep}}\left( I\right) }\left( \frac{\mathrm{P}%
^{\alpha }\left( J,\mathbf{1}_{S\setminus \gamma J}\sigma \right) }{%
\left\vert J\right\vert ^{\frac{1}{n}}}\right) ^{2}\left\Vert \mathsf{P}%
_{J}^{\limfunc{subgood},\omega }\mathbf{x}\right\Vert _{L^{2}\left( \omega
\right) }^{2}  \label{def stopping energy 3}
\end{equation}%
is (the square of) the $\alpha $\emph{-stopping quasienergy} of the weight
pair $\left( \sigma ,\omega \right) $ with respect to the corona $\mathcal{C}%
_{S}$, then we have the \emph{stopping quasienergy bounds}%
\begin{equation}
X_{\alpha }\left( \mathcal{C}_{S}\right) \leq \sqrt{C_{\limfunc{energy}}}%
\sqrt{\left( \mathcal{E}_{\alpha }^{\limfunc{strong}}\right)
^{2}+A_{2}^{\alpha }+A_{2}^{\alpha ,\limfunc{punct}}},\ \ \ \ \ S\in 
\mathcal{S},  \label{def stopping bounds 3}
\end{equation}%
where $A_{2}^{\alpha }+A_{2}^{\alpha ,\limfunc{punct}}$ and the the strong
quasienergy constant $\mathcal{E}_{\alpha }^{\limfunc{strong}}$ are
controlled by assumption.

\subsection{General stopping data}

It is useful to extend our notion of corona decomposition to more general
stopping data. Our general definition of stopping data will use a positive
constant $C_{0}\geq 4$.

\begin{definition}
\label{general stopping data}Suppose we are given a positive constant $%
C_{0}\geq 4$, a subset $\mathcal{F}$ of the dyadic quasigrid $\Omega 
\mathcal{D}^{\sigma }$ (called the stopping times), and a corresponding
sequence $\alpha _{\mathcal{F}}\equiv \left\{ \alpha _{\mathcal{F}}\left(
F\right) \right\} _{F\in \mathcal{F}}$ of nonnegative numbers $\alpha _{%
\mathcal{F}}\left( F\right) \geq 0$ (called the stopping data). Let $\left( 
\mathcal{F},\prec ,\pi _{\mathcal{F}}\right) $ be the tree structure on $%
\mathcal{F}$ inherited from $\Omega \mathcal{D}^{\sigma }$, and for each $%
F\in \mathcal{F}$ denote by $\mathcal{C}_{F}=\left\{ I\in \Omega \mathcal{D}%
^{\sigma }:\pi _{\mathcal{F}}I=F\right\} $ the corona associated with $F$: 
\begin{equation*}
\mathcal{C}_{F}=\left\{ I\in \Omega \mathcal{D}^{\sigma }:I\subset F\text{
and }I\not\subset F^{\prime }\text{ for any }F^{\prime }\prec F\right\} .
\end{equation*}%
We say the triple $\left( C_{0},\mathcal{F},\alpha _{\mathcal{F}}\right) $
constitutes \emph{stopping data} for a function $f\in L_{loc}^{1}\left(
\sigma \right) $ if

\begin{enumerate}
\item $\mathbb{E}_{I}^{\sigma }\left\vert f\right\vert \leq \alpha _{%
\mathcal{F}}\left( F\right) $ for all $I\in \mathcal{C}_{F}$ and $F\in 
\mathcal{F}$,

\item $\sum_{F^{\prime }\preceq F}\left\vert F^{\prime }\right\vert _{\sigma
}\leq C_{0}\left\vert F\right\vert _{\sigma }$ for all $F\in \mathcal{F}$,

\item $\sum_{F\in \mathcal{F}}\alpha _{\mathcal{F}}\left( F\right)
^{2}\left\vert F\right\vert _{\sigma }\mathbf{\leq }C_{0}^{2}\left\Vert
f\right\Vert _{L^{2}\left( \sigma \right) }^{2}$,

\item $\alpha _{\mathcal{F}}\left( F\right) \leq \alpha _{\mathcal{F}}\left(
F^{\prime }\right) $ whenever $F^{\prime },F\in \mathcal{F}$ with $F^{\prime
}\subset F$.
\end{enumerate}
\end{definition}

\begin{definition}
If $\left( C_{0},\mathcal{F},\alpha _{\mathcal{F}}\right) $ constitutes
(general) \emph{stopping data} for a function $f\in L_{loc}^{1}\left( \sigma
\right) $, we refer to the othogonal decomposition%
\begin{equation*}
f=\sum_{F\in \mathcal{F}}\mathsf{P}_{\mathcal{C}_{F}}^{\sigma }f;\ \ \ \ \ 
\mathsf{P}_{\mathcal{C}_{F}}^{\sigma }f\equiv \sum_{I\in \mathcal{C}%
_{F}}\bigtriangleup _{I}^{\sigma }f,
\end{equation*}%
as the (general) \emph{corona decomposition} of $f$ associated with the
stopping times $\mathcal{F}$.
\end{definition}

Property (1) says that $\alpha _{\mathcal{F}}\left( F\right) $ bounds the
quasiaverages of $f$ in the corona $\mathcal{C}_{F}$, and property (2) says
that the quasicubes at the tops of the coronas satisfy a Carleson condition
relative to the weight $\sigma $. Note that a standard `maximal quasicube'
argument extends the Carleson condition in property (2) to the inequality%
\begin{equation}
\sum_{F^{\prime }\in \mathcal{F}:\ F^{\prime }\subset A}\left\vert F^{\prime
}\right\vert _{\sigma }\leq C_{0}\left\vert A\right\vert _{\sigma }\text{
for all open sets }A\subset \mathbb{R}^{n}.  \label{Car ext}
\end{equation}%
Property (3) is the `quasi' orthogonality condition that says the sequence
of functions $\left\{ \alpha _{\mathcal{F}}\left( F\right) \mathbf{1}%
_{F}\right\} _{F\in \mathcal{F}}$ is in the vector-valued space $L^{2}\left(
\ell ^{2};\sigma \right) $, and property (4) says that the control on
stopping data is nondecreasing on the stopping tree $\mathcal{F}$. We
emphasize that we are \emph{not} assuming in this definition the stronger
property that there is $C>1$ such that $\alpha _{\mathcal{F}}\left(
F^{\prime }\right) >C\alpha _{\mathcal{F}}\left( F\right) $ whenever $%
F^{\prime },F\in \mathcal{F}$ with $F^{\prime }\subsetneqq F$. Instead, the
properties (2) and (3) substitute for this lack. Of course the stronger
property \emph{does} hold for the familiar \emph{Calder\'{o}n-Zygmund}
stopping data determined by the following requirements for $C>1$,%
\begin{equation*}
\mathbb{E}_{F^{\prime }}^{\sigma }\left\vert f\right\vert >C\mathbb{E}%
_{F}^{\sigma }\left\vert f\right\vert \text{ whenever }F^{\prime },F\in 
\mathcal{F}\text{ with }F^{\prime }\subsetneqq F,\ \ \ \ \ \mathbb{E}%
_{I}^{\sigma }\left\vert f\right\vert \leq C\mathbb{E}_{F}^{\sigma
}\left\vert f\right\vert \text{ for }I\in \mathcal{C}_{F},
\end{equation*}%
which are themselves sufficiently strong to automatically force properties
(2) and (3) with $\alpha _{\mathcal{F}}\left( F\right) =\mathbb{E}%
_{F}^{\sigma }\left\vert f\right\vert $.

We have the following useful consequence of (2) and (3) that says the
sequence $\left\{ \alpha _{\mathcal{F}}\left( F\right) \mathbf{1}%
_{F}\right\} _{F\in \mathcal{F}}$ has a \emph{`quasi' orthogonal} property
relative to $f$ with a constant $C_{0}^{\prime }$ depending only on $C_{0}$:%
\begin{equation}
\left\Vert \sum_{F\in \mathcal{F}}\alpha _{\mathcal{F}}\left( F\right) 
\mathbf{1}_{F}\right\Vert _{L^{2}\left( \sigma \right) }^{2}\leq
C_{0}^{\prime }\left\Vert f\right\Vert _{L^{2}\left( \sigma \right) }^{2}.
\label{q orth}
\end{equation}%
$\ $

We will use a construction that permits \emph{iteration} of general corona
decompositions.

\begin{lemma}
\label{iterating coronas}Suppose that $\left( C_{0},\mathcal{F},\alpha _{%
\mathcal{F}}\right) $ constitutes \emph{stopping data} for a function $f\in
L_{loc}^{1}\left( \sigma \right) $, and that for each $F\in \mathcal{F}$, $%
\left( C_{0},\mathcal{K}\left( F\right) ,\alpha _{\mathcal{K}\left( F\right)
}\right) $ constitutes \emph{stopping data} for the corona projection $%
\mathsf{P}_{\mathcal{C}_{F}}^{\sigma }f$, where in addition $F\in \mathcal{K}%
\left( F\right) $. There is a positive constant $C_{1}$, depending only on $%
C_{0}$, such that if%
\begin{eqnarray*}
\mathcal{K}^{\ast }\left( F\right) &\equiv &\left\{ K\in \mathcal{K}\left(
F\right) \cap \mathcal{C}_{F}:\alpha _{\mathcal{K}\left( F\right) }\left(
K\right) \geq \alpha _{\mathcal{F}}\left( F\right) \right\} \\
\mathcal{K} &\equiv &\mathop{\displaystyle \bigcup }\limits_{F\in \mathcal{F}%
}\mathcal{K}^{\ast }\left( F\right) \cup \left\{ F\right\} , \\
\alpha _{\mathcal{K}}\left( K\right) &\equiv &%
\begin{array}{ccc}
\alpha _{\mathcal{K}\left( F\right) }\left( K\right) & \text{ for } & K\in 
\mathcal{K}^{\ast }\left( F\right) \setminus \left\{ F\right\} \\ 
\max \left\{ \alpha _{\mathcal{F}}\left( F\right) ,\alpha _{\mathcal{K}%
\left( F\right) }\left( F\right) \right\} & \text{ for } & K=F%
\end{array}%
,\ \ \ \ \ \text{for }F\in \mathcal{F},
\end{eqnarray*}%
the triple $\left( C_{1},\mathcal{K},\alpha _{\mathcal{K}}\right) $
constitutes \emph{stopping data} for $f$. We refer to the collection of
quasicubes $\mathcal{K}$ as the \emph{iterated} stopping times, and to the
orthogonal decomposition $f=\sum_{K\in \mathcal{K}}P_{\mathcal{C}_{K}^{%
\mathcal{K}}}f$ as the \emph{iterated} corona decomposition of $f$, where 
\begin{equation*}
\mathcal{C}_{K}^{\mathcal{K}}\equiv \left\{ I\in \Omega \mathcal{D}:I\subset
K\text{ and }I\not\subset K^{\prime }\text{ for }K^{\prime }\prec _{\mathcal{%
K}}K\right\} .
\end{equation*}
\end{lemma}

Note that in our definition of $\left( C_{1},\mathcal{K},\alpha _{\mathcal{K}%
}\right) $ we have `discarded' from $\mathcal{K}\left( F\right) $ all of
those $K\in \mathcal{K}\left( F\right) $ that are not in the corona $%
\mathcal{C}_{F}$, and also all of those $K\in \mathcal{K}\left( F\right) $
for which $\alpha _{\mathcal{K}\left( F\right) }\left( K\right) $ is
strictly less than $\alpha _{\mathcal{F}}\left( F\right) $. Then the union
over $F$ of what remains is our new collection of stopping times. We then
define stopping data $\alpha _{\mathcal{K}}\left( K\right) $ according to
whether or not $K\in \mathcal{F}$: if $K\notin \mathcal{F}$ but $K\in 
\mathcal{C}_{F}$ then $\alpha _{\mathcal{K}}\left( K\right) $ equals $\alpha
_{\mathcal{K}\left( F\right) }\left( K\right) $, while if $K\in \mathcal{F}$%
, then $\alpha _{\mathcal{K}}\left( K\right) $ is the larger of $\alpha _{%
\mathcal{K}\left( F\right) }\left( F\right) $ and $\alpha _{\mathcal{F}%
}\left( K\right) $. See \cite{SaShUr7} for a proof.

\subsection{Doubly iterated coronas and the NTV quasicube size splitting}

Let 
\begin{equation*}
\mathcal{NTV}_{\alpha }\equiv \sqrt{\mathcal{A}_{2}^{\alpha }+\mathcal{A}%
_{2}^{\alpha ,\ast }+A_{2}^{\alpha ,\limfunc{punct}}+A_{2}^{\alpha ,\ast ,%
\limfunc{punct}}}+\mathfrak{T}_{T^{\alpha }}+\mathfrak{T}_{T^{\alpha
}}^{\ast }\ .
\end{equation*}%
Here is a brief schematic diagram of the decompositions, with bounds usingin 
$\fbox{}$: 
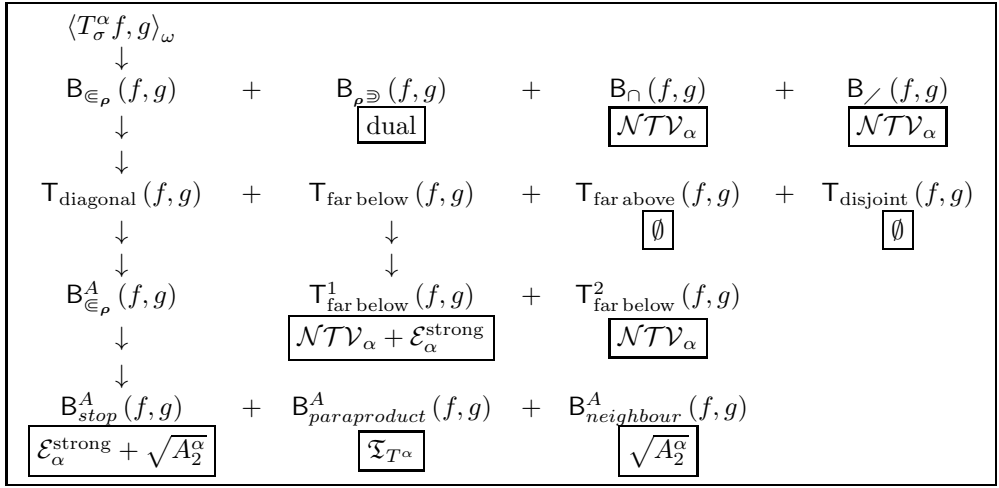
\begin{equation*}
\fbox{$%
\begin{array}{ccccccc}
\left\langle T_{\sigma }^{\alpha }f,g\right\rangle _{\omega } &  &  &  &  & 
&  \\ 
\downarrow &  &  &  &  &  &  \\ 
\mathsf{B}_{\Subset _{\mathbf{\rho }}}\left( f,g\right) & + & \mathsf{B}_{_{%
\mathbf{\rho }}\Supset }\left( f,g\right) & + & \mathsf{B}_{\cap }\left(
f,g\right) & + & \mathsf{B}_{\diagup }\left( f,g\right) \\ 
\downarrow &  & \fbox{dual} &  & \fbox{$\mathcal{NTV}_{\alpha }$} &  & \fbox{%
$\mathcal{NTV}_{\alpha }$} \\ 
\downarrow &  &  &  &  &  &  \\ 
\mathsf{T}_{\limfunc{diagonal}}\left( f,g\right) & + & \mathsf{T}_{\limfunc{%
far}\limfunc{below}}\left( f,g\right) & + & \mathsf{T}_{\limfunc{far}%
\limfunc{above}}\left( f,g\right) & + & \mathsf{T}_{\limfunc{disjoint}%
}\left( f,g\right) \\ 
\downarrow &  & \downarrow &  & \fbox{$\emptyset $} &  & \fbox{$\emptyset $}
\\ 
\downarrow &  & \downarrow &  &  &  &  \\ 
\mathsf{B}_{\Subset _{\mathbf{\rho }}}^{A}\left( f,g\right) &  & \mathsf{T}_{%
\limfunc{far}\limfunc{below}}^{1}\left( f,g\right) & + & \mathsf{T}_{%
\limfunc{far}\limfunc{below}}^{2}\left( f,g\right) &  &  \\ 
\downarrow &  & \fbox{$\mathcal{NTV}_{\alpha }+\mathcal{E}_{\alpha }^{%
\limfunc{strong}}$} &  & \fbox{$\mathcal{NTV}_{\alpha }$} &  &  \\ 
\downarrow &  &  &  &  &  &  \\ 
\mathsf{B}_{stop}^{A}\left( f,g\right) & + & \mathsf{B}_{paraproduct}^{A}%
\left( f,g\right) & + & \mathsf{B}_{neighbour}^{A}\left( f,g\right) &  &  \\ 
\fbox{$\mathcal{E}_{\alpha }^{\limfunc{strong}}+\sqrt{A_{2}^{\alpha }}$} & 
& \fbox{$\mathfrak{T}_{T^{\alpha }}$} &  & \fbox{$\sqrt{A_{2}^{\alpha }}$} & 
& 
\end{array}%
$}
\end{equation*}

We begin with the NTV \emph{quasicube size splitting} of the inner product $%
\left\langle T_{\sigma }^{\alpha }f,g\right\rangle _{\omega }$ - and later
apply the iterated corona construction - that splits the pairs of quasicubes 
$\left( I,J\right) $ in a simultaneous quasiHaar decomposition of $f$ and $g$
into four groups, namely those pairs that:

\begin{enumerate}
\item are below the size diagonal and $\mathbf{\rho }$-deeply embedded,

\item are above the size diagonal and $\mathbf{\rho }$-deeply embedded,

\item are disjoint, and

\item are of $\mathbf{\rho }$-comparable size.
\end{enumerate}

More precisely we have%
\begin{eqnarray*}
\left\langle T_{\sigma }^{\alpha }f,g\right\rangle _{\omega }
&=&\dsum\limits_{I\in \Omega \mathcal{D}^{\sigma },\ J\in \Omega \mathcal{D}%
^{\omega }}\left\langle T_{\sigma }^{\alpha }\left( \bigtriangleup
_{I}^{\sigma }f\right) ,\left( \bigtriangleup _{I}^{\omega }g\right)
\right\rangle _{\omega } \\
&=&\dsum\limits_{\substack{ I\in \Omega \mathcal{D}^{\sigma },\ J\in \Omega 
\mathcal{D}^{\omega }  \\ J\Subset _{\mathbf{\rho }}I}}\left\langle
T_{\sigma }^{\alpha }\left( \bigtriangleup _{I}^{\sigma }f\right) ,\left(
\bigtriangleup _{J}^{\omega }g\right) \right\rangle _{\omega }+\dsum\limits 
_{\substack{ I\in \Omega \mathcal{D}^{\sigma },\ J\in \Omega \mathcal{D}%
^{\omega }  \\ J_{\mathbf{\rho }}\Supset I}}\left\langle T_{\sigma }^{\alpha
}\left( \bigtriangleup _{I}^{\sigma }f\right) ,\left( \bigtriangleup
_{J}^{\omega }g\right) \right\rangle _{\omega } \\
&&+\dsum\limits_{\substack{ I\in \Omega \mathcal{D}^{\sigma },\ J\in \Omega 
\mathcal{D}^{\omega }  \\ J\cap I=\emptyset }}\left\langle T_{\sigma
}^{\alpha }\left( \bigtriangleup _{I}^{\sigma }f\right) ,\left(
\bigtriangleup _{J}^{\omega }g\right) \right\rangle _{\omega }+\dsum\limits 
_{\substack{ I\in \Omega \mathcal{D}^{\sigma },\ J\in \Omega \mathcal{D}%
^{\omega }  \\ 2^{-\mathbf{\rho }}\leq \ell \left( J\right) \diagup \ell
\left( I\right) \leq 2^{\mathbf{\rho }}}}\left\langle T_{\sigma }^{\alpha
}\left( \bigtriangleup _{I}^{\sigma }f\right) ,\left( \bigtriangleup
_{J}^{\omega }g\right) \right\rangle _{\omega } \\
&=&\mathsf{B}_{\Subset _{\mathbf{\rho }}}\left( f,g\right) +\mathsf{B}_{_{%
\mathbf{\rho }}\Supset }\left( f,g\right) +\mathsf{B}_{\cap }\left(
f,g\right) +\mathsf{B}_{\diagup }\left( f,g\right) .
\end{eqnarray*}%
Lemma \ref{standard delta} in the section on NTV preliminaries show that the 
\emph{disjoint} and \emph{comparable} forms $\mathsf{B}_{\cap }\left(
f,g\right) $ and $\mathsf{B}_{\diagup }\left( f,g\right) $ are both bounded
by the $\mathcal{A}_{2}^{\alpha }+A_{2}^{\alpha ,\limfunc{punct}}$,
quasitesting and quasiweak boundedness property constants. The \emph{below}
and \emph{above} forms are clearly symmetric, so we need only consider the
form $\mathsf{B}_{\Subset _{\mathbf{\rho }}}\left( f,g\right) $, to which we
turn for the remainder of the proof. For this we need functional energy.

\begin{definition}
\label{functional energy n}Let $\mathfrak{F}_{\alpha }$ be the smallest
constant in the `\textbf{f}unctional quasienergy' inequality below, holding
for all $h\in L^{2}\left( \sigma \right) $ and all $\sigma $-Carleson
collections $\mathcal{F}$ with Carleson norm $C_{\mathcal{F}}$ bounded by a
fixed constant $C$: 
\begin{equation}
\sum_{F\in \mathcal{F}}\sum_{J\in \mathcal{M}_{\left( \mathbf{r},\varepsilon
\right) -\limfunc{deep}}\left( F\right) }\left( \frac{\mathrm{P}^{\alpha
}\left( J,h\sigma \right) }{\left\vert J\right\vert ^{\frac{1}{n}}}\right)
^{2}\left\Vert \mathsf{P}_{\mathcal{C}_{F}^{\limfunc{good},\mathbf{\tau }-%
\limfunc{shift}};J}^{\omega }\mathbf{x}\right\Vert _{L^{2}\left( \omega
\right) }^{2}\leq \mathfrak{F}_{\alpha }\lVert h\rVert _{L^{2}\left( \sigma
\right) }\,.  \label{e.funcEnergy n}
\end{equation}
\end{definition}

The main ingredient used in reducing control of the below form $\mathsf{B}%
_{\Subset _{\mathbf{\rho }}}\left( f,g\right) $ to control of the functional
energy $\mathcal{F}_{\alpha }$ constant and the stopping form $\mathsf{B}%
_{stop}^{A}\left( f,g\right) $, is an adaptation of the Intertwining
Proposition from \cite{SaShUr7} to include quasicubes and common point
masses. This adaptation is easy because the measures $\omega $ and $\sigma $
only `see each other' in the proof through the energy Muckenhoupt conditions 
$A_{2}^{\alpha ,\limfunc{energy}}$ and $A_{2}^{\alpha ,\ast ,\limfunc{energy}%
}$, and the straightforward details can be found in \cite{SaShUr6}. We now
turn to controlling functional energy.

\section{Control of functional energy by energy modulo $\mathcal{A}_{2}^{%
\protect\alpha }$ and $A_{2}^{\protect\alpha ,\limfunc{punct}}\label{equiv}$}

Now we arrive at one of our main propositions in the proof of our theorem.
We show that the functional quasienergy constants $\mathfrak{F}_{\alpha }$
as in (\ref{e.funcEnergy n}) are controlled by $\mathcal{A}_{2}^{\alpha }$, $%
A_{2}^{\alpha ,\limfunc{punct}}$ and both the \emph{strong} quasienergy
constant $\mathcal{E}_{\alpha }^{\limfunc{strong}}$\ defined in Definition %
\ref{def strong quasienergy}. The proof of this fact is further complicated
when common point masses are permitted, accounting for the inclusion of the
punctured Muckenhoupt condition $A_{2}^{\alpha ,\limfunc{punct}}$. But apart
from this difference, the proof here is essentially the same as that in \cite%
{SaShUr7}, where common point masses were prohibited. As a consequence we
will refer to \cite{SaShUr7} in many of the places where the arguments are
unchanged. A complete and detailed proof can of course be found in \cite%
{SaShUr6}.

\begin{proposition}
\label{func ener control}We have%
\begin{equation*}
\mathfrak{F}_{\alpha }\lesssim \mathcal{E}_{\alpha }^{\limfunc{strong}}+%
\sqrt{\mathcal{A}_{2}^{\alpha }}+\sqrt{\mathcal{A}_{2}^{\alpha ,\ast }}+%
\sqrt{A_{2}^{\alpha ,\limfunc{punct}}}\text{ and }\mathfrak{F}_{\alpha
}^{\ast }\lesssim \mathcal{E}_{\alpha }^{\limfunc{strong},\ast }+\sqrt{%
\mathcal{A}_{2}^{\alpha }}+\sqrt{\mathcal{A}_{2}^{\alpha ,\ast }}+\sqrt{%
A_{2}^{\alpha ,\ast ,\limfunc{punct}}}\ .
\end{equation*}
\end{proposition}

To prove this proposition, we fix $\mathcal{F}$ as in (\ref{e.funcEnergy n}%
), and set 
\begin{equation}
\mu \equiv \sum_{F\in \mathcal{F}}\sum_{J\in \mathcal{M}_{\left( \mathbf{r}%
,\varepsilon \right) -\limfunc{deep}}\left( F\right) }\left\Vert \mathsf{P}%
_{F,J}^{\omega }\mathbf{x}\right\Vert _{L^{2}\left( \omega \right)
}^{2}\cdot \delta _{\left( c_{J},\ell \left( J\right) \right) }\text{ and }d%
\overline{\mu }\left( x,t\right) \equiv \frac{1}{t^{2}}d\mu \left(
x,t\right) \ ,  \label{def mu n}
\end{equation}%
where $\mathcal{M}_{\left( \mathbf{r},\varepsilon \right) -\limfunc{deep}%
}\left( F\right) $ consists of the maximal $\mathbf{r}$-deeply embedded
subquasicubes of $F$, and where $\delta _{\left( c_{J},\ell \left( J\right)
\right) }$ denotes the Dirac unit mass at the point $\left( c_{J},\ell
\left( J\right) \right) $ in the upper half-space $\mathbb{R}_{+}^{n+1}$.
Here $J$ is a dyadic quasicube with center $c_{J}$ and side length $\ell
\left( J\right) $. For convenience in notation, we denote for any dyadic
quasicube $J$ the localized projection $\mathsf{P}_{\mathcal{C}_{F}^{%
\limfunc{good},\mathbf{\tau }-\limfunc{shift}};J}^{\omega }$ given by%
\begin{equation*}
\mathsf{P}_{F,J}^{\omega }\equiv \mathsf{P}_{\mathcal{C}_{F}^{\limfunc{good},%
\mathbf{\tau }-\limfunc{shift}};J}^{\omega }=\sum_{J^{\prime }\subset J:\
J^{\prime }\in \mathcal{C}_{F}^{\limfunc{good},\mathbf{\tau }-\limfunc{shift}%
}}\bigtriangleup _{J^{\prime }}^{\omega }.
\end{equation*}%
We emphasize that the quasicubes $J\in \mathcal{M}_{\left( \mathbf{r}%
,\varepsilon \right) -\limfunc{deep}}\left( F\right) $ are not necessarily
good, but that the subquasicubes $J^{\prime }\subset J$ arising in the
projection $\mathsf{P}_{F,J}^{\omega }$ are good. We can replace $\mathbf{x}$
by $\mathbf{x}-\mathbf{c}$ inside the projection for any choice of $\mathbf{c%
}$ we wish; the projection is unchanged. More generally, $\delta _{q}$
denotes a Dirac unit mass at a point $q$ in the upper half-space $\mathbb{R}%
_{+}^{n+1}$.

We prove the two-weight inequality 
\begin{equation}
\left\Vert \mathbb{P}^{\alpha }\left( f\sigma \right) \right\Vert _{L^{2}(%
\mathbb{R}_{+}^{n+1},\overline{\mu })}\lesssim \left( \mathcal{E}_{\alpha }^{%
\limfunc{strong}}+\sqrt{\mathcal{A}_{2}^{\alpha }}+\sqrt{\mathcal{A}%
_{2}^{\alpha ,\ast }}+\sqrt{A_{2}^{\alpha ,\limfunc{punct}}}\right) \lVert
f\rVert _{L^{2}\left( \sigma \right) }\,,  \label{two weight Poisson n}
\end{equation}%
for all nonnegative $f$ in $L^{2}\left( \sigma \right) $, noting that $%
\mathcal{F}$ and $f$ are \emph{not} related here. Above, $\mathbb{P}^{\alpha
}(\cdot )$ denotes the $\alpha $-fractional Poisson extension to the upper
half-space $\mathbb{R}_{+}^{n+1}$,

\begin{equation*}
\mathbb{P}^{\alpha }\nu \left( x,t\right) \equiv \int_{\mathbb{R}^{n}}\frac{t%
}{\left( t^{2}+\left\vert x-y\right\vert ^{2}\right) ^{\frac{n+1-\alpha }{2}}%
}d\nu \left( y\right) ,
\end{equation*}%
so that in particular 
\begin{equation*}
\left\Vert \mathbb{P}^{\alpha }(f\sigma )\right\Vert _{L^{2}(\mathbb{R}%
_{+}^{n+1},\overline{\mu })}^{2}=\sum_{F\in \mathcal{F}}\sum_{J\in \mathcal{M%
}_{\mathbf{r}-\limfunc{deep}}\left( F\right) }\mathbb{P}^{\alpha }\left(
f\sigma \right) (c(J),\ell \left( J\right) )^{2}\left\Vert \mathsf{P}%
_{F,J}^{\omega }\frac{x}{\left\vert J\right\vert ^{\frac{1}{n}}}\right\Vert
_{L^{2}\left( \omega \right) }^{2}\,,
\end{equation*}%
and so (\ref{two weight Poisson n}) proves the first line in Proposition \ref%
{func ener control} upon inspecting (\ref{e.funcEnergy n}). Note also that
we can equivalently write $\left\Vert \mathbb{P}^{\alpha }\left( f\sigma
\right) \right\Vert _{L^{2}(\mathbb{R}_{+}^{n+1},\overline{\mu }%
)}=\left\Vert \widetilde{\mathbb{P}}^{\alpha }\left( f\sigma \right)
\right\Vert _{L^{2}(\mathbb{R}_{+}^{n+1},\mu )}$ where $\widetilde{\mathbb{P}%
}^{\alpha }\nu \left( x,t\right) \equiv \frac{1}{t}\mathbb{P}^{\alpha }\nu
\left( x,t\right) $ is the renormalized Poisson operator. Here we have
simply shifted the factor $\frac{1}{t^{2}}$ in $\overline{\mu }$ to $%
\left\vert \widetilde{\mathbb{P}}^{\alpha }\left( f\sigma \right)
\right\vert ^{2}$ instead, and we will do this shifting often throughout the
proof when it is convenient to do so.

The characterization of the two-weight inequality for fractional and Poisson
integrals in \cite{Saw3} was stated in terms of the collection $\mathcal{P}%
^{n}$ of cubes in $\mathbb{R}^{n}$ with sides parallel to the coordinate
axes. It is a routine matter to pullback the Poisson inequality under a
globally biLipschitz map $\Omega :\mathbb{R}^{n}\rightarrow \mathbb{R}^{n}$,
then apply the theorem in \cite{Saw3} (as a black box), and then to
pushforward the conclusions of the theorems so as to extend these
characterizations of fractional and Poisson integral inequalities to the
setting of quasicubes $Q\in \Omega \mathcal{P}^{n}$ and quasitents $Q\times %
\left[ 0,\ell \left( Q\right) \right] \subset \mathbb{R}_{+}^{n+1}$ with $%
Q\in \Omega \mathcal{P}^{n}$. Using this extended theorem for the two-weight
Poisson inequality, we see that inequality (\ref{two weight Poisson n})
requires checking these two inequalities for dyadic quasicubes $I\in \Omega 
\mathcal{D}$ and quasiboxes $\widehat{I}=I\times \left[ 0,\ell \left(
I\right) \right) $ in the upper half-space$\mathbb{R}_{+}^{n+1}$: 
\begin{equation}
\int_{\mathbb{R}_{+}^{n+1}}\mathbb{P}^{\alpha }\left( \mathbf{1}_{I}\sigma
\right) \left( x,t\right) ^{2}d\overline{\mu }\left( x,t\right) \equiv
\left\Vert \mathbb{P}^{\alpha }\left( \mathbf{1}_{I}\sigma \right)
\right\Vert _{L^{2}(\widehat{I},\overline{\mu })}^{2}\lesssim \left( \left( 
\mathcal{E}_{\alpha }^{\limfunc{strong}}\right) ^{2}+\mathcal{A}_{2}^{\alpha
}+\mathcal{A}_{2}^{\alpha ,\ast }+A_{2}^{\alpha ,\limfunc{punct}}\right)
\sigma (I)\,,  \label{e.t1 n}
\end{equation}%
\begin{equation}
\int_{\mathbb{R}^{n}}[\mathbb{Q}^{\alpha }(t\mathbf{1}_{\widehat{I}}%
\overline{\mu })]^{2}d\sigma (x)\lesssim \left( \left( \mathcal{E}_{\alpha
}^{\limfunc{strong}}\right) ^{2}+\mathcal{A}_{2}^{\alpha }+\mathcal{A}%
_{2}^{\alpha ,\ast }+A_{2}^{\alpha ,\limfunc{punct}}\right) \int_{\widehat{I}%
}t^{2}d\overline{\mu }(x,t),  \label{e.t2 n}
\end{equation}%
for all \emph{dyadic} quasicubes $I\in \Omega \mathcal{D}$, and where the
dual Poisson operator $\mathbb{Q}^{\alpha }$ is given by 
\begin{equation*}
\mathbb{Q}^{\alpha }(t\mathbf{1}_{\widehat{I}}\overline{\mu })\left(
x\right) =\int_{\widehat{I}}\frac{t^{2}}{\left( t^{2}+\lvert x-y\rvert
^{2}\right) ^{\frac{n+1-\alpha }{2}}}d\overline{\mu }\left( y,t\right) \,.
\end{equation*}%
It is important to note that we can choose for $\Omega \mathcal{D}$ any
fixed dyadic quasigrid, the compensating point being that the integrations
on the left sides of (\ref{e.t1 n}) and (\ref{e.t2 n}) are taken over the
entire spaces $\mathbb{R}_{+}^{n+1}$ and $\mathbb{R}^{n}$ respectively.

\begin{remark}
There is a gap in the proof of the Poisson inequality at the top of page 542
in \cite{Saw3}. However, this gap can be fixed as in \cite{SaWh} or \cite%
{LaSaUr1}.
\end{remark}

\subsection{Poisson testing}

We now turn to proving the Poisson testing conditions (\ref{e.t1 n}) and (%
\ref{e.t2 n}). The same testing conditions have been considered in \cite%
{SaShUr5} but in the setting of no common point masses, and the proofs there
carry over to the situation here, but careful attention must now be paid to
the possibility of common point masses. In \cite{Hyt2} Hyt\"{o}nen
circumvented this difficulty by introducing a Poisson operator `with holes',
which was then analyzed using shifted dyadic grids, but part of his argument
was heavily dependent on the dimension being $n=1$, and the extension of
this argument to higher dimensions is feasible (see earlier versions of this
paper on the \textit{arXiv}), but technically very involved. We circumvent
the difficulty of permitting common point masses here instead by using the
energy Muckenhoupt constants $A_{2}^{\alpha ,\limfunc{energy}}$ and $%
A_{2}^{\alpha ,\ast ,\limfunc{energy}}$, which require control by the
punctured Muckenhoupt constants $A_{2}^{\alpha ,\limfunc{punct}}$ and $%
A_{2}^{\alpha ,\ast ,\limfunc{punct}}$. The following elementary Poisson
inequalities (see e.g. \cite{Vol}) will be used extensively.

\begin{lemma}
\label{Poisson inequalities}Suppose that $J,K,I$ are quasicubes in $\mathbb{R%
}^{n}$, and that $\mu $ is a positive measure supported in $\mathbb{R}%
^{n}\setminus I$. If $J\subset K\subset 2K\subset I$, then%
\begin{equation*}
\frac{\mathrm{P}^{\alpha }\left( J,\mu \right) }{\left\vert J\right\vert ^{%
\frac{1}{n}}}\lesssim \frac{\mathrm{P}^{\alpha }\left( K,\mu \right) }{%
\left\vert K\right\vert ^{\frac{1}{n}}}\lesssim \frac{\mathrm{P}^{\alpha
}\left( J,\mu \right) }{\left\vert J\right\vert ^{\frac{1}{n}}},
\end{equation*}%
while if $2J\subset K\subset I$, then%
\begin{equation*}
\frac{\mathrm{P}^{\alpha }\left( K,\mu \right) }{\left\vert K\right\vert ^{%
\frac{1}{n}}}\lesssim \frac{\mathrm{P}^{\alpha }\left( J,\mu \right) }{%
\left\vert J\right\vert ^{\frac{1}{n}}}.
\end{equation*}
\end{lemma}

Now we record the bounded overlap of the projections $\mathsf{P}%
_{F,J}^{\omega }$.

\begin{lemma}
\label{tau ovelap}Suppose $\mathsf{P}_{F,J}^{\omega }$ is as above and fix
any $I_{0}\in \Omega \mathcal{D}$, so that $I_{0}$, $F$ and $J$ all lie in a
common quasigrid. If $J\in \mathcal{M}_{\left( \mathbf{r},\varepsilon
\right) -\limfunc{deep}}\left( F\right) $ for some $F\in \mathcal{F}$ with $%
F\supsetneqq I_{0}\supset J$ and $\mathsf{P}_{F,J}^{\omega }\neq 0$, then 
\begin{equation*}
F=\pi _{\mathcal{F}}^{\left( \ell \right) }I_{0}\text{ for some }0\leq \ell
\leq \mathbf{\tau }.
\end{equation*}%
As a consequence we have the bounded overlap,%
\begin{equation*}
\#\left\{ F\in \mathcal{F}:J\subset I_{0}\subsetneqq F\text{ for some }J\in 
\mathcal{M}_{\left( \mathbf{r},\varepsilon \right) -\limfunc{deep}}\left(
F\right) \text{ with }\mathsf{P}_{F,J}^{\omega }\neq 0\right\} \leq \mathbf{%
\tau }.
\end{equation*}
\end{lemma}

Finally we record the only places in the proof where the \emph{refined}
quasienergy conditions are used. This lemma will be used in bounding both of
the local Poisson testing conditions. Recall that $\mathcal{A}\Omega 
\mathcal{D}$ consists of all alternate $\Omega \mathcal{D}$-dyadic
quasicubes where $K$ is alternate dyadic if it is a union of $2^{n}$ $\Omega 
\mathcal{D}$-dyadic quasicubes $K^{\prime }$ with $\ell \left( K^{\prime
}\right) =\frac{1}{2}\ell \left( K\right) $. See \cite{SaShUr7} for a proof
when common point masses are prohibited, and the presence of common point
masses here requires no change.

\begin{remark}
The following lemma is another of the key results on the way to the proof of
our theorem, and is an analogue of the corresponding lemma from \cite%
{SaShUr5}, but with the right hand side involving only the plugged energy
constants and the energy Muckenhoupt constants.
\end{remark}

\begin{lemma}
\label{refined lemma}Let $\Omega \mathcal{D},\mathcal{F\subset }\Omega 
\mathcal{D}$ be quasigrids and $\left\{ \mathsf{P}_{F,J}^{\omega }\right\} 
_{\substack{ F\in \mathcal{F}  \\ J\in \mathcal{M}_{\left( \mathbf{r}%
,\varepsilon \right) -\limfunc{deep}}\left( F\right) }}$ be as above with $%
J,F$ in the dyadic quasigrid $\Omega \mathcal{D}$. For any alternate
quasicube $I\in \mathcal{A}\Omega \mathcal{D}$ define%
\begin{equation}
B\left( I\right) \equiv \sum_{F\in \mathcal{F}:\ F\supsetneqq I^{\prime }%
\text{ for some }I^{\prime }\in \mathfrak{C}\left( I\right) }\sum_{J\in 
\mathcal{M}_{\left( \mathbf{r},\varepsilon \right) -\limfunc{deep}}\left(
F\right) :\ J\subset I}\left( \frac{\mathrm{P}^{\alpha }\left( J,\mathbf{1}%
_{I}\sigma \right) }{\left\vert J\right\vert ^{\frac{1}{n}}}\right)
^{2}\left\Vert \mathsf{P}_{F,J}^{\omega }\mathbf{x}\right\Vert _{L^{2}\left(
\omega \right) }^{2}\ .  \label{term B}
\end{equation}%
Then%
\begin{equation}
B\left( I\right) \lesssim \mathbf{\tau }\left( \left( \mathcal{E}_{\alpha }^{%
\limfunc{strong}}\right) ^{2}+A_{2}^{\alpha ,\limfunc{energy}}\right)
\left\vert I\right\vert _{\sigma }\ .  \label{B bound}
\end{equation}
\end{lemma}

\subsection{The forward Poisson testing inequality}

Fix $I\in \Omega \mathcal{D}$. We split the integration on the left side of (%
\ref{e.t1 n}) into a local and global piece:%
\begin{equation*}
\int_{\mathbb{R}_{+}^{n+1}}\mathbb{P}^{\alpha }\left( \mathbf{1}_{I}\sigma
\right) ^{2}d\overline{\mu }=\int_{\widehat{I}}\mathbb{P}^{\alpha }\left( 
\mathbf{1}_{I}\sigma \right) ^{2}d\overline{\mu }+\int_{\mathbb{R}%
_{+}^{n+1}\setminus \widehat{I}}\mathbb{P}^{\alpha }\left( \mathbf{1}%
_{I}\sigma \right) ^{2}d\overline{\mu }\equiv \mathbf{Local}\left( I\right) +%
\mathbf{Global}\left( I\right) ,
\end{equation*}%
where more explicitly,%
\begin{eqnarray}
&&\mathbf{Local}\left( I\right) \equiv \int_{\widehat{I}}\left[ \mathbb{P}%
^{\alpha }\left( \mathbf{1}_{I}\sigma \right) \left( x,t\right) \right] ^{2}d%
\overline{\mu }\left( x,t\right) ;\ \ \ \ \ \overline{\mu }\equiv \frac{1}{%
t^{2}}\mu ,  \label{def local forward} \\
\text{i.e. }\overline{\mu } &\equiv &\sum_{J\in \Omega \mathcal{D}}\frac{1}{%
\ell \left( J\right) ^{2}}\ \sum_{F\in \mathcal{F}}\sum_{J\in \mathcal{M}%
_{\left( \mathbf{r},\varepsilon \right) -\limfunc{deep}}\left( F\right)
}\left\Vert \mathsf{P}_{F,J}^{\omega }\mathbf{x}\right\Vert _{L^{2}\left(
\omega \right) }^{2}\cdot \delta _{\left( c_{J},\ell \left( J\right) \right)
}.  \notag
\end{eqnarray}%
Here is a brief schematic diagram of the decompositions, with bounds in $%
\fbox{}$, used in this subsection:%
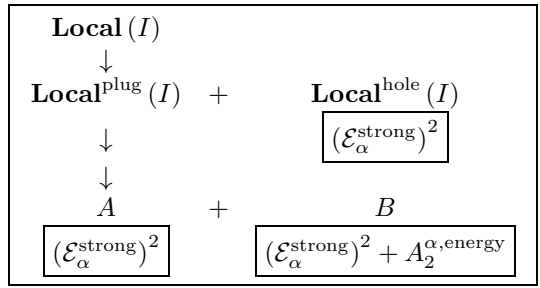
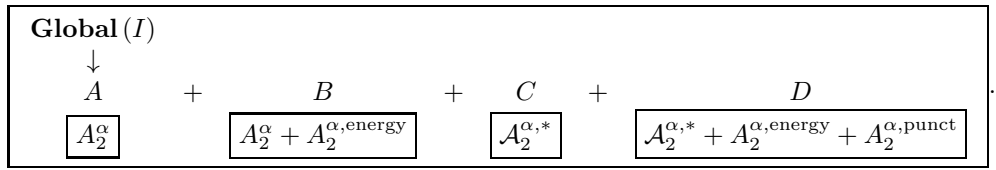
\begin{equation*}
\fbox{$%
\begin{array}{ccc}
\mathbf{Local}\left( I\right) &  &  \\ 
\downarrow &  &  \\ 
\mathbf{Local}^{\limfunc{plug}}\left( I\right) & + & \mathbf{Local}^{%
\limfunc{hole}}\left( I\right) \\ 
\downarrow &  & \fbox{$\left( \mathcal{E}_{\alpha }^{\limfunc{strong}%
}\right) ^{2}$} \\ 
\downarrow &  &  \\ 
A & + & B \\ 
\fbox{$\left( \mathcal{E}_{\alpha }^{\limfunc{strong}}\right) ^{2}$} &  & 
\fbox{$\left( \mathcal{E}_{\alpha }^{\limfunc{strong}}\right)
^{2}+A_{2}^{\alpha ,\limfunc{energy}}$}%
\end{array}%
$}
\end{equation*}%
and%
\begin{equation*}
\fbox{$%
\begin{array}{ccccccc}
\mathbf{Global}\left( I\right) &  &  &  &  &  &  \\ 
\downarrow &  &  &  &  &  &  \\ 
A & + & B & + & C & + & D \\ 
\fbox{$A_{2}^{\alpha }$} &  & \fbox{$A_{2}^{\alpha }+A_{2}^{\alpha ,\limfunc{%
energy}}$} &  & \fbox{$\mathcal{A}_{2}^{\alpha ,\ast }$} &  & \fbox{$%
\mathcal{A}_{2}^{\alpha ,\ast }+A_{2}^{\alpha ,\limfunc{energy}%
}+A_{2}^{\alpha ,\limfunc{punct}}$}%
\end{array}%
$}.
\end{equation*}

An important consequence of the fact that $I$ and $J$ lie in the same
quasigrid $\Omega \mathcal{D}=\Omega \mathcal{D}^{\omega }$, is that%
\begin{equation}
\left( c\left( J\right) ,\ell \left( J\right) \right) \in \widehat{I}\text{ 
\textbf{if and only if} }J\subset I.  \label{tent consequence}
\end{equation}%
We thus have

\begin{eqnarray*}
&&\mathbf{Local}\left( I\right) =\int_{\widehat{I}}\mathbb{P}^{\alpha
}\left( \mathbf{1}_{I}\sigma \right) \left( x,t\right) ^{2}d\overline{\mu }%
\left( x,t\right) \\
&=&\sum_{F\in \mathcal{F}}\sum_{J\in \mathcal{M}_{\mathbf{r}-\limfunc{deep}%
}(F):\ J\subset I}\mathbb{P}^{\alpha }\left( \mathbf{1}_{I}\sigma \right)
\left( c_{J},\left\vert J\right\vert ^{\frac{1}{n}}\right) ^{2}\left\Vert 
\mathsf{P}_{F,J}^{\omega }\frac{\mathbf{x}}{\left\vert J\right\vert ^{\frac{1%
}{n}}}\right\Vert _{L^{2}\left( \omega \right) }^{2} \\
&\approx &\sum_{F\in \mathcal{F}}\sum_{J\in \mathcal{M}_{\mathbf{r}-\limfunc{%
deep}}\left( F\right) :\ J\subset I}\mathrm{P}^{\alpha }\left( J,\mathbf{1}%
_{I}\sigma \right) ^{2}\lVert \mathsf{P}_{F,J}^{\omega }\frac{\mathbf{x}}{%
\left\vert J\right\vert ^{\frac{1}{n}}}\rVert _{L^{2}\left( \omega \right)
}^{2} \\
&\lesssim &\mathbf{Local}^{\limfunc{plug}}\left( I\right) +\mathbf{Local}^{%
\func{hole}}\left( I\right) ,
\end{eqnarray*}%
where the `plugged' local sum $\mathbf{Local}^{\limfunc{plug}}\left(
I\right) $ is given by 
\begin{align*}
& \mathbf{Local}^{\limfunc{plug}}\left( I\right) \equiv \sum_{F\in \mathcal{F%
}}\sum_{J\in \mathcal{M}_{\mathbf{r}-\limfunc{deep}}\left( F\right) :\
J\subset I}\left( \frac{\mathrm{P}^{\alpha }\left( J,\mathbf{1}_{F\cap
I}\sigma \right) }{\left\vert J\right\vert ^{\frac{1}{n}}}\right)
^{2}\left\Vert \mathsf{P}_{F,J}^{\omega }\mathbf{x}\right\Vert _{L^{2}\left(
\omega \right) }^{2} \\
& =\left\{ \sum_{F\in \mathcal{F}:\ F\subset I}+\sum_{F\in \mathcal{F}:\
F\supsetneqq I}\right\} \sum_{J\in \mathcal{M}_{\mathbf{r}-\limfunc{deep}%
}\left( F\right) :\ J\subset I}\left( \frac{\mathrm{P}^{\alpha }\left( J,%
\mathbf{1}_{F\cap I}\sigma \right) }{\left\vert J\right\vert ^{\frac{1}{n}}}%
\right) ^{2}\left\Vert \mathsf{P}_{F,J}^{\omega }\mathbf{x}\right\Vert
_{L^{2}\left( \omega \right) }^{2} \\
& =A+B.
\end{align*}%
Then a \emph{trivial} application of the deep quasienergy condition (where
`trivial' means that the outer decomposition is just a single quasicube)
gives 
\begin{eqnarray*}
A &\leq &\sum_{F\in \mathcal{F}:\ F\subset I}\sum_{J\in \mathcal{M}_{\mathbf{%
r}-\limfunc{deep}}\left( F\right) }\left( \frac{\mathrm{P}^{\alpha }\left( J,%
\mathbf{1}_{F}\sigma \right) }{\left\vert J\right\vert ^{\frac{1}{n}}}%
\right) ^{2}\left\Vert \mathsf{P}_{F,J}^{\omega }\mathbf{x}\right\Vert
_{L^{2}\left( \omega \right) }^{2} \\
&\leq &\sum_{F\in \mathcal{F}:\ F\subset I}\left( \mathcal{E}_{\alpha }^{%
\limfunc{strong}}\right) ^{2}\left\vert F\right\vert _{\sigma }\lesssim
\left( \mathcal{E}_{\alpha }^{\limfunc{strong}}\right) ^{2}\left\vert
I\right\vert _{\sigma }\,,
\end{eqnarray*}%
since $\left\Vert \mathsf{P}_{F,J}^{\omega }x\right\Vert _{L^{2}\left(
\omega \right) }^{2}\leq \left\Vert \mathsf{P}_{J}^{\limfunc{good},\omega }%
\mathbf{x}\right\Vert _{L^{2}\left( \omega \right) }^{2}$, where we recall
that the quasienergy constant $\mathcal{E}_{\alpha }^{\limfunc{strong}}$ is
defined in Definition \ref{def strong quasienergy}. We also used that the
stopping quasicubes $\mathcal{F}$ satisfy a $\sigma $-Carleson measure
estimate, 
\begin{equation*}
\sum_{F\in \mathcal{F}:\ F\subset F_{0}}\left\vert F\right\vert _{\sigma
}\lesssim \left\vert F_{0}\right\vert _{\sigma }.
\end{equation*}%
Lemma \ref{refined lemma} applies to the remaining term $B$ to obtain the
bound%
\begin{equation*}
B\lesssim \mathbf{\tau }\left( \left( \mathcal{E}_{\alpha }^{\limfunc{strong}%
}\right) ^{2}+A_{2}^{\alpha ,\limfunc{energy}}\right) \left\vert
I\right\vert _{\sigma }\ .
\end{equation*}

It remains then to show the inequality with `holes', where the support of $%
\sigma $ is restricted to the complement of the quasicube $F$. Thus for $%
J\in \mathcal{M}_{\left( \mathbf{r},\varepsilon \right) -\limfunc{deep}%
}\left( F\right) $ we may use $I\setminus F$ in the argument of the Poisson
integral. We consider%
\begin{equation*}
\mathbf{Local}^{\func{hole}}\left( I\right) =\sum_{F\in \mathcal{F}%
}\sum_{J\in \mathcal{M}_{\left( \mathbf{r},\varepsilon \right) -\limfunc{deep%
}}\left( F\right) :\ J\subset I}\left( \frac{\mathrm{P}^{\alpha }\left( J,%
\mathbf{1}_{I\setminus F}\sigma \right) }{\left\vert J\right\vert ^{\frac{1}{%
n}}}\right) ^{2}\left\Vert \mathsf{P}_{F,J}^{\omega }\mathbf{x}\right\Vert
_{L^{2}\left( \omega \right) }^{2}\ .
\end{equation*}

\begin{lemma}
\label{local hole}We have 
\begin{equation}
\mathbf{Local}^{\func{hole}}\left( I\right) \lesssim \left( \mathcal{E}%
_{\alpha }^{\limfunc{strong}}\right) ^{2}\left\vert I\right\vert _{\sigma
}\,.  \label{RTS n}
\end{equation}
\end{lemma}

Details are left to the reader,or see \cite{SaShUr7} or \cite{SaShUr6} for a
proof. This completes the proof of%
\begin{eqnarray}
&&  \label{local} \\
\mathbf{Local}\left( L\right) &\approx &\sum_{F\in \mathcal{F}}\sum_{J\in 
\mathcal{M}_{\left( \mathbf{r},\varepsilon \right) -\limfunc{deep}}\left(
F\right) :\ J\subset L}\left( \frac{\mathrm{P}^{\alpha }\left( J,\mathbf{1}%
_{L}\sigma \right) }{\left\vert J\right\vert ^{\frac{1}{n}}}\right)
^{2}\left\Vert \mathsf{P}_{F,J}^{\omega }\mathbf{x}\right\Vert _{L^{2}\left(
\omega \right) }^{2}  \notag \\
&\lesssim &\left( \left( \mathcal{E}_{\alpha }^{\limfunc{strong}}\right)
^{2}+A_{2}^{\alpha ,\limfunc{energy}}\right) \left\vert L\right\vert
_{\sigma },\ \ \ L\in \Omega \mathcal{D}.  \notag
\end{eqnarray}

\subsubsection{The alternate local estimate}

For future use, we prove a strengthening of the local estimate $\mathbf{Local%
}\left( L\right) $ to \emph{alternate} quasicubes $M\in \mathcal{A}\Omega 
\mathcal{D}$.

\begin{lemma}
\label{shifted}With notation as above and $M\in \mathcal{A}\Omega \mathcal{D}
$ an alternate quasicube, we have 
\begin{eqnarray}
&&  \label{shifted local} \\
\mathbf{Local}\left( M\right) &\equiv &\sum_{F\in \mathcal{F}}\sum_{J\in 
\mathcal{M}_{\left( \mathbf{r},\varepsilon \right) -\limfunc{deep}}\left(
F\right) :\ J\subset M}\left( \frac{\mathrm{P}^{\alpha }\left( J,\mathbf{1}%
_{M}\sigma \right) }{\left\vert J\right\vert ^{\frac{1}{n}}}\right)
^{2}\left\Vert \mathsf{P}_{F,J}^{\omega }\mathbf{x}\right\Vert _{L^{2}\left(
\omega \right) }^{2}  \notag \\
&\lesssim &\left( \left( \mathcal{E}_{\alpha }^{\limfunc{strong}}\right)
^{2}+A_{2}^{\alpha ,\limfunc{energy}}\right) \left\vert M\right\vert
_{\sigma },\ \ \ M\in \mathcal{A}\Omega \mathcal{D}.  \notag
\end{eqnarray}
\end{lemma}

Again details are left to the reader, or see \cite{SaShUr7} or \cite{SaShUr6}
for a proof.

\subsubsection{The global estimate}

Now we turn to proving the following estimate for the global part of the
first testing condition \eqref{e.t1 n}:%
\begin{equation*}
\mathbf{Global}\left( I\right) =\int_{\mathbb{R}_{+}^{n+1}\setminus \widehat{%
I}}\mathbb{P}^{\alpha }\left( \mathbf{1}_{I}\sigma \right) ^{2}d\overline{%
\mu }\lesssim \mathcal{A}_{2}^{\alpha ,\ast }\left\vert I\right\vert
_{\sigma }.
\end{equation*}%
We begin by decomposing the integral on the right into four pieces. As a
particular consequence of Lemma \ref{tau ovelap}, we note that given $J$,
there are at most a fixed number $\mathbf{\tau }$ of $F\in \mathcal{F}$ such
that $J\in \mathcal{M}_{\mathbf{r}-\limfunc{deep}}\left( F\right) $. We have:%
\begin{eqnarray*}
&&\int_{\mathbb{R}_{+}^{n+1}\setminus \widehat{I}}\mathbb{P}^{\alpha }\left( 
\mathbf{1}_{I}\sigma \right) ^{2}d\mu \leq \sum_{J:\ \left( c_{J},\ell
\left( J\right) \right) \in \mathbb{R}_{+}^{n+1}\setminus \widehat{I}}%
\mathbb{P}^{\alpha }\left( \mathbf{1}_{I}\sigma \right) \left( c_{J},\ell
\left( J\right) \right) ^{2}\sum_{\substack{ F\in \mathcal{F}  \\ J\in 
\mathcal{M}_{\left( \mathbf{r},\varepsilon \right) -\limfunc{deep}}\left(
F\right) }}\left\Vert \mathsf{P}_{F,J}^{\omega }\frac{\mathbf{x}}{\left\vert
J\right\vert ^{\frac{1}{n}}}\right\Vert _{L^{2}\left( \omega \right) }^{2} \\
&=&\left\{ \sum_{\substack{ J\cap 3I=\emptyset  \\ \ell \left( J\right) \leq
\ell \left( I\right) }}+\sum_{J\subset 3I\setminus I}+\sum_{\substack{ J\cap
I=\emptyset  \\ \ell \left( J\right) >\ell \left( I\right) }}%
+\sum_{J\supsetneqq I}\right\} \mathbb{P}^{\alpha }\left( \mathbf{1}%
_{I}\sigma \right) \left( c_{J},\ell \left( J\right) \right) ^{2}\sum 
_{\substack{ F\in \mathcal{F}:  \\ J\in \mathcal{M}_{\left( \mathbf{r}%
,\varepsilon \right) -\limfunc{deep}}\left( F\right) }}\left\Vert \mathsf{P}%
_{F,J}^{\omega }\frac{\mathbf{x}}{\left\vert J\right\vert ^{\frac{1}{n}}}%
\right\Vert _{L^{2}\left( \omega \right) }^{2} \\
&=&A+B+C+D.
\end{eqnarray*}

Terms $A$, $B$ and $C$ are handled almost the same as in \cite{SaShUr7}, and
we leave them for the reader. As always complete details are in \cite%
{SaShUr6}.

Finally, we turn to term $D$ which is significantly different due to the
presence of common point masses, more precisely a new \emph{`preparation to
puncture'} argument arises which is explained in detail below. The
quasicubes $J$ occurring here are included in the set of ancestors $%
A_{k}\equiv \pi _{\Omega \mathcal{D}}^{\left( k\right) }I$ of $I$, $1\leq
k<\infty $.%
\begin{eqnarray*}
D &=&\sum_{k=1}^{\infty }\mathbb{P}^{\alpha }\left( \mathbf{1}_{I}\sigma
\right) \left( c\left( A_{k}\right) ,\left\vert A_{k}\right\vert ^{\frac{1}{n%
}}\right) ^{2}\sum_{\substack{ F\in \mathcal{F}:  \\ A_{k}\in \mathcal{M}%
_{\left( \mathbf{r},\varepsilon \right) -\limfunc{deep}}\left( F\right) }}%
\left\Vert \mathsf{P}_{F,A_{k}}^{\omega }\frac{\mathbf{x}}{\lvert
A_{k}\rvert ^{\frac{1}{n}}}\right\Vert _{L^{2}\left( \omega \right) }^{2} \\
&=&\sum_{k=1}^{\infty }\mathbb{P}^{\alpha }\left( \mathbf{1}_{I}\sigma
\right) \left( c\left( A_{k}\right) ,\left\vert A_{k}\right\vert ^{\frac{1}{n%
}}\right) ^{2}\sum_{\substack{ F\in \mathcal{F}:  \\ A_{k}\in \mathcal{M}%
_{\left( \mathbf{r},\varepsilon \right) -\limfunc{deep}}\left( F\right) }}%
\sum_{J^{\prime }\in \mathcal{C}_{F}^{\limfunc{good},\mathbf{\tau }-\limfunc{%
shift}}:\ J^{\prime }\subset A_{k}\setminus I}\left\Vert \bigtriangleup
_{J^{\prime }}^{\omega }\frac{\mathbf{x}}{\lvert A_{k}\rvert ^{\frac{1}{n}}}%
\right\Vert _{L^{2}\left( \omega \right) }^{2} \\
&&+\sum_{k=1}^{\infty }\mathbb{P}^{\alpha }\left( \mathbf{1}_{I}\sigma
\right) \left( c\left( A_{k}\right) ,\left\vert A_{k}\right\vert ^{\frac{1}{n%
}}\right) ^{2}\sum_{\substack{ F\in \mathcal{F}:  \\ A_{k}\in \mathcal{M}%
_{\left( \mathbf{r},\varepsilon \right) -\limfunc{deep}}\left( F\right) }}%
\sum_{J^{\prime }\in \mathcal{C}_{F}^{\limfunc{good},\mathbf{\tau }-\limfunc{%
shift}}:\ J^{\prime }\subset I}\left\Vert \bigtriangleup _{J^{\prime
}}^{\omega }\frac{\mathbf{x}}{\lvert A_{k}\rvert ^{\frac{1}{n}}}\right\Vert
_{L^{2}\left( \omega \right) }^{2} \\
&&+\sum_{k=1}^{\infty }\mathbb{P}^{\alpha }\left( \mathbf{1}_{I}\sigma
\right) \left( c\left( A_{k}\right) ,\left\vert A_{k}\right\vert ^{\frac{1}{n%
}}\right) ^{2}\sum_{\substack{ F\in \mathcal{F}:  \\ A_{k}\in \mathcal{M}%
_{\left( \mathbf{r},\varepsilon \right) -\limfunc{deep}}\left( F\right) }}%
\sum_{J^{\prime }\in \mathcal{C}_{F}^{\limfunc{good},\mathbf{\tau }-\limfunc{%
shift}}:\ I\subsetneqq J^{\prime }\subset A_{k}}\left\Vert \bigtriangleup
_{J^{\prime }}^{\omega }\frac{\mathbf{x}}{\lvert A_{k}\rvert ^{\frac{1}{n}}}%
\right\Vert _{L^{2}\left( \omega \right) }^{2} \\
&\equiv &D_{\limfunc{disjoint}}+D_{\limfunc{descendent}}+D_{\limfunc{ancestor%
}}\ .
\end{eqnarray*}%
We thus have from Lemma \ref{tau ovelap} again,%
\begin{eqnarray*}
D_{\limfunc{disjoint}} &=&\sum_{k=1}^{\infty }\mathbb{P}^{\alpha }\left( 
\mathbf{1}_{I}\sigma \right) \left( c\left( A_{k}\right) ,\left\vert
A_{k}\right\vert ^{\frac{1}{n}}\right) ^{2} \\
&&\ \ \ \ \ \ \ \ \ \ \ \ \ \ \ \times \sum_{\substack{ F\in \mathcal{F}: 
\\ A_{k}\in \mathcal{M}_{\left( \mathbf{r},\varepsilon \right) -\limfunc{deep%
}}\left( F\right) }}\sum_{J^{\prime }\in \mathcal{C}_{F}^{\limfunc{good},%
\mathbf{\tau }-\limfunc{shift}}:\ J^{\prime }\subset A_{k}\setminus
I}\left\Vert \bigtriangleup _{J^{\prime }}^{\omega }\frac{\mathbf{x}}{\lvert
A_{k}\rvert ^{\frac{1}{n}}}\right\Vert _{L^{2}\left( \omega \right) }^{2} \\
&\lesssim &\sum_{k=1}^{\infty }\left( \frac{\left\vert I\right\vert _{\sigma
}\left\vert A_{k}\right\vert ^{\frac{1}{n}}}{\left\vert A_{k}\right\vert ^{1+%
\frac{1-\alpha }{n}}}\right) ^{2}\mathbf{\tau \;}\left\vert A_{k}\setminus
I\right\vert _{\omega }=\mathbf{\tau }\left\{ \frac{\left\vert I\right\vert
_{\sigma }}{\left\vert I\right\vert ^{1-\frac{\alpha }{n}}}%
\sum_{k=1}^{\infty }\frac{\left\vert I\right\vert ^{1-\frac{\alpha }{n}}}{%
\left\vert A_{k}\right\vert ^{2\left( 1-\frac{\alpha }{n}\right) }}%
\left\vert A_{k}\setminus I\right\vert _{\omega }\right\} \left\vert
I\right\vert _{\sigma } \\
&\lesssim &\mathbf{\tau }\left\{ \frac{\left\vert I\right\vert _{\sigma }}{%
\left\vert I\right\vert ^{1-\frac{\alpha }{n}}}\mathcal{P}^{\alpha }\left( I,%
\mathbf{1}_{I^{c}}\omega \right) \right\} \left\vert I\right\vert _{\sigma
}\lesssim \mathbf{\tau }\mathcal{A}_{2}^{\alpha ,\ast }\left\vert
I\right\vert _{\sigma },
\end{eqnarray*}%
since%
\begin{eqnarray*}
\sum_{k=1}^{\infty }\frac{\left\vert I\right\vert ^{1-\frac{\alpha }{n}}}{%
\left\vert A_{k}\right\vert ^{2\left( 1-\frac{\alpha }{n}\right) }}%
\left\vert A_{k}\setminus I\right\vert _{\omega } &=&\int \sum_{k=1}^{\infty
}\frac{\left\vert I\right\vert ^{1-\frac{\alpha }{n}}}{\left\vert
A_{k}\right\vert ^{2\left( 1-\frac{\alpha }{n}\right) }}\mathbf{1}%
_{A_{k}\setminus I}\left( x\right) d\omega \left( x\right) \\
&=&\int \sum_{k=1}^{\infty }\frac{1}{2^{2\left( 1-\frac{\alpha }{n}\right) k}%
}\frac{\left\vert I\right\vert ^{1-\frac{\alpha }{n}}}{\left\vert
I\right\vert ^{2\left( 1-\frac{\alpha }{n}\right) }}\mathbf{1}%
_{A_{k}\setminus I}\left( x\right) d\omega \left( x\right) \\
&\lesssim &\int_{I^{c}}\left( \frac{\left\vert I\right\vert ^{\frac{1}{n}}}{%
\left[ \left\vert I\right\vert ^{\frac{1}{n}}+\limfunc{quasidist}\left(
x,I\right) \right] ^{2}}\right) ^{n-\alpha }d\omega \left( x\right) =%
\mathcal{P}^{\alpha }\left( I,\mathbf{1}_{I^{c}}\omega \right) .
\end{eqnarray*}%
The next term $D_{\limfunc{descendent}}$ satisfies%
\begin{eqnarray*}
D_{\limfunc{descendent}} &\lesssim &\sum_{k=1}^{\infty }\left( \frac{%
\left\vert I\right\vert _{\sigma }\left\vert A_{k}\right\vert ^{\frac{1}{n}}%
}{\left\vert A_{k}\right\vert ^{1+\frac{1-\alpha }{n}}}\right) ^{2}\mathbf{%
\tau \;}\left\Vert \mathsf{P}_{I}^{\limfunc{good},\omega }\frac{\mathbf{x}}{%
2^{k}\lvert I\rvert ^{\frac{1}{n}}}\right\Vert _{L^{2}\left( \omega \right)
}^{2} \\
&=&\mathbf{\tau }\sum_{k=1}^{\infty }2^{-2k\left( n-\alpha +1\right) }\left( 
\frac{\left\vert I\right\vert _{\sigma }}{\left\vert I\right\vert ^{1-\frac{%
\alpha }{n}}}\right) ^{2}\left\Vert \mathsf{P}_{I}^{\limfunc{good},\omega }%
\frac{\mathbf{x}}{\lvert I\rvert ^{\frac{1}{n}}}\right\Vert _{L^{2}\left(
\omega \right) }^{2} \\
&\lesssim &\mathbf{\tau }\left\{ \frac{\left\vert I\right\vert _{\sigma
}\left\Vert \mathsf{P}_{I}^{\limfunc{good},\omega }\frac{\mathbf{x}}{\lvert
I\rvert ^{\frac{1}{n}}}\right\Vert _{L^{2}\left( \omega \right) }^{2}}{%
\left\vert I\right\vert ^{2\left( 1-\frac{\alpha }{n}\right) }}\right\}
\left\vert I\right\vert _{\sigma }\lesssim \mathbf{\tau }A_{2}^{\alpha ,%
\limfunc{energy}}\left\vert I\right\vert _{\sigma }\ .
\end{eqnarray*}

Finally for $D_{\limfunc{ancestor}}$ we note that each $J^{\prime }$ is of
the form $J^{\prime }=A_{\ell }\equiv \pi _{\Omega \mathcal{D}}^{\left( \ell
\right) }I$ for some $\ell \geq 1$, and that there are at most $C\mathbf{%
\tau }$ pairs $\left( F,A_{k}\right) $ with $k\geq \ell $ such that $%
A_{k}\in \mathcal{M}_{\left( \mathbf{r},\varepsilon \right) -\limfunc{deep}%
}\left( F\right) $ and $J^{\prime }=A_{\ell }\in \mathcal{C}_{F}^{\limfunc{%
good},\mathbf{\tau }-\limfunc{shift}}$. Now we write%
\begin{eqnarray*}
D_{\limfunc{ancestor}} &=&\sum_{k=1}^{\infty }\mathbb{P}^{\alpha }\left( 
\mathbf{1}_{I}\sigma \right) \left( c\left( A_{k}\right) ,\left\vert
A_{k}\right\vert ^{\frac{1}{n}}\right) ^{2}\sum_{\substack{ F\in \mathcal{F}%
:  \\ A_{k}\in \mathcal{M}_{\left( \mathbf{r},\varepsilon \right) -\limfunc{%
deep}}\left( F\right) }}\sum_{J^{\prime }\in \mathcal{C}_{F}^{\limfunc{good},%
\mathbf{\tau }-\limfunc{shift}}:\ I\subsetneqq J^{\prime }\subset
A_{k}}\left\Vert \bigtriangleup _{J^{\prime }}^{\omega }\frac{\mathbf{x}}{%
\lvert A_{k}\rvert ^{\frac{1}{n}}}\right\Vert _{L^{2}\left( \omega \right)
}^{2} \\
&\lesssim &\mathbf{\tau }\sum_{k=1}^{\infty }\left( \frac{\left\vert
I\right\vert _{\sigma }\left\vert A_{k}\right\vert ^{\frac{1}{n}}}{%
\left\vert A_{k}\right\vert ^{1+\frac{1-\alpha }{n}}}\right) ^{2}\sum_{\ell
=1}^{k}\left\Vert \bigtriangleup _{A_{\ell }}^{\omega }\frac{\mathbf{x}}{%
\lvert A_{k}\rvert ^{\frac{1}{n}}}\right\Vert _{L^{2}\left( \omega \right)
}^{2} \\
&\leq &\mathbf{\tau }\sum_{k=1}^{\infty }\left( \frac{\left\vert
I\right\vert _{\sigma }\left\vert A_{k}\right\vert ^{\frac{1}{n}}}{%
\left\vert A_{k}\right\vert ^{1+\frac{1-\alpha }{n}}}\right) ^{2}\left\Vert 
\mathsf{P}_{A_{k}}^{\limfunc{good},\omega }\frac{\mathbf{x}}{\lvert
A_{k}\rvert ^{\frac{1}{n}}}\right\Vert _{L^{2}\left( \omega \right) }^{2}.
\end{eqnarray*}

It is at this point that we must invoke a new \emph{`prepare to puncture'}
argument. Now define $\widetilde{\omega }=\omega -\omega \left( \left\{
p\right\} \right) \delta _{p}$ where $p$ is an atomic point in $I$ for which 
\begin{equation*}
\omega \left( \left\{ p\right\} \right) =\sup_{q\in \mathfrak{P}_{\left(
\sigma ,\omega \right) }:\ q\in I}\omega \left( \left\{ q\right\} \right) .
\end{equation*}%
(If $\omega $ has no atomic point in common with $\sigma $ in $I$ set $%
\widetilde{\omega }=\omega $.) Then we have $\left\vert I\right\vert _{%
\widetilde{\omega }}=\omega \left( I,\mathfrak{P}_{\left( \sigma ,\omega
\right) }\right) $ and%
\begin{equation*}
\frac{\left\vert I\right\vert _{\widetilde{\omega }}}{\left\vert
I\right\vert ^{\left( 1-\frac{\alpha }{n}\right) }}\frac{\left\vert
I\right\vert _{\sigma }}{\left\vert I\right\vert ^{\left( 1-\frac{\alpha }{n}%
\right) }}=\frac{\omega \left( I,\mathfrak{P}_{\left( \sigma ,\omega \right)
}\right) }{\left\vert I\right\vert ^{\left( 1-\frac{\alpha }{n}\right) }}%
\frac{\left\vert I\right\vert _{\sigma }}{\left\vert I\right\vert ^{\left( 1-%
\frac{\alpha }{n}\right) }}\leq A_{2}^{\alpha ,\limfunc{punct}}.
\end{equation*}%
A key observation, already noted in the proof of Lemma \ref{energy A2}
above, is that%
\begin{equation}
\left\Vert \bigtriangleup _{K}^{\omega }\mathbf{x}\right\Vert _{L^{2}\left(
\omega \right) }^{2}=\left\{ 
\begin{array}{ccc}
\left\Vert \bigtriangleup _{K}^{\omega }\left( \mathbf{x}-\mathbf{p}\right)
\right\Vert _{L^{2}\left( \omega \right) }^{2} & \text{ if } & p\in K \\ 
\left\Vert \bigtriangleup _{K}^{\omega }\mathbf{x}\right\Vert _{L^{2}\left( 
\widetilde{\omega }\right) }^{2} & \text{ if } & p\notin K%
\end{array}%
\right. \leq \ell \left( K\right) ^{2}\left\vert K\right\vert _{\widetilde{%
\omega }},\ \ \ \ \ \text{for all }K\in \Omega \mathcal{D}\ ,
\label{key obs}
\end{equation}%
and so, as in the proof of Lemma \ref{energy A2},%
\begin{equation*}
\left\Vert \mathsf{P}_{A_{k}}^{\limfunc{good},\omega }\frac{\mathbf{x}}{%
\left\vert A_{k}\right\vert ^{\frac{1}{n}}}\right\Vert _{L^{2}\left( \omega
\right) }^{2}\leq 3\left\vert A_{k}\right\vert _{\widetilde{\omega }}\ .
\end{equation*}%
Then we continue with%
\begin{eqnarray*}
&&\mathbf{\tau }\sum_{k=1}^{\infty }\left( \frac{\left\vert I\right\vert
_{\sigma }\left\vert A_{k}\right\vert ^{\frac{1}{n}}}{\left\vert
A_{k}\right\vert ^{1+\frac{1-\alpha }{n}}}\right) ^{2}\left\Vert \mathsf{P}%
_{A_{k}}^{\limfunc{good},\omega }\frac{\mathbf{x}}{\lvert A_{k}\rvert ^{%
\frac{1}{n}}}\right\Vert _{L^{2}\left( \omega \right) }^{2} \\
&\lesssim &\mathbf{\tau }\sum_{k=1}^{\infty }\left( \frac{\left\vert
I\right\vert _{\sigma }\left\vert A_{k}\right\vert ^{\frac{1}{n}}}{%
\left\vert A_{k}\right\vert ^{1+\frac{1-\alpha }{n}}}\right) ^{2}\left\vert
A_{k}\right\vert _{\widetilde{\omega }} \\
&=&\mathbf{\tau }\sum_{k=1}^{\infty }\left( \frac{\left\vert I\right\vert
_{\sigma }}{\left\vert A_{k}\right\vert ^{1-\frac{\alpha }{n}}}\right)
^{2}\left\vert A_{k}\setminus I\right\vert _{\omega }+\mathbf{\tau }%
\sum_{k=1}^{\infty }\left( \frac{\left\vert I\right\vert _{\sigma }}{%
2^{k\left( n-\alpha \right) }\left\vert I\right\vert ^{1-\frac{\alpha }{n}}}%
\right) ^{2}\left\vert I\right\vert _{\widetilde{\omega }} \\
&\lesssim &\mathbf{\tau }\left( \mathcal{A}_{2}^{\alpha ,\ast
}+A_{2}^{\alpha ,\limfunc{punct}}\right) \left\vert I\right\vert _{\sigma },
\end{eqnarray*}%
where the inequality $\sum_{k=1}^{\infty }\left( \frac{\left\vert
I\right\vert _{\sigma }}{\left\vert A_{k}\right\vert ^{1-\frac{\alpha }{n}}}%
\right) ^{2}\left\vert A_{k}\setminus I\right\vert _{\omega }\lesssim 
\mathcal{A}_{2}^{\alpha ,\ast }\left\vert I\right\vert _{\sigma }$ is
already proved above in the estimate for $D_{\limfunc{disjoint}}$.

\subsection{The backward Poisson testing inequality}

Fix $I\in \Omega \mathcal{D}$. It suffices to prove%
\begin{equation}
\mathbf{Back}\left( \widehat{I}\right) \equiv \int_{\mathbb{R}^{n}}\left[ 
\mathbb{Q}^{\alpha }\left( t\mathbf{1}_{\widehat{I}}\overline{\mu }\right)
\left( y\right) \right] ^{2}d\sigma (y)\lesssim \left\{ \mathcal{A}%
_{2}^{\alpha }+\left( \mathcal{E}_{\alpha }^{\limfunc{plug}}+\sqrt{%
A_{2}^{\alpha ,\limfunc{energy}}}\right) \sqrt{A_{2}^{\alpha ,\limfunc{punct}%
}}\right\} \int_{\widehat{I}}t^{2}d\overline{\mu }(x,t).  \label{e.t2 n'}
\end{equation}%
Note that in dimension $n=1$, Hyt\"{o}nen obtained in \cite{Hyt2} the
simpler bound $A_{2}^{\alpha }$ for the term analogous to (\ref{e.t2 n'}).
Here is a brief schematic diagram of the decompositions, with bounds in $%
\fbox{}$, used in this subsection:%
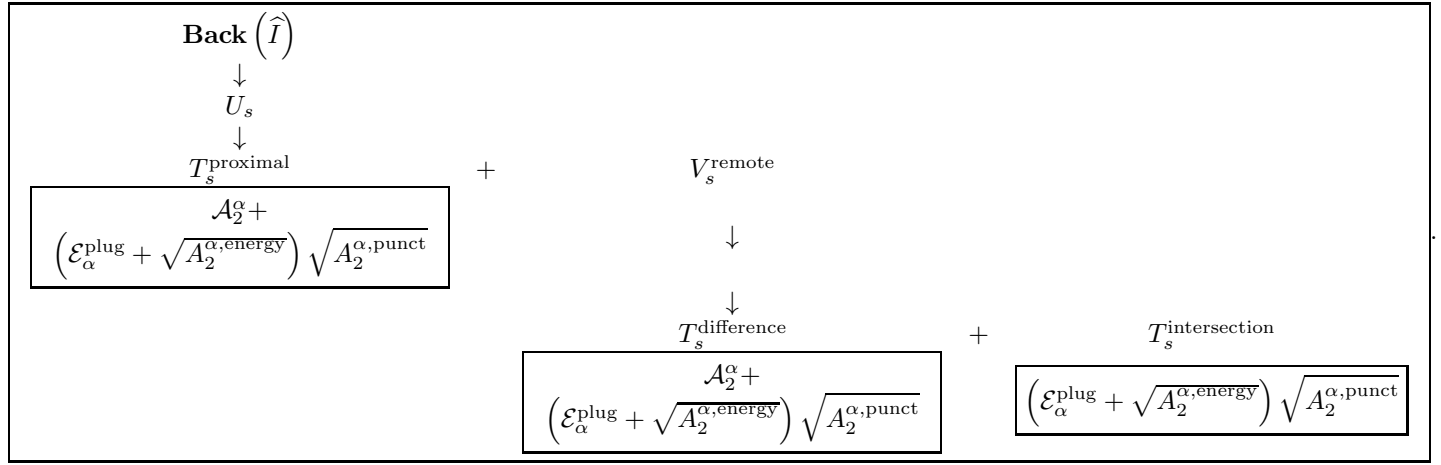
\begin{equation*}
\fbox{$%
\begin{array}{ccccc}
\mathbf{Back}\left( \widehat{I}\right) &  &  &  &  \\ 
\downarrow &  &  &  &  \\ 
U_{s} &  &  &  &  \\ 
\downarrow &  &  &  &  \\ 
T_{s}^{\limfunc{proximal}} & + & V_{s}^{\limfunc{remote}} &  &  \\ 
\fbox{$%
\begin{array}{c}
\mathcal{A}_{2}^{\alpha }+ \\ 
\left( \mathcal{E}_{\alpha }^{\limfunc{plug}}+\sqrt{A_{2}^{\alpha ,\limfunc{%
energy}}}\right) \sqrt{A_{2}^{\alpha ,\limfunc{punct}}}%
\end{array}%
$} &  & \downarrow &  &  \\ 
&  & \downarrow &  &  \\ 
&  & T_{s}^{\limfunc{difference}} & + & T_{s}^{\limfunc{intersection}} \\ 
&  & \fbox{$%
\begin{array}{c}
\mathcal{A}_{2}^{\alpha }+ \\ 
\left( \mathcal{E}_{\alpha }^{\limfunc{plug}}+\sqrt{A_{2}^{\alpha ,\limfunc{%
energy}}}\right) \sqrt{A_{2}^{\alpha ,\limfunc{punct}}}%
\end{array}%
$} &  & \fbox{$\left( \mathcal{E}_{\alpha }^{\limfunc{plug}}+\sqrt{%
A_{2}^{\alpha ,\limfunc{energy}}}\right) \sqrt{A_{2}^{\alpha ,\limfunc{punct}%
}}$}%
\end{array}%
$}.
\end{equation*}%
Using (\ref{tent consequence}) we see that the integral on the right hand
side of (\ref{e.t2 n'}) is 
\begin{equation}
\int_{\widehat{I}}t^{2}d\overline{\mu }=\sum_{F\in \mathcal{F}}\sum_{J\in 
\mathcal{M}_{\left( \mathbf{r},\varepsilon \right) -\limfunc{deep}}\left(
F\right) :\ J\subset I}\lVert \mathsf{P}_{F,J}^{\omega }\mathbf{x}\rVert
_{L^{2}\left( \omega \right) }^{2}\,.  \label{mu I hat}
\end{equation}%
where $\mathsf{P}_{F,J}^{\omega }$ was defined earlier.

We now compute using (\ref{tent consequence}) again that 
\begin{eqnarray}
\mathbb{Q}^{\alpha }\left( t\mathbf{1}_{\widehat{I}}\overline{\mu }\right)
\left( y\right)  &=&\int_{\widehat{I}}\frac{t^{2}}{\left( t^{2}+\left\vert
x-y\right\vert ^{2}\right) ^{\frac{n+1-\alpha }{2}}}d\overline{\mu }\left(
x,t\right)   \label{PI hat} \\
&\approx &\sum_{F\in \mathcal{F}}\sum_{\substack{ J\in \mathcal{M}_{\left( 
\mathbf{r},\varepsilon \right) -\limfunc{deep}}\left( F\right)  \\ J\subset I
}}\frac{\left\Vert \mathsf{P}_{F,J}^{\omega }\mathbf{x}\right\Vert
_{L^{2}\left( \omega \right) }^{2}}{\left( \left\vert J\right\vert ^{\frac{1%
}{n}}+\left\vert y-c_{J}\right\vert \right) ^{n+1-\alpha }},  \notag
\end{eqnarray}%
and then expand the square and integrate to obtain that the term $\mathbf{%
Back}\left( \widehat{I}\right) $ is 
\begin{equation*}
\sum_{\substack{ F\in \mathcal{F} \\ J\in \mathcal{M}_{\left( \mathbf{r}%
,\varepsilon \right) -\limfunc{deep}}\left( F\right)  \\ J\subset I}}\sum
_{\substack{ F^{\prime }\in \mathcal{F}: \\ J^{\prime }\in \mathcal{M}%
_{\left( \mathbf{r},\varepsilon \right) -\limfunc{deep}}\left( F^{\prime
}\right)  \\ J^{\prime }\subset I}}\int_{\mathbb{R}^{n}}\frac{\left\Vert 
\mathsf{P}_{F,J}^{\omega }\mathbf{x}\right\Vert _{L^{2}\left( \omega \right)
}^{2}}{\left( \left\vert J\right\vert ^{\frac{1}{n}}+\left\vert
y-c_{J}\right\vert \right) ^{n+1-\alpha }}\frac{\left\Vert \mathsf{P}%
_{F^{\prime },J^{\prime }}^{\omega }\mathbf{x}\right\Vert _{L^{2}\left(
\omega \right) }^{2}}{\left( \left\vert J^{\prime }\right\vert ^{\frac{1}{n}%
}+\left\vert y-c_{J^{\prime }}\right\vert \right) ^{n+1-\alpha }}d\sigma
\left( y\right) .
\end{equation*}

By symmetry we may assume that $\ell \left( J^{\prime }\right) \leq \ell
\left( J\right) $. We fix an integer $s$, and consider those quasicubes $J$
and $J^{\prime }$ with $\ell \left( J^{\prime }\right) =2^{-s}\ell \left(
J\right) $. For fixed $s$ we will control the expression 
\begin{eqnarray*}
U_{s} &\equiv &\sum_{\substack{ F,F^{\prime }\in \mathcal{F}}}\sum
_{\substack{ J\in \mathcal{M}_{\left( \mathbf{r},\varepsilon \right) -%
\limfunc{deep}}\left( F\right) ,\ J^{\prime }\in \mathcal{M}_{\left( \mathbf{%
r},\varepsilon \right) -\limfunc{deep}}\left( F^{\prime }\right)  \\ %
J,J^{\prime }\subset I,\ \ell \left( J^{\prime }\right) =2^{-s}\ell \left(
J\right) }} \\
&&\times \int_{\mathbb{R}^{n}}\frac{\left\Vert \mathsf{P}_{F,J}^{\omega }%
\mathbf{x}\right\Vert _{L^{2}\left( \omega \right) }^{2}}{\left( \left\vert
J\right\vert ^{\frac{1}{n}}+\left\vert y-c_{J}\right\vert \right)
^{n+1-\alpha }}\frac{\left\Vert \mathsf{P}_{F^{\prime },J^{\prime }}^{\omega
}\mathbf{x}\right\Vert _{L^{2}\left( \omega \right) }^{2}}{\left( \left\vert
J^{\prime }\right\vert ^{\frac{1}{n}}+\left\vert y-c_{J^{\prime
}}\right\vert \right) ^{n+1-\alpha }}d\sigma \left( y\right) ,
\end{eqnarray*}%
by proving that%
\begin{equation}
U_{s}\lesssim 2^{-\delta s}\left\{ \mathcal{A}_{2}^{\alpha }+\left( \mathcal{%
E}_{\alpha }^{\limfunc{strong}}+\sqrt{A_{2}^{\alpha ,\limfunc{energy}}}%
\right) \sqrt{A_{2}^{\alpha ,\limfunc{punct}}}\right\} \int_{\widehat{I}%
}t^{2}d\overline{\mu },\ \ \ \ \ \text{where }\delta =\frac{1}{2n}.
\label{Us bound}
\end{equation}%
With this accomplished, we can sum in $s\geq 0$ to control the term $\mathbf{%
Back}\left( \widehat{I}\right) $. The remaining details of the proof are
very similar to the corresponding arguments in \cite{SaShUr7}, with the only
exception being the repeated use of the \emph{`prepare to puncture'}
argument above whenever the measures $\sigma $ and $\omega $ can `see each
other' in an estimate. We refer the reader to \cite{SaShUr6} for complete
details\footnote{%
In \cite{SaShUr5} and \cite{SaShUr7} the bound for term $B$ in the global
estimate was mistakenly claimed without proof to be simply $\mathcal{A}%
_{2}^{\alpha }$ instead of the correct bound $\mathcal{A}_{2}^{\alpha
}+\left( \mathcal{E}_{\alpha }^{\limfunc{plug}}+\sqrt{A_{2}^{\alpha ,%
\limfunc{energy}}}\right) \sqrt{A_{2}^{\alpha ,\limfunc{punct}}}$ given in 
\cite{SaShUr6}.}.

\section{The stopping form}

This section is virtually unchanged from the corresponding section in \cite%
{SaShUr7}, so we content ourselves with a brief recollection. In the
one-dimensional setting of the Hilbert transform, Hyt\"{o}nen \cite{Hyt2}
observed that "...the innovative verification of the local estimate by Lacey 
\cite{Lac} is already set up in such a way that it is ready for us to borrow
as a black box." The same observation carried over in spirit regarding the
adaptation of Lacey's recursion and stopping time to proving the local
estimate in \cite{SaShUr7}. However, that adaptation involved the splitting
of the stopping form into two sublinear forms, the first handled by methods
in \cite{LaSaUr2}, and the second by the methods in \cite{Lac}. The
arguments are little changed when including common point masses, and we
leave them for the reader (or see \cite{SaShUr6} for the proofs written out
in detail).

\section{Energy dispersed measures}

In this final section we prove that the energy side conditions in our main
theorem hold if both measures are appropriately energy dispersed. We begin
with the definitions of energy dispersed and reversal of energy.

\subsection{Energy dispersed measures and reversal of energy}

Let $\mu $ be a locally finite positive Borel measure on $\mathbb{R}^{n}$.
Recall that for $0\leq k\leq n$, we denote by $\mathcal{L}_{k}^{n}$ the
collection of all $k$-dimensional planes in $\mathbb{R}^{n}$, and for a
quasicube $J$, we define the $k$\emph{-dimensional second moment} $\mathsf{M}%
_{k}^{n}\left( J,\mu \right) $ of $\mu $ on $J$ by%
\begin{equation*}
\mathsf{M}_{k}^{n}\left( J,\mu \right) ^{2}\equiv \inf_{L\in \mathcal{L}%
_{k}^{n}}\int_{J}\limfunc{dist}\left( x,L\right) ^{2}d\mu \left( x\right) .
\end{equation*}%
Finally we defined $\mu $ to be $k$\emph{-energy dispersed} if there is $c>0$
such that 
\begin{equation*}
\mathsf{M}_{k}^{n}\left( J,\mu \right) \geq c\mathsf{M}_{0}^{n}\left( J,\mu
\right) ,\ \ \ \ \ \text{for all quasicubes }J\text{ in }\mathbb{R}^{n}.
\end{equation*}

In order to introduce a useful reformulation of the $k$-dimensional second
moment, we will use the observation that minimizing $k$-planes $L$ pass
through the center of mass. More precisely, for any $k$-plane $L\in \mathcal{%
L}_{k}^{n}$ such that $\int_{A}\limfunc{dist}\left( x,L\right) ^{2}d\mu
\left( x\right) $ is minimized, where $A$ is a set of positive $\mu $%
-measure, we claim that%
\begin{equation*}
\mathbb{E}_{A}^{\mu }x\in L\ .
\end{equation*}%
Indeed, if we rotate coordinates so that $L=\left\{ \left(
x^{1},...,x^{k},a^{k+1},...,a^{n}\right) :\left( x^{1},...,x^{k}\right) \in 
\mathbb{R}^{k}\right\} $, then%
\begin{eqnarray*}
\int_{A}\limfunc{dist}\left( x,L\right) ^{2}d\mu \left( x\right)
&=&\int_{A}\sum_{j=k+1}^{n}\left( x^{j}-a^{j}\right) ^{2}d\mu \left( x\right)
\\
&=&\sum_{j=k+1}^{n}\left[ \int_{A}\left( x^{j}\right) ^{2}d\mu \left(
x\right) -2a^{j}\int_{A}x^{j}d\mu \left( x\right) +\left( a^{j}\right)
^{2}\int_{A}d\mu \left( x\right) \right] \\
&=&\sum_{j=k+1}^{n}\left[ \int_{A}\left( x^{j}\right) ^{2}d\mu \left(
x\right) +\left( \int_{A}d\mu \left( x\right) \right) \left\{ \left(
a^{j}\right) ^{2}-2\frac{\int_{A}x^{j}d\mu \left( x\right) }{\int_{A}d\mu
\left( x\right) }a^{j}\right\} \right]
\end{eqnarray*}%
is minimized over $a^{k+1},...,a^{n}$ when%
\begin{equation*}
a^{j}=\frac{\int_{A}x^{j}d\mu \left( x\right) }{\int_{A}d\mu \left( x\right) 
}=\left( \mathbb{E}_{A}^{\mu }x\right) ^{j},\ \ \ \ \ k+1\leq j\leq n.
\end{equation*}%
This shows that the point $\mathbb{E}_{A}^{\mu }x$ belongs to the $k$-plane $%
L$.

Now we can obtain our reformulation of the $k$-dimensional second moment.
Let $\mathcal{S}_{k}^{n}$ denote the collection of $k$-dimenional subspaces
in $\mathbb{R}^{n}$. If $\mathcal{P}_{S}$ denotes orthogonal projection onto
the subspace $S\in \mathcal{S}_{n-k}^{n}$ where $S=L_{0}^{\perp }$ and $%
L_{0}\in \mathcal{S}_{k}^{n}$ is the subspace parallel to $L$, then we have
the variance identity,%
\begin{eqnarray}
\mathsf{M}_{k}^{n}\left( J,\mu \right) ^{2} &=&\inf_{L\in \mathcal{L}%
_{k}^{n}}\int_{J}\limfunc{dist}\left( x,L\right) ^{2}d\mu \left( x\right)
=\inf_{S\in \mathcal{S}_{n-k}^{n}}\int_{J}\left\vert \mathcal{P}_{S}x-%
\mathcal{P}_{S}\left( \mathbb{E}_{J}^{\mu }x\right) \right\vert ^{2}d\mu
\left( x\right)  \label{variance} \\
&=&\frac{1}{2}\inf_{S\in \mathcal{S}_{n-k}^{n}}\frac{1}{\left\vert
J\right\vert _{\mu }}\int_{J}\int_{J}\left\vert \mathcal{P}_{S}x-\mathcal{P}%
_{S}y\right\vert ^{2}d\mu \left( x\right) d\mu \left( y\right)  \notag \\
&=&\frac{1}{2}\inf_{L_{0}\in \mathcal{S}_{k}^{n}}\frac{1}{\left\vert
J\right\vert _{\mu }}\int_{J}\int_{J}\limfunc{dist}\left( x,L_{0}+y\right)
^{2}d\mu \left( x\right) d\mu \left( y\right) ,  \notag
\end{eqnarray}%
since $\mathcal{P}_{S}\left( \mathbb{E}_{J}^{\mu }x\right) =\mathbb{E}%
_{J}^{\mu }\left( \mathcal{P}_{S}x\right) $. Here we have used in the first
line the fact that the minimizing $k$-planes $L$ pass through the center of
mass $\mathbb{E}_{J}^{\mu }x$ of $x$ in $J$.

Note that if $\mu $ is supported on a $k$-dimensional plane $L$ in $\mathbb{R%
}^{n}$, then $\mathsf{M}_{k}^{n}\left( J,\mu \right) $ vanishes for all
quasicubes $J$. On the other hand, $\mathsf{M}_{0}^{n}\left( J,\mu \right) $
is positive for any quasicube $J$ on which the restriction of $\mu $ is 
\emph{not} a point mass, and we conclude that measures $\mu $ supported on a 
$k$-plane. and whose restriction to $J$ is not a point mass, are \emph{not} $%
k$-energy dispersed. Thus $\mathsf{M}_{k}^{n}\left( J,\mu \right) $ measures
the extent to which a certain `energy' of $\mu $ is not localized to a $k$%
-plane. In this final section we will prove the necessity of the energy
conditions for boundedness of the vector Riesz transform $\mathbf{R}^{\alpha
,n}$ when the locally finite Borel measures $\sigma $ and $\omega $ on $%
\mathbb{R}^{n}$ are $k$-energy dispersed with%
\begin{equation}
\left\{ 
\begin{array}{ccc}
n-k<\alpha <n,\ \alpha \neq n-1 & \text{ if } & 1\leq k\leq n-2 \\ 
0\leq \alpha <n,\ \alpha \neq 1,n-1 & \text{ if } & k=n-1%
\end{array}%
\right. .  \label{index relations}
\end{equation}

Now we recall the definition of strong energy reversal from \cite{SaShUr2}.
We say that a vector $\mathbf{T}^{\alpha }=\left\{ T_{\ell }^{\alpha
}\right\} _{\ell =1}^{2}$ of $\alpha $-fractional transforms in the plane
has \emph{strong} reversal of $\omega $-energy on a cube $J$ if there is a
positive constant $C_{0}$ such that for all $2\leq \gamma \leq 2^{\mathbf{r}%
\left( 1-\varepsilon \right) }$ and for all positive measures $\mu $
supported outside $\gamma J$, we have the inequality%
\begin{equation}
\mathbb{E}_{J}^{\omega }\left[ \left( \mathbf{x}-\mathbb{E}_{J}^{\omega }%
\mathbf{x}\right) ^{2}\right] \left( \frac{\mathrm{P}^{\alpha }\left( J,\mu
\right) }{\left\vert J\right\vert ^{\frac{1}{n}}}\right) ^{2}=\mathsf{E}%
\left( J,\omega \right) ^{2}\mathrm{P}^{\alpha }\left( J,\mu \right)
^{2}\leq C_{0}\ \mathbb{E}_{J}^{\omega }\left\vert \mathbf{T}^{\alpha }\mu -%
\mathbb{E}_{J}^{d\omega }\mathbf{T}^{\alpha }\mu \right\vert ^{2},
\label{will fail}
\end{equation}%
Now note that if $\omega $ is $k$-energy dispersed, then we have%
\begin{equation*}
\mathsf{E}\left( J,\omega \right) ^{2}=\frac{1}{\left\vert J\right\vert
_{\omega }\left\vert J\right\vert ^{\frac{2}{n}}}\mathsf{M}_{0}^{n}\left(
J,\omega \right) ^{2}\lesssim \frac{1}{\left\vert J\right\vert _{\omega
}\left\vert J\right\vert ^{\frac{2}{n}}}\mathsf{M}_{k}^{n}\left( J,\omega
\right) ^{2}\equiv \mathsf{E}_{k}\left( J,\omega \right) ^{2},
\end{equation*}%
and where we have defined on the right hand side the analogous notion of
energy $\mathsf{E}_{k}\left( J,\omega \right) $ in terms of $\mathsf{M}%
_{k}\left( J,\omega \right) $, and which is smaller than $\mathsf{E}\left(
J,\omega \right) $. We now state the main result of this first subsection.

\begin{lemma}
\label{k partial reversal}Let $0\leq \alpha <n$. Suppose that $\omega $ is $%
k $-energy dispersed and that $k$ and $\alpha $ satisfy (\ref{index
relations}). Then the $\alpha $-fractional Riesz transform $\mathbf{R}%
^{\alpha ,n}=\left\{ R_{\ell }^{n,\alpha }\right\} _{\ell =1}^{n}$ has
strong reversal (\ref{will fail}) of $\omega $-energy on \emph{all} cubes $J$
provided $\gamma $ is chosen large enough depending only on $n$ and $\alpha $%
.
\end{lemma}

In \cite{SaShUr4} we showed that energy reversal can fail spectacularly for
measures in general, but left open the possibility of reversing at least one
direction in the energy for $\mathbf{R}^{\alpha ,n}$ when $\alpha \neq 1$ in
the plane $n=2$, and we will show in the next subsection that this is indeed
possible, with even more directions included in higher dimensions.

\subsection{Fractional Riesz transforms and semi-harmonicity}

Now we fix $1\leq \ell \leq n$ and write $x=\left( x^{\prime },x^{\prime
\prime }\right) $ with $x^{\prime }=\left( x_{1},...,x_{\ell }\right) \in 
\mathbb{R}^{\ell }$ and $x^{\prime \prime }=\left( x_{\ell
+1},...,x_{n}\right) \in \mathbb{R}^{n-\ell }$ (when $\ell =n$ we have $%
x=x^{\prime }$). Then we compute for $\beta $ real that%
\begin{eqnarray*}
\bigtriangleup _{x^{\prime }}\left\vert x\right\vert ^{\beta }
&=&\bigtriangleup _{x^{\prime }}\left( \left\vert x^{\prime }\right\vert
^{2}+\left\vert x^{\prime \prime }\right\vert ^{2}\right) ^{\frac{\beta }{2}%
}=\nabla _{x^{\prime }}\cdot \nabla _{x^{\prime }}\left( \left\vert
x^{\prime }\right\vert ^{2}+\left\vert x^{\prime \prime }\right\vert
^{2}\right) ^{\frac{\beta }{2}} \\
&=&\nabla _{x^{\prime }}\cdot \left\{ \frac{\beta }{2}\left( \left\vert
x^{\prime }\right\vert ^{2}+\left\vert x^{\prime \prime }\right\vert
^{2}\right) ^{\frac{\beta }{2}-1}2x^{\prime }\right\} =\beta \nabla
_{x^{\prime }}\cdot \left\{ x^{\prime }\left( \left\vert x^{\prime
}\right\vert ^{2}+\left\vert x^{\prime \prime }\right\vert ^{2}\right) ^{%
\frac{\beta }{2}-1}\right\} \\
&=&\beta \left\{ \left( \nabla _{x^{\prime }}\cdot x^{\prime }\right) \left(
\left\vert x^{\prime }\right\vert ^{2}+\left\vert x^{\prime \prime
}\right\vert ^{2}\right) ^{\frac{\beta -2}{2}}+x^{\prime }\cdot \nabla
_{x^{\prime }}\left( \left\vert x^{\prime }\right\vert ^{2}+\left\vert
x^{\prime \prime }\right\vert ^{2}\right) ^{\frac{\beta -2}{2}}\right\} \\
&=&\beta \left\{ \ell \left( \left\vert x^{\prime }\right\vert
^{2}+\left\vert x^{\prime \prime }\right\vert ^{2}\right) ^{\frac{\beta -2}{2%
}}+x^{\prime }\cdot \frac{\beta -2}{2}\left( \left\vert x^{\prime
}\right\vert ^{2}+\left\vert x^{\prime \prime }\right\vert ^{2}\right) ^{%
\frac{\beta -2}{2}-1}2x^{\prime }\right\} \\
&=&\beta \left\{ \ell \left( \left\vert x^{\prime }\right\vert
^{2}+\left\vert x^{\prime \prime }\right\vert ^{2}\right) ^{\frac{\beta -2}{2%
}}+\left( \beta -2\right) \left\vert x^{\prime }\right\vert ^{2}\left(
\left\vert x^{\prime }\right\vert ^{2}+\left\vert x^{\prime \prime
}\right\vert ^{2}\right) ^{\frac{\beta -4}{2}}\right\} \\
&=&\beta \left\{ \ell \left( \left\vert x^{\prime }\right\vert
^{2}+\left\vert x^{\prime \prime }\right\vert ^{2}\right) \left( \left\vert
x^{\prime }\right\vert ^{2}+\left\vert x^{\prime \prime }\right\vert
^{2}\right) ^{\frac{\beta -4}{2}}+\left( \beta -2\right) \left\vert
x^{\prime }\right\vert ^{2}\left( \left\vert x^{\prime }\right\vert
^{2}+\left\vert x^{\prime \prime }\right\vert ^{2}\right) ^{\frac{\beta -4}{2%
}}\right\} \\
&=&\beta \left\{ \left( \ell +\beta -2\right) \left\vert x^{\prime
}\right\vert ^{2}+\ell \left\vert x^{\prime \prime }\right\vert ^{2}\right\}
\left( \left\vert x^{\prime }\right\vert ^{2}+\left\vert x^{\prime \prime
}\right\vert ^{2}\right) ^{\frac{\beta -4}{2}}.
\end{eqnarray*}

The case of interest for us is when $\beta =\alpha -n+1$, since then 
\begin{equation}
\bigtriangleup _{x^{\prime }}\left\vert x\right\vert ^{\beta }=\nabla
_{x^{\prime }}\cdot \nabla _{x^{\prime }}\left\vert x\right\vert ^{\alpha
-n+1}=\nabla _{x^{\prime }}\cdot \nabla \left\vert x\right\vert ^{\alpha
-n+1}=c_{\alpha ,n}\nabla _{x^{\prime }}\cdot \mathbf{K}^{\alpha ,n}\left(
x\right) ,  \label{interest}
\end{equation}%
where $\mathbf{K}^{\alpha ,n}$ is the vector convolution kernel of the $%
\alpha $-fractional Riesz transform $\mathbf{R}^{\alpha ,n}$. Now if $\ell
=1 $ in this case, then the factor 
\begin{equation*}
F_{\ell ,\beta }\left( x\right) \equiv \left( \ell +\beta -2\right)
\left\vert x^{\prime }\right\vert ^{2}+\ell \left\vert x^{\prime \prime
}\right\vert ^{2}
\end{equation*}%
is $\left( \beta -1\right) \left\vert x^{\prime }\right\vert ^{2}+\left\vert
x^{\prime \prime }\right\vert ^{2}$, and thus in dimension $n\geq 2$, the
factor $F_{1,\beta }\left( x\right) $ will be of one sign for all $x$ if and
only if $\alpha -n+1=\beta >1$, i.e. $\alpha >n$, which is of no use since
the Riesz transform $\mathbf{R}^{\alpha ,n}$\ is defined only for $0\leq
\alpha <n$.

Thus we must assume $\ell \geq 2$ and $\beta =\alpha -n+1$ when $n\geq 2$.
Under these assumptions, we then note that $F_{\ell ,\beta }\left( x\right) $
will be of one sign for all $x$ if $\ell +\beta -2>0$, i.e. $\alpha
>n+1-\ell $, in which case we conclude that%
\begin{eqnarray}
\left\vert \bigtriangleup _{x^{\prime }}\left\vert x\right\vert ^{\alpha
-n+1}\right\vert &=&\left\vert \alpha -n+1\right\vert \left\{ \left( \ell
+\alpha -n-1\right) \left\vert x^{\prime }\right\vert ^{2}+\ell \left\vert
x^{\prime \prime }\right\vert ^{2}\right\} \left( \left\vert x^{\prime
}\right\vert ^{2}+\left\vert x^{\prime \prime }\right\vert ^{2}\right) ^{%
\frac{\alpha -n-3}{2}}  \label{conclude} \\
&\approx &\left( \left\vert x^{\prime }\right\vert ^{2}+\left\vert x^{\prime
\prime }\right\vert ^{2}\right) ^{\frac{\alpha -n-1}{2}}=\left\vert
x\right\vert ^{\alpha -n-1},\ \ \ \ \ \text{for }\alpha \neq n-1.  \notag
\end{eqnarray}%
When $\ell =n$, this shows that $\left\vert \bigtriangleup _{x}\left\vert
x\right\vert ^{\alpha -n+1}\right\vert \approx \left\vert x\right\vert
^{\alpha -n-1}$ for $\alpha >1$ with $\alpha \neq n-1$. But in the case $%
\ell =n$ we can obtain more. Indeed, since $x^{\prime \prime }$ is no longer
present, we have for $0\leq \alpha <1$ that%
\begin{equation*}
\bigtriangleup _{x}\left\vert x\right\vert ^{\alpha -n+1}\approx \left\vert
x\right\vert ^{\alpha -n-1}.
\end{equation*}%
(This includes dimension $n=1$ but only for $0<\alpha <1$).

We summarize these results as follows. For dimension $n\geq 2$ and $x=\left(
x^{\prime },x^{\prime \prime }\right) $ with $x^{\prime }\in \mathbb{R}%
^{\ell }$ and $x^{\prime \prime }\in \mathbb{R}^{n-\ell }$, we have%
\begin{equation*}
\left\vert \bigtriangleup _{x^{\prime }}\left\vert x\right\vert ^{\alpha
-n+1}\right\vert \approx \left\vert x\right\vert ^{\alpha -n-1},
\end{equation*}%
provided 
\begin{eqnarray}
&&\text{\textbf{either} }2\leq \ell \leq n-1\text{ and }n+1-\ell <\alpha <n%
\text{ with }\alpha \neq n-1,  \label{either or} \\
&&\text{\textbf{or} }\ell =n\text{ and }0\leq \alpha <n\text{ with }\alpha
\neq 1,n-1.  \notag
\end{eqnarray}%
Thus the two cases not included are $\alpha =1$ and $\alpha =n-1$. The case $%
\alpha =1$ is not included since $\left\vert x\right\vert ^{\alpha
-n+1}=\left\vert x\right\vert ^{2-n}$ is the fundamental solution of the
Laplacian for $n>2$ and constant for $n=2$. The case $\alpha =n-1$ is not
included since $\left\vert x\right\vert ^{\alpha -n+1}=1$ is constant.

So we now suppose that $\alpha $ and $\ell $ are as in (\ref{either or}),
and we consider $\ell $-planes $L$ intersecting the cube $J$. Recall that
the trace of a matrix is invariant under rotations. Thus for each such $\ell 
$-plane $L$, and for $z\in J\cap L$, we have from (\ref{interest}) and (\ref%
{conclude}), and with $\mathbf{I}^{\alpha +1,n}\mu \left( z\right) \equiv
\int_{\mathbb{R}^{n}}\left\vert z-y\right\vert ^{\alpha +1-n}d\mu \left(
y\right) $ denoting the convolution of $\left\vert x\right\vert ^{\alpha
+1-n}$ with $\mu $, that 
\begin{equation}
\left\vert \nabla _{L}\mathbf{R}^{\alpha ,n}\mu \left( z\right) \right\vert
\gtrsim \left\vert \limfunc{trace}\nabla _{L}\mathbf{R}^{\alpha ,n}\mu
\left( z\right) \right\vert =\left\vert \bigtriangleup _{L}\mathbf{I}%
^{\alpha +1,n}\mu \left( z\right) \right\vert \approx \int \left\vert
y-z\right\vert ^{\alpha -n-1}d\mu \left( y\right) \approx \frac{\mathrm{P}%
^{\alpha }\left( J,\mu \right) }{\left\vert J\right\vert ^{\frac{1}{n}}},
\label{L control}
\end{equation}%
where $\nabla _{L}$ denotes the gradient in the $\ell $-plane $L$, i.e. $%
\nabla _{L}=\mathcal{P}_{S}\nabla $ where $S$ is the subspace parallel to $L$
and $\mathcal{P}_{S}$ is orthogonal projection onto $S$, and where we assume
that the positive measure $\mu $ is supported outside the expanded cube $%
\gamma J$.

We now claim that for every $z\in J\cap L$, the full matrix gradient $\nabla 
\mathbf{R}^{\alpha ,n}\mu \left( z\right) $ is `missing' at most $\ell -1$
`large' directions, i.e. has at least $n-\ell +1$ eigenvalues each of size
at least $c\frac{\mathrm{P}^{\alpha }\left( J,\mu \right) }{\left\vert
J\right\vert ^{\frac{1}{n}}}$. Indeed, to see this, suppose instead that the
matrix $\nabla \mathbf{R}^{\alpha ,n}\mu \left( z\right) $ has at most $%
n-\ell $ eigenvalues of size at least $c\frac{\mathrm{P}^{\alpha }\left(
J,\mu \right) }{\left\vert J\right\vert ^{\frac{1}{n}}}$. Then there is an $%
\ell $-dimensional subspace $S$ such that 
\begin{equation*}
\left\vert \nabla _{S}\mathbf{R}^{\alpha ,n}\mu \left( z\right) \right\vert
=\left\vert \left( \mathcal{P}_{S}\nabla \right) \mathbf{R}^{\alpha ,n}\mu
\left( z\right) \right\vert =\left\vert \mathcal{P}_{S}\left( \nabla \mathbf{%
R}^{\alpha ,n}\mu \left( z\right) \right) \right\vert \leq c\frac{\mathrm{P}%
^{\alpha }\left( J,\mu \right) }{\left\vert J\right\vert ^{\frac{1}{n}}},
\end{equation*}%
which contradicts (\ref{L control}) if $c$ is chosen small enough. This
proves our claim, and moreover, it satisfies the quantitative quadratic
estimate%
\begin{equation*}
\left\vert \xi \cdot \nabla \mathbf{R}^{\alpha ,n}\mu \left( z\right) \xi
\right\vert \geq c\frac{\mathrm{P}^{\alpha }\left( J,\mu \right) }{%
\left\vert J\right\vert ^{\frac{1}{n}}}\left\vert \xi \right\vert ^{2},
\end{equation*}%
for all vectors $\xi $ in some $\left( n-\ell +1\right) $-dimensional
subspace 
\begin{equation*}
\mathsf{S}_{z}^{n-\ell +1}\equiv \limfunc{Span}\left\{ \mathbf{v}%
_{z}^{1},...,\mathbf{v}_{z}^{n-\ell +1}\right\} \in \mathcal{S}_{n-\ell
+1}^{n},
\end{equation*}%
with $\mathbf{v}_{z}^{j}\in \mathbb{S}^{n-1}$ for $1\leq j\leq n-\ell +1$.

It is convenient at this point to let 
\begin{equation*}
k=\ell -1,
\end{equation*}%
so that $1\leq k\leq n-1$ and the assumptions (\ref{either or}) become%
\begin{eqnarray}
&&\text{\textbf{either} }1\leq k\leq n-2\text{ and }n-k<\alpha <n\text{ with 
}\alpha \neq n-1,  \label{either or'} \\
&&\text{\textbf{or} }k=n-1\text{ and }0\leq \alpha <n\text{ with }\alpha
\neq 1,n-1,  \notag
\end{eqnarray}%
and our conclusion becomes%
\begin{equation}
\left\vert \xi \cdot \nabla \mathbf{R}^{\alpha ,n}\mu \left( z\right) \xi
\right\vert \geq c\frac{\mathrm{P}^{\alpha }\left( J,\mu \right) }{%
\left\vert J\right\vert ^{\frac{1}{n}}}\left\vert \xi \right\vert ^{2},\ \ \
\ \ \xi \in \mathsf{S}_{z}^{n-k},\ z\in J.  \label{form est}
\end{equation}

\subsubsection{Proof of strong reversal of energy}

We are now in a position to prove the strong reversal of energy for Riesz
transforms in Lemma \ref{k partial reversal}.

\begin{proof}
(of Lemma \ref{k partial reversal}) Recall that $\mathsf{E}_{k}\left(
J,\omega \right) ^{2}=\inf_{L\in \mathcal{L}_{k}^{n}}\frac{1}{\left\vert
J\right\vert _{\omega }}\int_{J}\left( \frac{\limfunc{dist}\left( x,L\right) 
}{\left\vert J\right\vert ^{\frac{1}{n}}}\right) ^{2}d\omega \left( x\right) 
$ and%
\begin{equation}
\frac{1}{\left\vert J\right\vert _{\omega }}\int_{J}\left( \frac{\limfunc{%
dist}\left( x,L\right) }{\left\vert J\right\vert ^{\frac{1}{n}}}\right)
^{2}d\omega \left( x\right) =\frac{1}{2}\frac{1}{\left\vert J\right\vert
_{\omega }}\int_{J}\frac{1}{\left\vert J\right\vert _{\omega }}%
\int_{J}\left( \frac{\limfunc{dist}\left( x,z+L_{0}\right) }{\left\vert
J\right\vert ^{\frac{1}{n}}}\right) ^{2}d\omega \left( x\right) d\omega
\left( z\right) ,  \label{Ek}
\end{equation}%
where we recall that $L_{0}\in \mathcal{S}_{k}^{n}$ is parallel to $L$. The
real matrix%
\begin{equation}
M\left( x\right) \equiv \nabla \mathbf{R}^{\alpha ,n}\mu \left( x\right) ,\
\ \ \ \ x\in J,  \label{assumption2}
\end{equation}%
is a scalar multiple of the Hessian of $\left\vert x\right\vert ^{\alpha +1}$%
, hence is symmetric, and so we can rotate coordinates to diagonalize the
matrix, 
\begin{equation*}
M\left( x\right) =\left[ 
\begin{array}{cccc}
\lambda _{1}\left( x\right) & 0 & \cdots & 0 \\ 
0 & \lambda _{2}\left( x\right) &  & \vdots \\ 
\vdots &  & \ddots & 0 \\ 
0 & \cdots & 0 & \lambda _{n}\left( x\right)%
\end{array}%
\right] ,
\end{equation*}%
where $\left\vert \lambda _{1}\left( x\right) \right\vert \leq \left\vert
\lambda _{2}\left( x\right) \right\vert \leq ...\leq \left\vert \lambda
_{n}\left( x\right) \right\vert $. We now fix $x=c_{J}$ to be the center of $%
J$ in the matrix $M\left( c_{J}\right) $ and fix the eigenvalues
corresponding to $M\left( c_{J}\right) $: 
\begin{equation*}
\left\vert \lambda _{1}\right\vert \leq \left\vert \lambda _{2}\right\vert
\leq ...\leq \left\vert \lambda _{n}\right\vert ,\ \ \ \ \ \lambda
_{j}\equiv \lambda _{j}\left( c_{J}\right) ,
\end{equation*}%
and define also the subspaces $\mathsf{S}^{n-i}$ to be $\mathsf{S}%
_{c_{J}}^{n-i}$ for $1\leq i\leq k$. Note that we then have $\mathsf{S}%
^{n-i}=\limfunc{Span}\left\{ \mathbf{e}_{i+1},...,\mathbf{e}_{n}\right\} $.
Let $L_{z}^{i}$ be the $i$-plane%
\begin{equation}
L_{z}^{i}\equiv z+\left( \mathsf{S}^{n-i}\right) ^{\perp }=\left\{ \left(
u^{1},...,u^{i},z^{i+1},...,z^{n}\right) :\left( u^{1},...,u^{i}\right) \in 
\mathbb{R}^{i}\right\} .  \label{tag}
\end{equation}

By (\ref{form est}) we have%
\begin{equation*}
\left\vert \lambda _{k+1}\right\vert \geq c\frac{\mathrm{P}^{\alpha }\left(
J,\mu \right) }{\left\vert J\right\vert ^{\frac{1}{n}}}.
\end{equation*}%
For convenience define $\left\vert \lambda _{0}\right\vert \equiv 0$ and
then define $0\leq m\leq k$ be the unique integer such that%
\begin{equation}
\left\vert \lambda _{m}\right\vert <c\frac{\mathrm{P}^{\alpha }\left( J,\mu
\right) }{\left\vert J\right\vert ^{\frac{1}{n}}}\leq \left\vert \lambda
_{m+1}\right\vert .  \label{unique}
\end{equation}%
Now consider the largest $0\leq \ell \leq m$ that satisfies%
\begin{equation}
\left\vert \lambda _{\ell }\right\vert \leq \gamma ^{-\frac{1}{2n}%
}\left\vert \lambda _{\ell +1}\right\vert .  \label{largest}
\end{equation}%
Note that this use of $\ell $ is quite different than that used in (\ref%
{either or}).

So suppose first that $\ell $ satisfies $1\leq \ell \leq m$ and is the
largest index satisfying (\ref{largest}). Then if $\ell <m$ we have $%
\left\vert \lambda _{i}\right\vert >\gamma ^{-\frac{1}{2n}}\left\vert
\lambda _{i+1}\right\vert $ for $\ell +1\leq i\leq m$, and so both%
\begin{eqnarray}
\left\vert \lambda _{\ell +1}\right\vert &>&\gamma ^{-\frac{1}{2n}%
}\left\vert \lambda _{\ell +2}\right\vert >...>\gamma ^{-\frac{m-\ell }{2n}%
}\left\vert \lambda _{m+1}\right\vert \geq \gamma ^{-\frac{m-\ell }{2n}}c%
\frac{\mathrm{P}^{\alpha }\left( J,\mu \right) }{\left\vert J\right\vert ^{%
\frac{1}{n}}},  \label{lamb} \\
\left\vert \lambda _{1}\right\vert &\leq &...\leq \left\vert \lambda _{\ell
}\right\vert \leq \gamma ^{-\frac{1}{2n}}\left\vert \lambda _{\ell
+1}\right\vert .  \notag
\end{eqnarray}%
Both inequalities in the display above also hold for $\ell =m$ by (\ref%
{unique}) and (\ref{largest}). Roughly speaking, in this case where $1\leq
\ell \leq m$, the gradient of $\mathbf{R}^{\alpha ,n}\mu $ has modulus at
least $\left\vert \lambda _{\ell +1}\right\vert $ in the directions of $%
\mathbf{e}_{\ell +1},...,\mathbf{e}_{n}$, while the gradient of $\mathbf{R}%
^{\alpha ,n}\mu $ has modulus at most $\gamma ^{-\frac{1}{2n}}\left\vert
\lambda _{\ell +1}\right\vert $ in the directions of $\mathbf{e}_{1},...,%
\mathbf{e}_{\ell }$.

Recall that $\mathsf{S}^{n-\ell }=\mathsf{S}_{c_{J}}^{n-\ell }$ is the
subspace on which the symmetric matrix $M\left( c_{J}\right) =\nabla \left( 
\mathbf{R}^{\alpha ,n}\mu \right) \left( c_{J}\right) $ has energy $\xi ^{%
\limfunc{tr}}M\left( c_{J}\right) \xi $ bounded below by $\left\vert \lambda
_{\ell +1}\right\vert $. Now we proceed to show that%
\begin{equation}
\left\vert \lambda _{\ell +1}\right\vert ^{2}\left\vert J\right\vert ^{\frac{%
2}{n}}\mathsf{E}\left( J,\omega \right) ^{2}\lesssim \frac{1}{\left\vert
J\right\vert _{\omega }^{2}}\int_{J}\int_{J}\left\vert \mathbf{R}^{\alpha
,n}\mu \left( x\right) -\mathbf{R}^{\alpha ,n}\mu \left( z\right)
\right\vert ^{2}d\omega \left( x\right) d\omega \left( z\right) .
\label{will show}
\end{equation}%
We will use our hypothesis that $\omega $ is $k$-energy dispersed to obtain%
\begin{equation*}
\mathsf{E}\left( J,\omega \right) \leq \mathsf{E}_{k}\left( J,\omega \right)
\leq \mathsf{E}_{m}\left( J,\omega \right) \leq \mathsf{E}_{\ell }\left(
J,\omega \right)
\end{equation*}%
since $\ell \leq m\leq k$. To prove (\ref{will show}), we take $L_{z}\equiv
L_{z}^{\ell }$ as in (\ref{tag}) and begin with%
\begin{eqnarray}
\limfunc{dist}\left( x,L_{z}\right) ^{2} &=&\limfunc{dist}\left( x,z+\left( 
\mathsf{S}^{n-\ell }\right) ^{\perp }\right) ^{2}  \label{dist} \\
&=&\left( x_{\ell +1}-z_{\ell +1}\right) ^{2}+...+\left( x_{n}-z_{n}\right)
^{2}=\left\vert x^{\prime \prime }-z^{\prime \prime }\right\vert ^{2}, 
\notag
\end{eqnarray}%
where $x=\left( x^{\prime },x^{\prime \prime }\right) $ with $x^{\prime }\in 
\mathbb{R}^{\ell }$ and $x^{\prime \prime }\in \mathbb{R}^{n-\ell }$, and $%
L_{z}=\left\{ \left( u^{\prime },z^{\prime \prime }\right) :u^{\prime }\in 
\mathbb{R}^{\ell }\right\} $. Now for $x,z\in J$ we take $\xi \equiv \left(
0,\frac{x^{\prime \prime }-z^{\prime \prime }}{\left\vert x^{\prime \prime
}-z^{\prime \prime }\right\vert }\right) \in \mathsf{S}^{n-\ell }$ (where $%
\frac{0}{0}=0$). We use the estimate%
\begin{equation}
\left\vert J\right\vert ^{\frac{1}{n}}\left\Vert \nabla ^{2}\mathbf{R}%
^{\alpha ,n}\mu \right\Vert _{L^{\infty }\left( J\right) }\lesssim
\left\vert J\right\vert ^{\frac{1}{n}}\int_{\mathbb{R}^{n}\setminus \gamma J}%
\frac{d\mu \left( y\right) }{\left\vert y-c_{J}\right\vert ^{n-\alpha +2}}%
\lesssim \frac{1}{\gamma }\int_{\mathbb{R}^{n}\setminus \gamma J}\frac{d\mu
\left( y\right) }{\left\vert y-c_{J}\right\vert ^{n-\alpha +1}}\approx \frac{%
1}{\gamma }\frac{\mathrm{P}^{\alpha }\left( J,\mu \right) }{\left\vert
J\right\vert ^{\frac{1}{n}}},  \label{error ineq}
\end{equation}%
to obtain%
\begin{eqnarray}
&&\frac{1}{\left\vert J\right\vert _{\omega }^{2}}\int_{J}\int_{J}\left(
\left\Vert \nabla ^{2}\mathbf{R}^{\alpha ,n}\mu \right\Vert _{L^{\infty
}\left( J\right) }\left\vert x-z\right\vert \left\vert J\right\vert ^{\frac{1%
}{n}}\right) ^{2}d\omega \left( x\right) d\omega \left( z\right)
\label{error ineq'} \\
&\lesssim &\frac{1}{\gamma ^{2}}\mathrm{P}^{\alpha }\left( J,\mu \right) ^{2}%
\frac{1}{\left\vert J\right\vert _{\omega }^{2}}\int_{J}\int_{J}\left( \frac{%
\left\vert x-z\right\vert }{\left\vert J\right\vert ^{\frac{1}{n}}}\right)
^{2}d\omega \left( x\right) d\omega \left( z\right) =\frac{1}{\gamma ^{2}}%
\mathrm{P}^{\alpha }\left( J,\mu \right) ^{2}\mathsf{E}\left( J,\omega
\right) ^{2}.  \notag
\end{eqnarray}

We then start with a decomposition into big $B$ and small $S$ pieces,%
\begin{eqnarray*}
&&\frac{1}{\left\vert J\right\vert _{\omega }^{2}}\int_{J}\int_{J}\left\vert 
\mathbf{R}^{\alpha ,n}\mu \left( x\right) -\mathbf{R}^{\alpha ,n}\mu \left(
z\right) \right\vert ^{2}d\omega \left( x\right) d\omega \left( z\right) \\
&\gtrsim &\frac{1}{\left\vert J\right\vert _{\omega }^{2}}%
\int_{J}\int_{J}\left\vert \mathbf{R}^{\alpha ,n}\mu \left( z^{\prime
},x^{\prime \prime }\right) -\mathbf{R}^{\alpha ,n}\mu \left( z^{\prime
},z^{\prime \prime }\right) \right\vert ^{2}d\omega \left( x\right) d\omega
\left( z\right) \\
&&-\frac{1}{\left\vert J\right\vert _{\omega }^{2}}\int_{J}\int_{J}\left%
\vert \mathbf{R}^{\alpha ,n}\mu \left( x^{\prime },x^{\prime \prime }\right)
-\mathbf{R}^{\alpha ,n}\mu \left( z^{\prime },x^{\prime \prime }\right)
\right\vert ^{2}d\omega \left( x\right) d\omega \left( z\right) \\
&\equiv &B-S.
\end{eqnarray*}%
For $w\in J$ we have%
\begin{eqnarray}
\left\vert \nabla \mathbf{R}^{\alpha ,n}\mu \left( w\right) -M\left(
c_{J}\right) \right\vert &=&\left\vert \nabla \mathbf{R}^{\alpha ,n}\mu
\left( w\right) -\nabla \mathbf{R}^{\alpha ,n}\mu \left( c_{J}\right)
\right\vert  \label{w est} \\
&\lesssim &\left\vert w-c_{J}\right\vert \left\Vert \nabla ^{2}\mathbf{R}%
^{\alpha ,n}\mu \right\Vert _{L^{\infty }\left( J\right) }\lesssim \frac{1}{%
\gamma }\frac{\mathrm{P}^{\alpha }\left( J,\mu \right) }{\left\vert
J\right\vert ^{\frac{1}{n}}},  \notag
\end{eqnarray}%
from (\ref{error ineq}), and this inequality will allow us to replace $x$ or 
$z$ by $c_{J}$ at appropriate places in the estimates below, introducing a
harmless error. We now use the second inequality in (\ref{lamb}) with the
diagonal form of $M\left( c_{J}\right) =\nabla \mathbf{R}^{\alpha ,n}\mu
\left( c_{J}\right) $, along with the error estimates (\ref{error ineq'})
and (\ref{w est}), to control $S$ by%
\begin{eqnarray*}
S &\leq &\frac{1}{\left\vert J\right\vert _{\omega }^{2}}\int_{J}\int_{J}%
\left\vert \left( x^{\prime }-z^{\prime }\right) \cdot \nabla ^{\prime }%
\mathbf{R}^{\alpha ,n}\mu \left( x\right) \right\vert ^{2}d\omega \left(
x\right) d\omega \left( z\right) \\
&&+\frac{1}{\left\vert J\right\vert _{\omega }^{2}}\int_{J}\int_{J}\left\{
\left\Vert \nabla ^{2}\mathbf{R}^{\alpha ,n}\mu \right\Vert _{L^{\infty
}\left( J\right) }\left\vert x^{\prime }-z^{\prime }\right\vert ^{2}\right\}
^{2}d\omega \left( x\right) d\omega \left( z\right) \\
&\lesssim &\frac{1}{\left\vert J\right\vert _{\omega }^{2}}%
\int_{J}\int_{J}\left\vert \left( x^{\prime }-z^{\prime }\right) \cdot
\nabla ^{\prime }\mathbf{R}^{\alpha ,n}\mu \left( c_{J}\right) \right\vert
^{2}d\omega \left( x\right) d\omega \left( z\right) \\
&&+\frac{1}{\left\vert J\right\vert _{\omega }^{2}}\int_{J}\int_{J}\left\{
\left\Vert \nabla ^{2}\mathbf{R}^{\alpha ,n}\mu \right\Vert _{L^{\infty
}\left( J\right) }\left\vert x^{\prime }-z^{\prime }\right\vert \left\vert
J\right\vert ^{\frac{1}{n}}\right\} ^{2}d\omega \left( x\right) d\omega
\left( z\right) ,
\end{eqnarray*}%
and then continuing with%
\begin{eqnarray*}
S &\lesssim &\frac{1}{\left\vert J\right\vert _{\omega }^{2}}%
\int_{J}\int_{J}\left\{ \left\vert x^{\prime }-z^{\prime }\right\vert \
\left\vert \lambda _{\ell }\right\vert \right\} ^{2}d\omega \left( x\right)
d\omega \left( z\right) +\frac{1}{\gamma ^{2}}\mathrm{P}^{\alpha }\left(
J,\mu \right) ^{2}\mathsf{E}\left( J,\omega \right) ^{2} \\
&\lesssim &\frac{1}{\gamma }\left\vert \lambda _{\ell +1}\right\vert ^{2}%
\frac{1}{\left\vert J\right\vert _{\omega }^{2}}\int_{J}\int_{J}\left\vert
x-z\right\vert ^{2}d\omega \left( x\right) d\omega \left( z\right) +\frac{1}{%
\gamma ^{2}}\mathrm{P}^{\alpha }\left( J,\mu \right) ^{2}\mathsf{E}\left(
J,\omega \right) ^{2} \\
&=&\frac{1}{\gamma }\left\vert J\right\vert ^{\frac{2}{n}}\left\vert \lambda
_{\ell +1}\right\vert ^{2}\mathsf{E}\left( J,\omega \right) ^{2}+\frac{1}{%
\gamma ^{2}}\mathrm{P}^{\alpha }\left( J,\mu \right) ^{2}\mathsf{E}\left(
J,\omega \right) ^{2},
\end{eqnarray*}%
which is small enough to be absorbed later on in the proof. To bound term $B$
from below we use (\ref{w est}) in%
\begin{eqnarray*}
\mathbf{R}^{\alpha ,n}\mu \left( z^{\prime },x^{\prime \prime }\right) -%
\mathbf{R}^{\alpha ,n}\mu \left( z^{\prime },z^{\prime \prime }\right)
&=&\left( x^{\prime \prime }-z^{\prime \prime }\right) \cdot \nabla ^{\prime
\prime }\mathbf{R}^{\alpha ,n}\mu \left( z\right) +O\left( \left\Vert \nabla
^{2}\mathbf{R}^{\alpha ,n}\mu \right\Vert _{L^{\infty }\left( J\right)
}\left\vert x-z\right\vert ^{2}\right) \\
&=&\left( x^{\prime \prime }-z^{\prime \prime }\right) \cdot \nabla ^{\prime
\prime }\mathbf{R}^{\alpha ,n}\mu \left( c_{J}\right) +O\left( \left\Vert
\nabla ^{2}\mathbf{R}^{\alpha ,n}\mu \right\Vert _{L^{\infty }\left(
J\right) }\left\vert x-z\right\vert \left\vert J\right\vert ^{\frac{1}{n}%
}\right) ,
\end{eqnarray*}%
and then (\ref{form est}) with the choice $\xi \equiv \left( 0,\frac{%
x^{\prime \prime }-z^{\prime \prime }}{\left\vert x^{\prime \prime
}-z^{\prime \prime }\right\vert }\right) \in \mathsf{S}^{n-\ell }$, to obtain%
\begin{eqnarray*}
\left\vert x^{\prime \prime }-z^{\prime \prime }\right\vert \left\vert
\lambda _{\ell +1}\right\vert &\leq &\left\vert x^{\prime \prime }-z^{\prime
\prime }\right\vert \left\vert \left( \xi \cdot \nabla ^{\prime \prime
}\right) \mathbf{R}^{\alpha ,n}\mu \left( c_{J}\right) \cdot \xi \right\vert
\\
&=&\left\vert \left( x^{\prime \prime }-z^{\prime \prime }\right) \cdot
\nabla ^{\prime \prime }\mathbf{R}^{\alpha ,n}\mu \left( c_{J}\right) \cdot
\xi \right\vert \\
&\leq &\left\vert \left( x^{\prime \prime }-z^{\prime \prime }\right) \cdot
\nabla ^{\prime \prime }\mathbf{R}^{\alpha ,n}\mu \left( c_{J}\right)
\right\vert \\
&\leq &\left\vert \mathbf{R}^{\alpha ,n}\mu \left( z^{\prime },x^{\prime
\prime }\right) -\mathbf{R}^{\alpha ,n}\mu \left( z^{\prime },z^{\prime
\prime }\right) \right\vert +O\left( \left\Vert \nabla ^{2}\mathbf{R}%
^{\alpha ,n}\mu \right\Vert _{L^{\infty }\left( J\right) }\left\vert
x-z\right\vert \left\vert J\right\vert ^{\frac{1}{n}}\right) .
\end{eqnarray*}%
Then using (\ref{error ineq'}) and (\ref{w est}) we continue with%
\begin{eqnarray*}
&&\frac{1}{\left\vert J\right\vert _{\omega }^{2}}\int_{J}\int_{J}\left\vert 
\mathbf{R}^{\alpha ,n}\mu \left( z^{\prime },x^{\prime \prime }\right) -%
\mathbf{R}^{\alpha ,n}\mu \left( z^{\prime },z^{\prime \prime }\right)
\right\vert ^{2}d\omega \left( x\right) d\omega \left( z\right) \\
&\gtrsim &\left\vert \lambda _{\ell +1}\right\vert ^{2}\frac{1}{\left\vert
J\right\vert _{\omega }^{2}}\int_{J}\int_{J}\left\vert x^{\prime \prime
}-z^{\prime \prime }\right\vert ^{2}d\omega \left( x\right) d\omega \left(
z\right) -\frac{1}{\gamma ^{2}}\mathrm{P}^{\alpha }\left( J,\mu \right) ^{2}%
\mathsf{E}\left( J,\omega \right) ^{2},
\end{eqnarray*}%
and then%
\begin{eqnarray}
&&\left\vert \lambda _{\ell +1}\right\vert ^{2}\left\vert J\right\vert ^{%
\frac{2}{n}}\mathsf{E}\left( J,\omega \right) ^{2}\leq C\left\vert \lambda
_{\ell +1}\right\vert ^{2}\left\vert J\right\vert ^{\frac{2}{n}}\mathsf{E}%
_{\ell }\left( J,\omega \right) ^{2}  \label{use} \\
&=&\left\vert \lambda _{\ell +1}\right\vert ^{2}\frac{1}{\left\vert
J\right\vert _{\omega }^{2}}\int_{J}\int_{J}\limfunc{dist}\left(
x,L_{z}\right) ^{2}d\omega \left( x\right) d\omega \left( z\right)
=\left\vert \lambda _{\ell +1}\right\vert ^{2}\frac{1}{\left\vert
J\right\vert _{\omega }^{2}}\int_{J}\int_{J}\left\vert x^{\prime \prime
}-z^{\prime \prime }\right\vert ^{2}d\omega \left( x\right) d\omega \left(
z\right)  \notag \\
&\lesssim &\frac{1}{\left\vert J\right\vert _{\omega }^{2}}%
\int_{J}\int_{J}\left\vert \mathbf{R}^{\alpha ,n}\mu \left( z^{\prime
},x^{\prime \prime }\right) -\mathbf{R}^{\alpha ,n}\mu \left( z^{\prime
},z^{\prime \prime }\right) \right\vert ^{2}d\omega \left( x\right) d\omega
\left( z\right) +\frac{1}{\gamma ^{2}}\mathrm{P}^{\alpha }\left( J,\mu
\right) ^{2}\mathsf{E}\left( J,\omega \right) ^{2}  \notag \\
&\lesssim &\frac{1}{\left\vert J\right\vert _{\omega }^{2}}%
\int_{J}\int_{J}\left\vert \mathbf{R}^{\alpha ,n}\mu \left( x\right) -%
\mathbf{R}^{\alpha ,n}\mu \left( z\right) \right\vert ^{2}d\omega \left(
x\right) d\omega \left( z\right) +S+\frac{1}{\gamma ^{2}}\mathrm{P}^{\alpha
}\left( J,\mu \right) ^{2}\mathsf{E}\left( J,\omega \right) ^{2}  \notag \\
&\lesssim &\frac{1}{\left\vert J\right\vert _{\omega }^{2}}%
\int_{J}\int_{J}\left\vert \mathbf{R}^{\alpha ,n}\mu \left( x\right) -%
\mathbf{R}^{\alpha ,n}\mu \left( z\right) \right\vert ^{2}d\omega \left(
x\right) d\omega \left( z\right) +\frac{1}{\gamma }\left\vert \lambda _{\ell
+1}\right\vert ^{2}\left\vert J\right\vert ^{\frac{2}{n}}\mathsf{E}\left(
J,\omega \right) ^{2},  \notag
\end{eqnarray}%
since $\frac{1}{\gamma ^{2}}\mathrm{P}^{\alpha }\left( J,\mu \right) ^{2}%
\mathsf{E}\left( J,\omega \right) ^{2}\leq \frac{1}{\gamma }\left\vert
J\right\vert ^{\frac{2}{n}}\left\vert \lambda _{\ell +1}\right\vert ^{2}%
\mathsf{E}\left( J,\omega \right) ^{2}$ for $\gamma $ large enough depending
only on $n$ and $\alpha $. Finally then, for $\gamma $ large enough
depending only on $n$ and $\alpha $ we can absorb the last term on the right
hand side of (\ref{use}) into the left hand side to obtain (\ref{will show}):%
\begin{equation*}
\left\vert \lambda _{\ell +1}\right\vert ^{2}\left\vert J\right\vert ^{\frac{%
2}{n}}\mathsf{E}\left( J,\omega \right) ^{2}\lesssim \frac{1}{\left\vert
J\right\vert _{\omega }^{2}}\int_{J}\int_{J}\left\vert \mathbf{R}^{\alpha
,n}\mu \left( x\right) -\mathbf{R}^{\alpha ,n}\mu \left( z\right)
\right\vert ^{2}d\omega \left( x\right) d\omega \left( z\right) .
\end{equation*}%
But since $\gamma ^{-\frac{m-\ell }{2n}}c\frac{\mathrm{P}^{\alpha }\left(
J,\mu \right) }{\left\vert J\right\vert ^{\frac{1}{n}}}\leq \left\vert
\lambda _{\ell +1}\right\vert $ by (\ref{lamb}), we have obtained%
\begin{eqnarray*}
\mathrm{P}^{\alpha }\left( J,\mu \right) ^{2}\mathsf{E}\left( J,\omega
\right) ^{2} &\leq &\frac{1}{c^{2}}\gamma \left\vert \lambda _{\ell
+1}\right\vert ^{2}\left\vert J\right\vert ^{\frac{2}{n}}\mathsf{E}\left(
J,\omega \right) ^{2} \\
&\lesssim &\frac{1}{\left\vert J\right\vert _{\omega }^{2}}%
\int_{J}\int_{J}\left\vert \mathbf{R}^{\alpha ,n}\mu \left( x\right) -%
\mathbf{R}^{\alpha ,n}\mu \left( z\right) \right\vert ^{2}d\omega \left(
x\right) d\omega \left( z\right) ,
\end{eqnarray*}%
which is the strong reverse energy inequality for $J$ since%
\begin{equation*}
\frac{1}{2\left\vert J\right\vert _{\omega }^{2}}\int_{J}\int_{J}\left\vert 
\mathbf{R}^{\alpha ,n}\mu \left( x\right) -\mathbf{R}^{\alpha ,n}\mu \left(
z\right) \right\vert ^{2}d\omega \left( x\right) d\omega \left( z\right) =%
\mathbb{E}_{J}^{\omega }\left\vert \mathbf{R}^{\alpha ,n}\mu -\mathbb{E}%
_{J}^{d\omega }\mathbf{R}^{\alpha ,n}\mu \right\vert ^{2}.
\end{equation*}%
This completes the proof of strong reversal of energy under the assumption
that $1\leq \ell \leq m$.

If instead $\ell =0$, then $\left\vert \lambda _{i}\right\vert >\gamma ^{-%
\frac{1}{2n}}\left\vert \lambda _{i+1}\right\vert $ for all $1\leq i\leq m$,
and so the smallest eigenvalue satisfies 
\begin{equation*}
\left\vert \lambda _{1}\right\vert >\gamma ^{-\frac{1}{2n}}\left\vert
\lambda _{2}\right\vert >\gamma ^{-\frac{2}{2n}}\left\vert \lambda
_{3}\right\vert >...>\gamma ^{-\frac{k}{2n}}\left\vert \lambda
_{m+1}\right\vert >\gamma ^{-\frac{1}{2}}c\frac{\mathrm{P}^{\alpha }\left(
J,\mu \right) }{\left\vert J\right\vert ^{\frac{1}{n}}}.
\end{equation*}%
In this case the arguments above show that%
\begin{eqnarray*}
\left( \gamma ^{-\frac{1}{2}}c\frac{\mathrm{P}^{\alpha }\left( J,\mu \right) 
}{\left\vert J\right\vert ^{\frac{1}{n}}}\right) ^{2}\mathsf{E}\left(
J,\omega \right) ^{2} &\lesssim &\frac{1}{\left\vert J\right\vert _{\omega
}^{2}}\int_{J}\int_{J}\left\vert \mathbf{R}^{\alpha ,n}\mu \left( x\right) -%
\mathbf{R}^{\alpha ,n}\mu \left( z\right) \right\vert ^{2}d\omega \left(
x\right) d\omega \left( z\right) \\
&&+\frac{1}{\gamma ^{2}}\mathrm{P}^{\alpha }\left( J,\mu \right) ^{2}\mathsf{%
E}\left( J,\omega \right) ^{2},
\end{eqnarray*}%
which again yields the strong reverse energy inequality for $J$ since the
second term on the right hand side can then be absorbed into the left hand
side for $\gamma $ sufficiently large depending only on $n$ and $\alpha $.
\end{proof}

\subsection{Necessity of the energy conditions}

Now we demonstrate in a standard way the necessity of the energy conditions
for the vector Riesz transform $\mathbf{R}^{\alpha ,n}$ when the measures $%
\sigma $ and $\omega $ are appropriately energy dispersed. Indeed, we can
then establish the inequality 
\begin{equation*}
\mathcal{E}_{\alpha }^{\limfunc{strong}}\lesssim \sqrt{\mathcal{A}%
_{2}^{\alpha }}+\mathfrak{T}_{\mathbf{R}^{\alpha ,n}}.
\end{equation*}%
So assume that (\ref{either or'}) holds. We use Lemma \ref{k partial
reversal} to obtain that the $\alpha $-fractional Riesz transform $\mathbf{R}%
^{\alpha ,n}$ has strong reversal of $\omega $-energy on \emph{all}
quasicubes $J$. Then we use the next lemma to obtain the energy condition $%
\mathcal{E}_{\alpha }^{\limfunc{strong}}\lesssim \mathfrak{T}_{\mathbf{R}%
^{n,\alpha }}+\sqrt{A_{2}^{\alpha }}$.

\begin{lemma}
Let $0\leq \alpha <n$ and suppose that $\mathbf{R}^{\alpha ,n}$ has\ strong
reversal of $\omega $-energy on \emph{all} quasicubes $J$. Then we have the
energy condition inequality,%
\begin{equation*}
\mathcal{E}_{\alpha }^{\limfunc{strong}}\lesssim \mathfrak{T}_{\mathbf{T}%
^{n,\alpha }}+\sqrt{A_{2}^{\alpha ,\limfunc{punct}}}\ .
\end{equation*}
\end{lemma}

\begin{proof}
Fix $\gamma \geq 2$ large enough depending only on $n$ and $\alpha $, and
fix goodness parameters $\mathbf{r}$ and $\varepsilon $ so that $\gamma \leq
2^{\mathbf{r}\left( 1-\varepsilon \right) }$. Then Lemma \ref{k partial
reversal} holds. From the strong reversal of $\omega $-energy with $d\mu
\equiv \mathbf{1}_{I_{r}\setminus \gamma J}d\sigma $, we have%
\begin{eqnarray*}
&&\mathsf{E}\left( J,\omega \right) ^{2}\mathrm{P}^{\alpha }\left( J,\mathbf{%
1}_{I_{r}\setminus \gamma J}d\sigma \right) ^{2} \\
&\leq &C\ \mathbb{E}_{J}^{\omega }\left\vert \mathbf{T}^{\alpha }\left( 
\mathbf{1}_{I_{r}\setminus \gamma J}d\sigma \right) -\mathbb{E}_{J}^{d\omega
}\mathbf{T}^{\alpha }\left( \mathbf{1}_{I_{r}\setminus \gamma J}d\sigma
\right) \right\vert ^{2} \\
&\lesssim &\mathbb{E}_{J}^{\omega }\left\vert \mathbf{T}^{\alpha }\left( 
\mathbf{1}_{I_{r}\setminus \gamma J}d\sigma \right) \right\vert ^{2}\lesssim 
\mathbb{E}_{J}^{\omega }\left\vert \mathbf{T}^{\alpha }\left( \mathbf{1}%
_{I_{r}}d\sigma \right) \right\vert ^{2}+\mathbb{E}_{J}^{\omega }\left\vert 
\mathbf{T}^{\alpha }\left( \mathbf{1}_{\gamma J}d\sigma \right) \right\vert
^{2},
\end{eqnarray*}%
and so%
\begin{eqnarray*}
\sum_{J\in M_{\left( \mathbf{r},\varepsilon \right) -\limfunc{deep}}\left(
I_{r}\right) }\left\vert J\right\vert _{\omega }\mathsf{E}\left( J,\omega
\right) ^{2}\mathrm{P}^{\alpha }\left( J,\mu \right) ^{2} &\lesssim
&\sum_{J}\int_{J}\left\vert \mathbf{T}^{\alpha }\left( \mathbf{1}%
_{I_{r}}d\sigma \right) \left( x\right) \right\vert ^{2}d\omega \left(
x\right) +\sum_{J}\int_{J}\left\vert \mathbf{T}^{\alpha }\left( \mathbf{1}%
_{\gamma J}d\sigma \right) \left( x\right) \right\vert ^{2}d\omega \left(
x\right) \\
&\lesssim &\int_{I_{r}}\left\vert \mathbf{T}^{\alpha }\left( \mathbf{1}%
_{I_{r}}d\sigma \right) \left( x\right) \right\vert ^{2}d\omega \left(
x\right) +\sum_{J}\int_{\gamma J}\left\vert \mathbf{T}^{\alpha }\left( 
\mathbf{1}_{\gamma J}d\sigma \right) \left( x\right) \right\vert ^{2}d\omega
\left( x\right) \\
&\lesssim &\mathfrak{T}_{\mathbf{T}^{n,\alpha }}\left\vert I_{r}\right\vert
_{\sigma }+\sum_{J}\mathfrak{T}_{\mathbf{T}^{n,\alpha }}\left\vert \gamma
J\right\vert _{\sigma }\lesssim \mathfrak{T}_{\mathbf{T}^{n,\alpha
}}\left\vert I_{r}\right\vert _{\sigma }
\end{eqnarray*}%
since $\gamma J\subset I_{r}$ for $\gamma \leq 2^{\mathbf{r}\left(
1-\varepsilon \right) }$, and since the quasicubes $\gamma J$ have bounded
overlap (see \cite[Lemma 2 in v3]{SaShUr6}). We also have%
\begin{equation*}
\sum_{J\in M_{\left( \mathbf{r},\varepsilon \right) -\limfunc{deep}}\left(
I_{r}\right) }\left\vert J\right\vert _{\omega }\mathsf{E}\left( J,\omega
\right) ^{2}\mathrm{P}^{\alpha }\left( J,\mathbf{1}_{\gamma J}d\sigma
\right) ^{2}\lesssim \sum_{J\in M_{\left( \mathbf{r},\varepsilon \right) -%
\limfunc{deep}}\left( I_{r}\right) }A_{2}^{\alpha ,\limfunc{energy}%
}\left\vert \gamma J\right\vert _{\sigma }\lesssim A_{2}^{\alpha ,\limfunc{%
energy}}\left\vert I_{r}\right\vert _{\sigma }
\end{equation*}%
by the bounded overlap of the quasicubes $\gamma J$ in $I_{r}$ once more. We
can now easily complete the proof of $\mathcal{E}_{\alpha }^{\limfunc{strong}%
}\lesssim \mathfrak{T}_{\mathbf{T}^{n,\alpha }}+\sqrt{A_{2}^{\alpha ,%
\limfunc{punct}}}$.
\end{proof}


\begin{thebibliography}{LaSaShUrWi}
\bibitem[CoFe]{CoFe} \textsc{R. R. Coifman and C. L. Fefferman,} \textit{%
Weighted norm inequalities for maximal functions and singular integrals,}
Studia Math. \textbf{51} (1974), 241-250.

\bibitem[DaJo]{DaJo} \textsc{David, Guy, Journ\'{e}, Jean-Lin,} \textit{A
boundedness criterion for generalized Calder\'{o}n-Zygmund operators,} Ann.
of Math. (2) \textbf{120} (1984), 371--397, MR763911 (85k:42041).

\bibitem[HuMuWh]{HuMuWh} \textsc{R. Hunt, B. Muckenhoupt} \textsc{and R. L.
Wheeden,} \textit{Weighted norm inequalities for the conjugate function and
the Hilbert transform}, Trans. Amer. Math. Soc. \textbf{176} (1973), 227-251.

\bibitem[Hyt]{Hyt} \textsc{Hyt\"{o}nen, Tuomas,} \textit{On Petermichl's
dyadic shift and the Hilbert transform}, C. R. Math. Acad. Sci. Paris 
\textbf{346} (2008), MR2464252.

\bibitem[Hyt2]{Hyt2} \textsc{Hyt\"{o}nen, Tuomas, }\textit{The two weight
inequality for the Hilbert transform with general measures, \texttt{%
arXiv:1312.0843v2}.}

\bibitem[HyLaPe]{HyLaPe} \textsc{Hyt\"{o}nen, Tuomas, Lacey, Michael T., and
P\'{e}rez, C., }\textit{Sharp weighted bounds for the }$q$\textit{-variation
of singular integrals}, Bull. Lon. Math. Soc. \textbf{45} (2013), 529-540.

\bibitem[Lac]{Lac} \textsc{Lacey, Michael T.,}\textit{\ Two weight
inequality for the Hilbert transform: A real variable characterization, II},
Duke Math. J. Volume \textbf{163}, Number 15 (2014), 2821-2840.

\bibitem[Lac2]{Lac2} \textsc{Lacey, Michael T.,}\textit{\ The two weight
inequality for the Hilbert transform: a primer}, \texttt{arXiv:1304.5004v1}.

\bibitem[LaSaUr1]{LaSaUr1} \textsc{Lacey, Michael T., Sawyer, Eric T.,
Uriarte-Tuero, Ignacio,} \textit{A characterization of two weight norm
inequalities for maximal singular integrals with one doubling measure,}
Analysis \& PDE, Vol. \textbf{5} (2012), No. 1, 1-60.

\bibitem[LaSaUr2]{LaSaUr2} \textsc{Lacey, Michael T., Sawyer, Eric T.,
Uriarte-Tuero, Ignacio,} \textit{A Two Weight Inequality for the Hilbert
transform assuming an energy hypothesis, } Journal of Functional Analysis,
Volume \textbf{263} (2012), Issue 2, 305-363.

\bibitem[LaSaShUr]{LaSaShUr} \textsc{Lacey, Michael T., Sawyer, Eric T.,
Shen, Chun-Yen, Uriarte-Tuero, Ignacio,} \textit{The Two weight inequality
for Hilbert transform, coronas, and energy conditions, }\texttt{arXiv:}
(2011).

\bibitem[LaSaShUr2]{LaSaShUr2} \textsc{Lacey, Michael T., Sawyer, Eric T.,
Shen, Chun-Yen, Uriarte-Tuero, Ignacio,} \textit{Two Weight Inequality for
the Hilbert Transform: A Real Variable Characterization, }\texttt{%
arXiv:1201.4319} (2012).

\bibitem[LaSaShUr3]{LaSaShUr3} \textsc{Lacey, Michael T., Sawyer, Eric T.,
Shen, Chun-Yen, Uriarte-Tuero, Ignacio,} \textit{Two weight inequality for
the Hilbert transform: A real variable characterization I}, Duke Math. J,
Volume \textbf{163}, Number 15 (2014), 2795-2820.

\bibitem[LaSaShUrWi]{LaSaShUrWi} \textsc{Lacey, Michael T., Sawyer, Eric T.,
Shen, Chun-Yen, Uriarte-Tuero, Ignacio, Wick, Brett D.,} \textit{Two weight
inequalities for the Cauchy transform from }$\mathbb{R}$ to $\mathbb{C}_{+}$%
, \textit{\texttt{arXiv:1310.4820v4}}.

\bibitem[LaWi1]{LaWi1} \textsc{Lacey, Michael T., Wick, Brett D.,} \textit{%
Two weight inequalities for the Cauchy transform from }$\mathbb{R}$ to $%
\mathbb{C}_{+}$, \textit{\texttt{arXiv:1310.4820v1}}.

\bibitem[LaWi]{LaWi} \textsc{Lacey, Michael T., Wick, Brett D.,} \textit{Two
weight inequalities for Riesz transforms: uniformly full dimension weights}, 
\textit{\texttt{arXiv:1312.6163v1,v2,v3}}.

\bibitem[NTV1]{NTV1} \textsc{F. Nazarov, S. Treil and A. Volberg,} \textit{%
The Bellman function and two weight inequalities for Haar multipliers}, J.
Amer. Math. Soc. \textbf{12} (1999), 909-928, MR\{1685781 (2000k:42009)\}.

\bibitem[NTV2]{NTV2} \textsc{Nazarov, F., Treil, S. and Volberg, A.,} 
\textit{The }$Tb$\textit{-theorem on non-homogeneous spaces,} Acta Math. 
\textbf{190} (2003), no. 2, MR 1998349 (2005d:30053).

\bibitem[NTV4]{NTV3} \textsc{F. Nazarov, S. Treil and A. Volberg,} \textit{%
Two weight estimate for the Hilbert transform and corona decomposition for
non-doubling measures}, preprint (2004) \texttt{arxiv:1003.1596}

\bibitem[Saw]{Saw3} \textsc{E. Sawyer,} \textit{A characterization of two
weight norm inequalities for fractional and Poisson integrals}, Trans.
A.M.S. \textbf{308} (1988), 533-545, MR\{930072 (89d:26009)\}.

\bibitem[SaShUr2]{SaShUr2} \textsc{Sawyer, Eric T., Shen, Chun-Yen,
Uriarte-Tuero, Ignacio,} A \textit{two weight theorem for }$\alpha $\textit{%
-fractional singular integrals with an energy side condition}, \texttt{%
arXiv:1302.5093v8.}

\bibitem[SaShUr3]{SaShUr3} \textsc{Sawyer, Eric T., Shen, Chun-Yen,
Uriarte-Tuero, Ignacio,} \textit{A geometric condition, necessity of energy,
and two weight boundedness of fractional Riesz transforms}, \texttt{%
arXiv:1310.4484v1.}

\bibitem[SaShUr4]{SaShUr4} \textsc{Sawyer, Eric T., Shen, Chun-Yen,
Uriarte-Tuero, Ignacio,} \textit{A note on failure of energy reversal for
classical fractional singular integrals}, IMRN, Volume \textbf{2015}, Issue
19, 9888-9920.

\bibitem[SaShUr5]{SaShUr5} \textsc{Sawyer, Eric T., Shen, Chun-Yen,
Uriarte-Tuero, Ignacio,} A \textit{two weight theorem for }$\alpha $\textit{%
-fractional singular integrals with an energy side condition and quasicube
testing}, \texttt{arXiv:1302.5093v10.}

\bibitem[SaShUr6]{SaShUr6} \textsc{Sawyer, Eric T., Shen, Chun-Yen,
Uriarte-Tuero, Ignacio,} A \textit{two weight theorem for }$\alpha $\textit{%
-fractional singular integrals with an energy side condition, quasicube
testing and common point masses}, \texttt{arXiv:1505.07816v2,v3.}

\bibitem[SaShUr7]{SaShUr7} \textsc{Sawyer, Eric T., Shen, Chun-Yen,
Uriarte-Tuero, Ignacio,} A \textit{two weight theorem for }$\alpha $\textit{%
-fractional singular integrals with an energy side condition}, Revista Mat.
Iberoam. \textbf{32} (2016), no. 1, 79-174.

\bibitem[SaShUr8]{SaShUr8} \textsc{Sawyer, Eric T., Shen, Chun-Yen,
Uriarte-Tuero, Ignacio,} The \textit{two weight }$T1$ \textit{theorem for
fractional Riesz transforms when one measure is supported on a curve}, 
\texttt{arXiv:1505.07822v4}.

\bibitem[SaWh]{SaWh} \textsc{E. Sawyer and R. L. Wheeden,} Weighted
inequalities for fractional integrals on Euclidean and homogeneous spaces, 
\textit{Amer. J. Math. }\textbf{114} (1992), 813-874.

\bibitem[Ste]{Ste} \textsc{E. M. Stein,} \textit{Harmonic Analysis:
real-variable methods, orthogonality, and oscillatory integrals},\textit{\ }%
Princeton University Press, Princeton, N. J., 1993.

\bibitem[Vol]{Vol} \textsc{A. Volberg,} \textit{Calder\'{o}n-Zygmund
capacities and operators on nonhomogeneous spaces,} CBMS Regional Conference
Series in Mathematics (2003), MR\{2019058 (2005c:42015)\}.
\end{thebibliography}
\end{document}